\documentclass[twoside]{article}

\usepackage{a4,times,amsmath,theorem,latexsym,amssymb}
\advance\hoffset by -0.3cm

\newcommand{\SDpagebegin}{1}

\makeatletter
\def\@maketitle{%
  \newpage
  {\vspace*{-11ex}%
  \small\noindent}
  \null
  \vskip 5em%
  \begin{flushleft}%
    {\LARGE \bf\@title \par}
    \vskip 1.5em%
       {\large\bf \lineskip .5em \@author \par}%
    \vskip 1.0em%
        {\normalsize\it \lineskip .5em }%
  \end{flushleft}%
  \par
  \vskip 1.5em}
\renewcommand\section{\@startsection {section}{1}{\z@}%
                                   {-3.5ex \@plus -1ex \@minus -.2ex}%
                                   {1.5ex \@plus.2ex}%
                                   {\normalfont\Large\bfseries}}
\renewcommand\subsection{\@startsection{subsection}{2}{\z@}%
                                     {-3.0ex\@plus -2ex \@minus -.2ex}%
                                     {1.0ex \@plus .2ex}%
                                     {\normalfont\large\bfseries}}
\renewcommand\subsubsection{\@startsection{subsubsection}{3}{\z@}%
                                     {-2.5ex\@plus -2ex \@minus -.2ex}%
                                     {0.9ex \@plus .2ex}%
                                     {\normalfont\normalsize\bfseries}}
\renewcommand\paragraph{\@startsection{paragraph}{4}{\z@}%
                                     {-2.0ex \@plus1ex \@minus.2ex}%
                                     {-1em}%
                                     {\normalfont\normalsize\bfseries}}
\renewcommand\subparagraph{\@startsection{subparagraph}{5}{\parindent}%
                                       {-1.5ex \@plus1ex \@minus .2ex}%
                                       {-1em}%
                                      {\normalfont\normalsize\bfseries}}
\def\@evenhead{\footnotesize\thepage\hfil\slshape\leftmark}%
\def\@oddhead{\footnotesize{\slshape\rightmark}\hfil\thepage}%
\numberwithin{equation}{section} \numberwithin{figure}{section}
\numberwithin{table}{section}

\renewcommand{\@makecaption}[2]{\begin{quote}
\footnotesize {\bf #1}~#2
\end{quote}}
\makeatother
\newenvironment{summary}{\vskip\baselineskip \noindent\small\bf Summary: \rm}%
{\vskip\baselineskip}
\newenvironment{proof}{{\vskip\baselineskip\noindent\textbf{Proof:}}}%
{\hspace*{.1pt}\hspace*{\fill}\BOX\vskip\baselineskip}

\newcommand{\BOX}{\ensuremath\Box}
\newtheorem{theorem}{Theorem }[section]

\newtheorem{corollary}[theorem]{Corollary}
\newtheorem{definition}[theorem]{Definition}
{\theorembodyfont{\rmfamily}}
{\theorembodyfont{\rmfamily}}
{\theorembodyfont{\rmfamily}}

\newtheorem{proposition}[theorem]{Proposition}
{\theorembodyfont{\rmfamily}\newtheorem{remark}[theorem]{Remark}}
{\theorembodyfont{\rmfamily}}
{\vskip\baselineskip\noindent\textbf{Proof of {#1}:}}%
{\hspace*{.1pt}\hspace*{\fill}\BOX\vskip\baselineskip}
{\vskip\baselineskip\noindent\textbf{Proof of Theorem \protect\ref{#1}:}}%
{\hspace*{.1pt}\hspace*{\fill}\BOX\vskip\baselineskip}
{\vskip\baselineskip\noindent\textbf{Proof of Theorems \protect\ref{#1} --
\protect\ref{#2}:}}%
{\hspace*{.1pt}\hspace*{\fill}\BOX\vskip\baselineskip}
\renewcommand{\SDpagebegin}{1}        
\usepackage{breakcites}
\usepackage{mathrsfs}
\usepackage{graphicx}
\usepackage{url}
\usepackage[utf8]{inputenc}
\usepackage[round,sort&compress]{natbib}
\usepackage{eurosym}
\usepackage{hyperref} 
\hypersetup{colorlinks=true,linkcolor=black,citecolor=black,urlcolor=black,breaklinks=true}

\newcommand*{\R}{\mathbb{R}}
\newcommand{\bbr}{\R}
\newcommand*{\N}{\mathbb{N}}
\newcommand{\bbn}{\N}
\newcommand*{\E}{\mathbb{E}}
\newcommand*{\var}{\mathrm{var}}
\newcommand*{\cov}{\mathrm{cov}}\newcommand*{\corr}{\mathrm{corr}}

\newtheorem{assume}{Assumption}[section]

\title{Moment based estimation of supOU
processes and a related stochastic volatility model}
\author{Robert Stelzer, Thomas Tosstorff
and Marc Wittlinger}

\begin{document}
\maketitle
\thispagestyle{empty}
\begin{summary}
After a quick review of superpositions of OU (supOU) processes, integrated sup\-OU processes and the supOU stochastic volatility model we estimate these processes by using the generalized method of moments (GMM). We show that the GMM approach yields consistent estimators and that it works very well in practice. Moreover, we discuss the influence of long memory effects.
\end{summary}

\renewcommand{\thefootnote}{}
\footnotetext{\hspace*{-.51cm}AMS 2010 subject classification: Primary: 62M09, 62M10; Secondary: 60G51, 91B25, 91B84\\ %
}

\section{Introduction}

L\'{e}vy-driven Ornstein-Uhlenbeck processes, short OU processes, are a widely studied class of stochastic processes. When used to describe the volatility in a financial model the resulting stochastic volatility model covers many of the stylized facts such as heavy tails, volatility clustering, jumps, etc. (see \cite{Cont.2008}). A L\'{e}vy-driven Ornstein-Uhlenbeck process $Y = (Y_t)_{t\in \mathbb{R}}$ is the solution of the stochastic differential equation
\begin{align}
dY_t = aY_t dt + dL_t\;,
\label{583}
\end{align}
where $a\in\mathbb{R}$ and $L=(L_t)_{t\in\mathbb{R}}$ is a L\'{e}vy process. Under the assumptions $a<0$ and $\mathbb{E}(\log(|L_1|\vee 1))< \infty$ there exists a unique stationary solution of (\ref{583}) which is given by
\begin{align*}
Y_t = \int_{-\infty}^{t}e^{a(t-s)}dL_s\;.
\end{align*}
These L\'{e}vy-driven Ornstein-Uhlenbeck processes are popular mean-reverting jump processes. Since the mean-reversion parameter $a$ is constant, these processes have always the same exponential decay at all times. Likewise the autocorrelation function is simply $e^{a h}$. Typically, however, the autocovariance function of the squared returns of financial prices decays much faster in the beginning than at higher lags. An obvious generalization would be a random mean-reverting parameter, i.e. we substitute the constant $a$ by a random variable $A$ which is different for every jump of the L\'{e}vy process. This feature allows to model more flexible autocovariance functions. The works of \cite{BarndorffNielsen.2000} and \cite{Stelzer.} focus on that generalization and ended up with a superposition of OU processes, called supOU process. Furthermore, it turned out that supOU processes may have the nice feature of exhibiting long memory (long range dependence), i.e. they may have a slowly
polynomially decaying autocovariance function. In \cite{BarndorffNielsen.2011} the authors went a step further and studied a stochastic volatility model in which the volatility process is modeled by a positive supOU process and call it a supOU SV model. Moreover, they showed that long memory in the volatility process yields long memory in the squared log-returns of a supOU SV model. This may be a desirable stylized fact of the log-returns which can only be exhibited by few models. This makes the supOU processes and the supOU SV model particularly interesting for modelling  financial data.
In \cite{StelzerZavisin2014} derivative pricing and the calibration of the model to market option prices are discussed.

However, the modeling of financial data also demands statistical estimation procedures for supOU processes and for the supOU SV model. Unfortunately, the classical and efficient Maximum-Likelihood approach seems not applicable, since the density of supOU processes is not known. Therefore, in this paper we propose  the generalized method of moments which leads to a consistent estimation of supOU processes,  integrated  supOU processes and of the supOU SV model. In a  semiparametric framework we consider in detail examples in which the random mean-reverting parameter $A$ is Gamma distributed and we calculate the moment functions in closed form. Afterwards we show how to estimate the parameters and we discuss the estimation approach in a simulation study. We use a two step iterated GMM estimator i.e. we weight all moments equally in the first step and in the second step we weight the different moments according to the estimation result of the step before. In the illustrations we find out that the GMM estimator works  well and yields (for sufficiently many observations) good and well-balanced estimators. 

This paper is organized as follows. In the second section we give a short review of supOU processes, integrated supOU processes and the supOU SV model. Moreover we give the second order structure of these processes and consider a special case in which we discuss the occurrence of long memory. In Section 3 we introduce the generalized method of moments, give the moment functions and show that the GMM approach yields consistent estimators. In the next section we illustrate how the GMM approach works in practice. In the last section we give a short conclusion.

\section{Review of supOU processes}

In this section we give a short review and some intuition on supOU processes,  integrated  supOU processes and of the supOU stochastic volatility model. For a comprehensive study we refer to \cite{BarndorffNielsen.2000,Stelzer.,BarndorffNielsen.2011}.

\subsection{supOU and  integrated  supOU processes}\label{sec:supOU}

To introduce a random mean-reverting parameter $A$ for the jumps of an OU process we generalize the driving L\'{e}vy process to a so called L\'{e}vy basis which is also known as infinitely divisible independently scattered random measure (abbreviated i.d.i.s.r.m.).

In the following $\mathbb{R}_{-}$ denotes the set of negative real numbers and $\mathcal{B}_b(\mathbb{R}_{-}\times \mathbb{R})$ denotes the bounded Borel sets of $\mathbb{R}_{-}\times \mathbb{R}$.

\begin{definition}
A family $\Lambda = \{\Lambda(B): B\in\mathcal{B}_b(\mathbb{R}_{-}\times \mathbb{R})\}$ of real-valued random variables is called a real-valued L\'{e}vy basis on $\mathbb{R}_{-}\times \mathbb{R}$ if:

\begin{itemize}
\item the distribution of $\Lambda(B)$ is infinitely divisible for all $B\in\mathcal{B}_b(\mathbb{R}_{-}\times \mathbb{R})$,
\item for any $n\in\mathbb{N}$ and pairwise disjoint sets $B_1,...,B_n \in \mathcal{B}_b(\mathbb{R}_{-}\times \mathbb{R})$ the random variables $\Lambda(B_1),...,\Lambda(B_n)$ are independent,
\item for any sequence of pairwise disjoint sets $B_n\in\mathcal{B}_b(\mathbb{R}_{-}\times \mathbb{R})$ with $n\in\mathbb{N}$ satisfying $\cup_{n\in\mathbb{N}} B_n \in \mathcal{B}_b(\mathbb{R}_{-}\times \mathbb{R})$ the series $\sum_{n=1}^{\infty} \Lambda(B_n)$ converges a.s. and $\Lambda(\cup_{n\in\mathbb{N}}B_n) = \sum_{n=1}^{\infty}\Lambda(B_n)$.
\end{itemize}
\end{definition}

As in \cite{Stelzer.} and most other previous works on supOU processes we consider only L\'{e}vy bases whose characteristic functions have the following form
\begin{align*}
 \mathbb{E}(\exp (iu\Lambda(B))) = \exp(\phi(u)\Pi(B))
\end{align*}
for all $u\in\mathbb{R}$ and all $B\in \mathcal{B}_b(\mathbb{R}_{-}\times \mathbb{R})$, where $\Pi = \pi \times \lambda$ is the product of a probability measure $\pi$ on $\mathbb{R}_{-}$ and the Lebesgue measure on $\mathbb{R}$ and
\begin{align*}
\phi(u) = iu\gamma -\frac{1}{2}\Sigma u^2 +\int_{\mathbb{R}}\bigg{(} e^{iux}-1-iux1_{\{|x|\le 1\}} \bigg{)} \nu(dx)
\end{align*}
is the cumulant transform of an infinitely divisible distribution on $\mathbb{R}$ with L\'{e}vy-Khint\-chine triplet $(\gamma,\Sigma,\nu)$. We call the quadruple $(\gamma,\Sigma,\nu,\pi)$ the \emph{generating quadruple}, since it determines completely the distribution of the L\'{e}vy basis. It follows that the L\'{e}vy process $L$ defined by
\begin{align*}
 L_t = \Lambda(\mathbb{R}_{-}\times (0,t]) \text{ and } L_{-t} = \Lambda(\mathbb{R}_{-}\times (-t,0))
\end{align*}
has characteristic triplet $(\gamma, \Sigma , \nu)$ and is called the underlying L\'{e}vy process. Using such a L\'{e}vy basis we finally end up with a superposition of Ornstein-Uhlenbeck processes which is called supOU process. The following definition is analogous to Proposition 2.1  of \cite{Fasen.2007} and \cite{BarndorffNielsen.2000}.



\begin{theorem}[supOU process] \label{supou}
Let $\Lambda$ be a real-valued L\'{e}vy basis on $\mathbb{R}_{-}\times \mathbb{R}$ with generating quadruple $(\gamma,\Sigma,\nu,\pi)$ which satisfies
\begin{align*}
\int_{|x|>1} \log(|x|) \nu(dx) < \infty \text{ and } \int_{\mathbb{R}_{-}}-\frac{1}{A} \pi(dA)<\infty\;.
\end{align*}
Then the process $(X_t)_{t\in \mathbb{R}}$ given by
\begin{align*}
X_t &= \int_{\mathbb{R}_{-}}\int_{-\infty}^{t} e^{A(t-s)}\Lambda(dA,ds)
\end{align*}
is well defined for all $t\in\mathbb{R}$ and stationary. We call the process $X$ a supOU process.\end{theorem}

To apply supOU processes in a practical or a financial framework, we need to estimate the generating quadruple. Maximum Likelihood or a similar approach is not feasible since the density of a supOU process is not known. Hence, we propose here a moment based  estimation for which the second order structure of supOU processes is needed.

\begin{theorem}[\cite{Stelzer.}, Theorem 3.9]
Let $X$ be a stationary real-valued supOU process driven by a L\'{e}vy basis $\Lambda$ satisfying the conditions of Theorem \ref{supou}. If
\begin{align}
 \int_{x>1} x^2 \nu(dx) < \infty
\label{cond-moment}
\end{align}
then $X$ has finite second moments and it holds
\begin{align*}
\mathbb{E}(X_0) &= - \mu \int_{\mathbb{R}_{-}} \frac{1}{A} \pi(dA)\;,\qquad \var(X_0) = - \sigma^2 \int_{\mathbb{R}_{-}} \frac{1}{2A} \pi(dA)\;,\\
\cov(X_h,X_0) &=  - \sigma^2 \int_{\mathbb{R}_{-}} \frac{e^{Ah}}{2A} \pi(dA)\;,
\end{align*}
where $\mu:=\mathbb{E}[L_1] = \gamma + \int_{|x| > 1} x \nu(dx)$, $\sigma^2 :=\var(L_1) = \Sigma + \int_{\mathbb{R}} x^2 \nu(dx)$ and $L$ is the underlying L\'{e}vy process.
\end{theorem}

SupOU processes may exhibit the very interesting stylized fact of a slowly decaying autocorrelation function. More precisely, a stochastic process is said to have long memory effects (or long range dependence) if the autororrelation function $\rho(h)$ satisfies
\begin{align*}
\rho(h) \sim l(h) h^{-H} \qquad \text{for  } h\rightarrow \infty \;,
\end{align*}
where $H\in(0,1)$ and the function $l$ is slowly varying, i.e. $\lim_{t\rightarrow\infty} \tfrac{l(xt)}{l(t)} = 1 \quad \forall x>0$. Of course, this means that $\rho$ is regularly varying at infinity with index $-H$. Long memory effects are discussed in detail in \cite{Cont2010} and \cite{Taqqu.2003}. In Section 2.2 we focus on these long memory effects and present a special case in which supOU processes have such a slowly decaying autocorrelation function.

In some empirical studies, see e.g. \cite{Cont.2001} or \cite{gui}, it is suggested  that the prices of financial assets may have long memory effects. Positive Ornstein-Uhlenbeck type processes are convenient to model the volatility in stochastic volatility models, see e.g. \cite{Shepard}, but they do not yield long memory effects. Therefore \cite{BarndorffNielsen.2011} replaced the positive Ornstein-Uhlenbeck type process by a positive supOU process and called it a supOU SV model. In their work, they show that long memory effects in the volatility process yield long memory effects in the squared log-returns which makes this stochastic volatility model very interesting for modelling financial data.

As before it seems appropriate to use moment estimators for estimating a supOU SV model. Later we will see that the  integrated  supOU process, which we introduce now, can be used to determine the moments of such a supOU SV model. We only consider positive  integrated  supOU processes as they are mainly of interest in connection with stochastic volatility models where they are naturally positive. However, the results remain true in general as an inspection of the proofs in \cite{Stelzer.} shows. To the best of our knowledge integrated supOU processes have not been used in modelling so far. The main purpose of our analyis of integrated supOU processes is thus to better understand the behaviour of our estimators, because they can be seen as an intermediate step between supOU processes and supOU SV models.

\begin{definition}[integrated supOU process] 
Let $X$ be a supOU process with generating quadruple $(\gamma,0, \nu,\pi)$ such that \[\gamma_0:=\gamma - \int_{|x|\le 1} x \nu(dx)\geq0, \quad\int_{|x|\le 1} |x| \nu(dx)< \infty \quad \mbox{and} \quad\nu(\mathbb{R}_{-})=0\]hold. Assume that $(V_n)_{n\in\mathbb{N}}$ is given by
\begin{align*}
V_n := \int_{(n-1)\Delta}^{n \Delta} X_s ds
\end{align*}
where $\Delta$ is a fixed positive number. Then we call the process $V$ an integrated supOU process.
\end{definition}
The assumption  $\nu(\mathbb{R}_{-})=0$ implies that all jumps of the underlying L\'evy process  are positive, $\int_{|x|\le 1} |x| \nu(dx)< \infty$ gives that the paths have finite variation and $\gamma_0:=\gamma - \int_{|x|\le 1} x \nu(dx)\geq 0$ ensures that the drift of the underlying L\'evy process is non-negative. Together the assumptions imply that the underlying L\'evy process is a subordinator and the resulting supOU process is non-negative.

Again we want to estimate such processes via moment estimators and therefore we need their second order structure.

\begin{theorem} [\cite{BarndorffNielsen.2011}, Theorem 3.4] 
Let $V$ be an integrated supOU process such that (\ref{cond-moment}) holds. Then the process $(V_n)_{n\in\mathbb{N}}$ is stationary and square-integrable with
\begin{align*}
\mathbb{E}(V_1) &= - \Delta \mu \int_{\mathbb{R}_{-}} \frac{1}{ A} \pi(dA)\;,\\
\var(V_1) &= - \sigma^2 \int_{\mathbb{R}_{-}} \frac{1}{A^2} \bigg{(}\frac{e^{A \Delta}}{A} -\frac{1}{A}-\Delta \bigg{)} \pi(dA)\;,\\
\cov(V_{h+1},V_1) &=  - \sigma^2 \int_{\mathbb{R}_{-}} \frac{1}{2 A^3}(f_{h+1}-2f_h+f_{h-1}) \pi(dA)\;,
\end{align*}
where $f_h:=e^{A\Delta h}$, $\mu:=\mathbb{E}[L_1] = \gamma_0+\int_{\mathbb{R}} x \,\nu(dx)$, $\sigma^2 :=\var(L_1) = \int_{\mathbb{R}} x^2\, \nu(dx)$ and $L$ is the underlying L\'{e}vy process.
\label{moments}
\end{theorem}

\subsection{A concrete specification with possible long memory effects}\label{sec:longmemgam}

In this section we present the first and second order structure of a supOU process $X$ and an  integrated  supOU process $V$ under the assumption that the stochastic mean-reverting parameter $A$ is Gamma distributed. Furthermore, we investigate in which cases these supOU processes have long memory effects.

Let us consider a semiparametric framework in which we assume that $\pi$ is the distribution of $BR$ where $B\in\mathbb{R}_{-}$ and $R\sim\Gamma(\alpha_\pi,1)$ with $\alpha_\pi>1$ (cf. \cite{Stelzer.}, Example 3.1). Furthermore, we emphasize that setting the second parameter of the Gamma distribution equal to one does not restrict the model since this is equivalent to  varying $B$, cp. \cite{BarndorffNielsen.2011}. From Example 3.1 in \cite{Stelzer.} we get that the supOU process $X$ has finite second moments and applying Theorem \ref{moments} yields
\begin{align*}
 \mathbb{E}(X_0) &=  - \frac{\mu}{B(\alpha_\pi -1)}  \;, \qquad \var(X_0) = - \frac{\sigma^2}{2B(\alpha_\pi -1)}\;,\\
   \cov(X_0,X_h) &= - \frac{\sigma^2(1-Bh)^{1-\alpha_\pi}}{2B(\alpha_\pi -1)}\;,
\end{align*}
where $\mu$ and $\sigma^2$ denote the expectation and the variance of the underlying L\'{e}vy process, respectively. Moreover, the moment structure depends only on the parameter vector $\beta:=(\mu,\sigma^2,\alpha_\pi,B)$ and the autocorrelation function $$\rho(h)= (1-Bh)^{1-\alpha_{\pi}}$$ exhibits  long memory effects for $\alpha_\pi \in (1,2)$ as one can see immediately.

In the case of the corresponding  integrated  supOU process $V$ we get from \cite{BarndorffNielsen.2011}, Theorem 3.4, that

\begin{align*}
 \mathbb{E}(V_1) &=  - \frac{\Delta \mu}{B(\alpha_\pi -1)}  \;, \; \var(V_1) = -\sigma^2 \frac{(1-B\Delta)^{3-\alpha_\pi}-1-\Delta B (\alpha_\pi-3)}{B^3(\alpha_\pi-1)(\alpha_\pi-2)(\alpha_\pi-3)}\;,\\
   \cov(V_1,V_{1+h}) &= - \frac{\sigma^2(f_{h+1}-2f_{h}+f_{h-1})}{2B^3(\alpha_\pi-1)(\alpha_\pi-2)(\alpha_\pi-3)}\;,
\end{align*}
where $f_{h}:=(1-B\Delta h)^{3-\alpha_\pi} $ and $\mu$, $\sigma^2$ denote the expectation and the variance of the underlying L\'{e}vy process, respectively. As in the case of a supOU process, the moments depend only on the parameter vector $\beta:=(\mu,\sigma^2,\alpha_\pi,B)$ and for $\alpha_\pi\in(1,2)$ the process exhibits long memory effects, see \cite{Stelzer.}, Example 3.1.

\begin{remark}\label{rem:annualize}
In applications there are often several natural choices for the time scale possible. For example, for financial data one can quite often either use one trading day or one trading year (c. 250 trading days) as the unit time interval. It is easy to see that our supOU processes are supOU processes regardless of the choice of the unit time interval and that $\alpha_\pi$ in the above concrete specification does not depend on this choice, whereas $\mu,\sigma^2,B$ scale proportionally to the length of the unit time interval.
So for example if $\mu,\sigma^2,B$ are obtained for the unit time interval being one trading day, then the ``annualized'' parameters are  $250\cdot\mu,250\cdot\sigma^2,250\cdot B$.
\end{remark}
\subsection{SupOU SV model}

Stochastic volatility models in which the volatility process is a positive Ornstein-Uhlen\-beck type process capture most of the stylized facts as heavy tails, volatility clustering, jumps, etc. If we model the volatility process by a positive supOU process, we may add the feature of long memory effects as described in Section 2.2.

\begin{definition}
Let $W$ be a standard Brownian motion independent of the L\'{e}vy basis and $\Sigma$ be a supOU process with generating quadruple $(\gamma,0,\nu,\pi)$ such that \[\gamma - \int_{|x|\le 1} x \nu(dx)\geq0,\quad \int_{|x|\le 1} |x| \nu(dx)< \infty\quad \mbox{and}\quad  \nu(\mathbb{R}_{-})=0\] hold. Then we define (the log price process) $(X_t)_{t\ge 0}$ by
\begin{align*}
dX_t =  \sqrt{\Sigma_t} dW_t\;, \qquad X_0 =0\;,
\end{align*}
and say that the process $X$ follows a supOU type SV model. In the following we abbreviate the supOU type SV model by SVsupOU $(\gamma,0,\nu,\pi)$.
\end{definition}

There is no drift in the supOU SV model included. The reason is compared to \cite{Stelzer.} that in the presence of a drift one has no longer an explicit formula for the second order structure available. To obtain meaningful estimates one thus should apply our estimation procedure to demeaned observations. Likewise, we should mention that our specification implies that the distribution of the log returns is symmetric. In general by including a leverage effect as in \cite{BarndorffNielsen.2011} asymmetric distributions can be achieved. However, then the second order moment structure of the squared returns seems not to be obtainable in a reasonably explicit manner and additional parameters appear. In the end this implies that the model with leverage seems not to be estimatable in a simple GMM approach like ours in the following sections. One could resort to use estimation methods based on the characteristic function (cf. \cite{TauferLeonenkoBee2009,Pigorsch.2009} for the OU case). However, such an approach is beyond the scope of the present paper. One important advantage of our upcoming approach is that it is semi-parametric, since we only specify in detail the distribution $\pi$ of the mean reversion parameter whereas the underlying L\'evy process only is required to have finite second moments. Unlike in methods based on the characteristic function, this implies that our estimators are robust to specification errors in the underlying L\'evy process. Of course, if one assumes a model for the underlying L\'evy process that is fully specified by mean and variance, as it would be the case e.g. for a Gamma L\'evy process, our estimation methodology allows to obtain all parameters of the model. In general one should bear in mind that our simple model is not suitable for markedly skewed data. For instance, for exchange rates the log returns are typically rather symmetric.

In financial markets one usually observes the log-returns on a discrete-time basis. This suggests that we focus on the log-returns $(Y_n)_{n\in\mathbb{N}}$ which are given by
\begin{align*}
Y_n := X_{n \Delta} - X_{(n-1)\Delta} = \int_{(n-1)\Delta}^{n\Delta} \sqrt{\Sigma_t}dW_t\;, \qquad \text{for some fixed $\Delta>0$.}
\end{align*}

Using the It\^{o} Isometry as in \cite{Pigorsch.2009}, it turns out that the second order structure of the supOU SV model can be determined by using the second order structure of the  integrated  supOU process. This is the main reason why we considered  integrated  supOU processes before.

\begin{theorem} [\cite{BarndorffNielsen.2011}, Theorem 3.4]
Let $X,\Sigma$ be an \linebreak[4] $SVsupOU(\gamma,0,\nu,\pi)$ model satisfying (\ref{cond-moment}). Then $(Y_n)_{n\in\mathbb{N}}$ as well as $(Y^2_n)_{n\in\mathbb{N}}$ are stationary and square integrable with
\begin{align}
&\mathbb{E}(Y_1) = 0\;, \qquad \var(Y_1) = \mathbb{E}(V_1)\;,\qquad \cov(Y_{h+1},Y_1)=0 \; \forall h>0\;,\\
&\mathbb{E}(Y^2_1) = \mathbb{E}(V_1)\;, \qquad \var(Y^2_1) = 3\var(V_1) + 2\mathbb{E}(V_1)^2\;, \label{984} \\
&\cov(Y^2_{h+1},Y^2_1)=\cov(V_{h+1},V_1) \; \forall h>0\;. \label{481}
\end{align}
\label{svmodel}
\end{theorem}

Due to Equation (\ref{481}) long memory effects carry over from the integrated supOU process to the squared log-retuns. Hence the squared log-returns exhibit long memory effects if $\pi$ is the distribution of $BR$ where $B\in\mathbb{R}_{-}$, $R\sim\Gamma(\alpha_\pi,1)$ and $\alpha_\pi\in(1,2)$.

\subsection{Alternative specifications for the distribution of the mean reversion parameter}

The only concrete specification of $\pi$ discussed so far was a Gamma distribution on the negative half axis (i.e. a Gamma distribution mirrored at the origin). In most of the literature on supOU processes this is the only concrete specification discussed apart from simple discrete distributions on finitely many points. The main motivation of going from OU to supOU processes is to be able to obtain long memory or at least models which do not have exponentially fast decay rates for the autocovariance function.

It is immediate that this desired feature can only be obtained if $\pi((-\epsilon,0))>0$ for every $\epsilon>0$, i.e. it needs to be possible that mean reversion rates arbitrarily close to zero can occur.
The necessary (and sufficient) condition
\begin{align}\label{excond}
 \int_{\mathbb{R}_{-}}-\frac{1}{A} \pi(dA)<\infty
\end{align}
for the supOU process to exist is equivalent to
\begin{align}
 \int_{-1}^0-\frac{1}{A} \pi(dA)<\infty,
\end{align}
as $\pi$ is a probability measure. Clearly, if $\pi$ has a density which is regularly varying with index $\alpha>0$ at zero from the left (a function $f:(-\infty,0)\to \bbr$ is said to be regularly varying at zero from the left with index $\rho\in\bbr$ if $f(\lambda x)/f(x)\to\lambda ^\rho$ as $x\nearrow 0$ for all $\lambda>0$, cf. \cite[p. 18]{bingham:goldie:teugels:1987}), then \eqref{excond} is satisfied.
Likewise, \eqref{excond} is definitely violated whenever $\pi$ has a density which is regularly varying with index $\alpha<0$ at zero. Note that if as above  $\pi$ is the distribution of $BR$ where $B\in\mathbb{R}_{-}$, $R\sim\Gamma(\alpha_\pi,1)$ and $\alpha_\pi\in(1,2)$, then the density of $\pi$ is regularly varying with index $\alpha_\pi-1$.

In general one can easily see that in the case of a  continuous density of $\pi$ on $(-\infty,0)$ we  have $\lim_{x\nearrow 0} d\pi(x)/dx=0$ as a necessary condition for \eqref{excond} to hold.

This  shows immediately that many concrete specifications for $\pi$ lead to valid supOU models, because any probability distribution with a density going faster to zero than $(-x)^\epsilon$ for some $\epsilon>0$ as $x\nearrow 0$ can be employed. Likewise, many discrete distributions can be used. For example, if $\pi$ is a probability distribution concentrated on $(-1/k)_{k\in\bbn}$, it suffices that $\pi(-1/k)$ goes faster to zero than $1/k^\epsilon$ for some $\epsilon>0$ as $k\to\infty$. 

The question now is for which choices of $\pi$ we get indeed long memory or at least a regularly varying (at $\infty$) autocovariance/autocorrelation function. The autocorrelation function
\[
\corr(X_h,X_0) = \frac{ -\int_{\mathbb{R}_{-}} \frac{e^{Ah}}{2A} \pi(dA)}{-\int_{\mathbb{R}_{-}} \frac{1}{2A} \pi(dA)}
\]
only depends on $\pi$ provided it exists. Thus it suffices to understand when the mapping $h\mapsto-\int_{\mathbb{R}_{-}} \frac{e^{Ah}}{2A} \pi(dA)$ is regularly varying at $\infty$.

In \cite{Fasen.2007} actually necessary and sufficient conditions for regular variation of the autocorrelation function at $\infty$ have been obtained for general measures $\pi$ in  terms of the behaviour of an  auxiliary measure at zero. As our focus is mainly on applications where it seems very natural to look only at absolutely continuous $\pi$, we give the following sufficient conditions for regular variation of the autocorrelation function and we decided to present an elementary direct proof instead of employing \cite{Fasen.2007}, Proposition 2.5.
\begin{proposition}\label{th:supoulongmem} Let $\Lambda$ be a real-valued L\'{e}vy basis on $\mathbb{R}_{-}\times \mathbb{R}$ with  quadruple $(\gamma,\Sigma,\nu,\pi)$ which satisfies
\begin{align*}
\int_{|x|>1} x^2 \nu(dx) < \infty
\end{align*}
and where $\pi $ is absolutely continuous with the density $\pi'(x)=(-x)^\alpha l(x)$ being regularly varying at zero from the left with index $\alpha>0$ (i.e. $l$ is slowly varying at zero).
\begin{itemize}
 \item[a)]
Then the supOU process $(X_t)_{t\in \mathbb{R}}$ given by
\begin{align*}
X_t &= \int_{\mathbb{R}_{-}}\int_{-\infty}^{t} e^{A(t-s)}\Lambda(dA,ds)
\end{align*}
is well defined for all $t\in\mathbb{R}$ and (second order) stationary.
\item[b)] If additionally there exists a function $f:(-\infty,0)\to\bbr^+$ such that $\frac{l(x/h)}{l(-1/h)}\leq f(x)$ for all  $x\in\bbr^-,\,h>0$ sufficiently big and $\int_{\bbr^-} e^z (-z)^{\alpha-1} f(z)dz<\infty$, then the autocorrelation function of $X$ is regularly varying with index $-\alpha$ at infinity.

In particular, if $\alpha\in(0,1)$, then the supOU process has long memory.
\end{itemize}
\end{proposition}
\begin{proof}
 The existence and (second order) stationarity is clear from the results of Section  \ref{sec:supOU} and the discussion preceding the theorem.

It remains to show b).

Substituting $z=Ah$, we have for all $h>0$
\[
 -\int_{\mathbb{R}_{-}} \frac{e^{Ah}}{2A} \pi(dA)=\int_{\mathbb{R}_{-}} \frac{e^{Ah}}{2}(-A)^{\alpha -1} l(A)dA=\frac{1}{2h^\alpha}\int_{\bbr^-}e^z(-z)^{\alpha-1}l(z/h)dz.
\]

As $l(-1/x)$ is slowly varying at infinity, dominated convergence implies
\[
\lim_{h\to\infty}\frac{\int_{\bbr^-}e^z(-z)^{\alpha-1}l(z/h)dz}{l(-1/h)}dz=\int_{\bbr^-}e^z(-z)^{\alpha-1}dz=\Gamma(\alpha).
\]
Together this implies that
\[
 \corr(X_h,X_0) \sim \frac{1}{h^\alpha}l(-1/h)\frac{\Gamma(\alpha)}{-2\int_{\mathbb{R}_{-}} \frac{1}{2A} \pi(dA)}\,\, as\,\, h\to\infty,
\]
which concludes.
\end{proof}

\begin{remark}\label{rem:supOUlongmem}
\begin{itemize}
 \item[a)] If $l$ is continuous and $\lim_{x\nearrow 0} l(x)>0$ exists, then the conditions of b) are satisfied (at least for all $h$ big enough), as $\lim_{x\to-\infty}l(x)<\infty$ follows from the fact that $\pi'$ is a probability density.
Actually, one even has $\corr(X_h,X_0) \sim \frac{C}{h^\alpha}$ with a constant $C>0$.
\item[b)] The above theorem also applies to cases where the slowly varying function $l$ goes to infinity as $x$ goes to $0$.\\
For instance, consider $\pi'(x)=C\ln(-1/x)(-x)^\alpha 1_{(-1,0)}(x)$ with $\alpha>0$ and $C>0$ such that we have a probability density. Then one can easily see that $\frac{l(x/h)}{l(-1/h)}=1-\frac{\ln(-x)}{\ln(h)}\leq 1-\ln(-x)$ for $h$ big enough and that  \[\int_{\bbr^-} e^z (-z)^{\alpha-1} (1-\ln(-z))dz=\int_{\bbr^-} e^z (-z)^{\frac\alpha2-1} (-z)^{\frac\alpha2}(1-\ln(-z))dz<\infty,\] as $\lim_{z\nearrow 0}(-z)^{\frac\alpha2}\ln(-z)=0$ and $\int_{-1}^0 e^z (-z)^{\frac\alpha2-1}dz<\infty$. Hence, for this choice of $\pi$ we have
\[
 \corr(X_h,X_0) \sim \frac{1}{h^\alpha}\ln(h)\tilde C\,\, as\,\, h\to\infty, 
\]
with a constant $\tilde C>0$.
\end{itemize}
\end{remark}

This result and the general Proposition 2.5 of \cite{Fasen.2007}  show that it is the regular variation of $\pi$ at zero that causes the power decay of the autocorrelation function at infinity in the Gamma example of Section \ref{sec:longmemgam}. This implies that many choices of $\pi$ can give long memory. On the other hand our previous Gamma example is somewhat representative as varying the parameter $\alpha_\pi$ gives all possible asymptotic decay rates.

When thinking about which other popular continuous probability distributions  one could use, one may be tempted to think about the Generalized Inverse Gaussian (GIG) family, as it includes the Gamma distribution. However, it is easy to see that in that case
we have the regular variation at $0$ if and only if the GIG distribution is a Gamma distribution. On the positive side, for example, the Beta distribution on $[-1,0]$ (i.e. $\pi'(x)=\frac{1}{B(p,q)}(-x)^{p-1}(1+x)^{q-1}1_{[-1,0]}(x)$ with $p,q>0$) gives another concrete specification satisfying the regular variation condition of Proposition \ref{th:supoulongmem} (and Remark \ref{rem:supOUlongmem} a)) provided $p>1$.

In the following we carry out a moment based estimation of the supOU (SV) model assuming the Gamma choice for $\pi$. The methodology can in principle be applied to many other choices of $\pi$. The crucial step is to show that the parameters are identifiable from the chosen moments. For our methodology it is not really necessary that one can compute the moments of the model in a form as nice as the one below, since in the worst case one can compute them via numerical integration.

\section{Moment based estimation under a Gamma distributed mean reversion parameter}

In this section we study a moment based estimation approach of supOU processes,  integrated  supOU processes and of the supOU SV model. For comprehensive introductions to the generalized method of moments we refer to \cite{Hansen,Hall.2005} or \cite{Matyas.1999}.
\begin{assume}
Let $\pi$ be the distribution of $BR$ where $B\in\mathbb{R}_{-}$ and $R\sim\Gamma(\alpha_\pi,1)$.
\label{221}
\end{assume}

We recall that Assumption \ref{221} implies that we are in a semiparametric setting and that we estimate the parameter vector $\beta = (\mu,\sigma^2, \alpha_\pi, B)$.

Let $X = (X_t)_{t\in\mathbb{R}}$ be  the underlying process (either the supOU process, the integrated one or the discrete-time returns in the supOU SV model) and $\mathbf{X}=(X_1,X_{2},...,X_{N})$ be a vector of $N\in\mathbb{N}$ equidistant observations made from it. We introduce the vector
\begin{align*}
 X^{(m)}_t:=(X_t,...,X_{t+m})\qquad \text{ for } t\in\{1,...,N-m\},
\end{align*}
since the estimation procedure will include autocovariances up to a lag $m\ge 2$. In the first step, we have to find a measurable function $f:\mathbb{R}^{m+1}\times W\rightarrow \mathbb{R}^{d}$, called moment function, such that
\begin{align*}
 \mathbb{E} [f(X^{(m)}_t,\beta)] = 0  \quad \Leftrightarrow  \quad  \beta = \beta_0\;,
\end{align*}
where $W\subset \mathbb{R}^4$ denotes a compact parameter space which includes the true parameter vector $\beta_0$. In the second step we estimate $\beta_0$ by minimizing the objective
\begin{align}
 \beta \rightarrow g_{N,m}(\mathbf{X},\beta)' I  g_{N,m}(\mathbf{X},\beta)
\label{mini}
\end{align}
where $g_{N,m}(\mathbf{X},\beta) = \frac{1}{N-m}\sum_{i=1}^{N-m} f(X^{(m)}_{i},\beta)$ and $I$ is a positive definite matrix to weight the $d$ different moments collected in  $g_{N,m}$. It is well-known that there exists an optimal choice of the weighting matrix $I$, but determining that matrix in the forefront of the estimation is in practice mostly impossible. Because of that we use a two-step iterated GMM estimator which is easy to implement and improves the estimates. For more details on that topic we refer to \cite{Hall.2005}, Sections 3.5 and 3.6.

\begin{theorem} Let $X$ be either a supOU process, an {integrated} supOU process or a supOU SV model. Moreover, let $f:\mathbb{R}^{m+1}\times W\rightarrow \mathbb{R}^{d}$ be a measurable function such that $\mathbb{E} [f(X^{(m)}_t,\beta)]$ identifies the true parameter vector $\beta_0$ , i.e. such that $\mathbb{E} [f(X^{(m)}_t,\beta)] = 0$ if and only if $\beta=\beta_0$. Then the above described GMM estimator is consistent.
\label{Th1}
\end{theorem}

\begin{proof}
From \cite{Fuchs2011} we know that supOU processes,  integrated  supOU processes and the supOU SV model are ergodic. Hence we have ergodicity of the mean and the result follows by \cite{Matyas.1999}, Theorem 1.1.
\end{proof}

\begin{proposition}[Moment function for supOU processes]\label{th:identsupou}
Let $X$ be a supOU process as introduced in Section 2.1, $m\ge 2$ be a fixed integer and
$$
f_X(X^{(m)}_t,\beta)=\left(\begin{array}{c}
f_{\mathbb{E}}(X^{(m)}_t,\beta) \\ f_{\var}(X^{(m)}_t,\beta) \\ f_1(X^{(m)}_t,\beta)\\ \vdots \\ f_m(X^{(m)}_t,\beta)
\end{array} \right)
$$
where
\begin{align*}
f_\mathbb{E}(X^{(m)}_t,{\beta}) &= X_t + \frac{\mu}{B(\alpha_\pi - 1)} \\
f_{\var}(X^{(m)}_t,{\beta}) &= X_t^2 -\left( \frac{\mu}{B(\alpha_\pi - 1)}\right)^2 + \frac{\sigma^2}{2B(\alpha_\pi - 1)} \\
f_h(X^{(m)}_t,{\beta})&= X_tX_{t+h}-\left( \frac{\mu}{B(\alpha_\pi-1)}\right)^2+ \frac{\sigma^2(1-Bh)^{1-\alpha_\pi}}{2B(\alpha_\pi - 1)}\;.
\end{align*}
Then the parameter vector $\beta_0$ is identifiable.
\end{proposition}

\begin{proof} Taking expectations of $f_X(X^{(m)}_t,\beta)$ it is elementary to see that the identifiability is equivalent to showing unique identifiability from the stationary expectation, variance and the  stationary autocovariance function.

Hence, to prove the identifiability of the parameter vector $\beta_0$ it is enough to consider four equations, the equation with the expectation, with the variance and with the autocorrelations $\rho(h_1),\,\rho(h_2)$ with lags $h_1,h_2$ where $h_1\neq h_2 $ and $h_1,h_2 > 0$. From the autocovariances/-cor\-rela\-tions it follows that
$
 \tfrac{\log \rho(h_1)}{\log \rho(h_2)} = \tfrac{\log (1-Bh_1)}{\log(1-Bh_2)}$.
Defining $c:=\tfrac{\log \rho(h_1)}{\log \rho(h_2)}$ gives us $(1-Bh_2)^c + Bh_1-1 =0$. The left hand side of the last equation is a function in $B$ which has a positive second derivative. Hence, it is a strictly convex function which has at most two zeros. Because one zero is at zero, the parameter $B$ is the unique strictly negative zero. With that uniquely determined $B$ we are able to determine $\alpha_\pi$ uniquely by
\begin{align*}
 \alpha_\pi = 1 - \frac{\log\rho(h_1)}{\log(1-Bh_1)} \;.
\end{align*}
The expectation and the variance equations yield unique $\mu$ and $\sigma^2$ which completes the identifiability of $\beta_0$. Now Theorem \ref{Th1} yields the result.
\end{proof}

In the case of an  integrated  supOU process we have not been able to show the identifiability based on a finite number of moments. Instead we can show an \emph{asymptotical} identifiability.

\begin{definition}
A parameter vector $\beta$ of a stochastic process $Y$ is said to be asymptotically identifiable if the mapping $\beta:W\rightarrow \mathbb{R}^{\mathbb{N}}$ with $\beta \rightarrow \mathbb{E}\big{(}f_k(Y,\beta)\big{)}_{k\in\mathbb{N}}$ has a unique zero at the true parameter $\beta_0$ where $f_k$ is the k-th component of the moment function $f:\mathbb{R}^{\mathbb{N}}\times W \rightarrow \mathbb{R}^{\mathbb{N}}$.
\end{definition}

Now we show that the parameters of our model are asymptotically identifiable from the expectation, the variance and (all lags of the) autocovariance function of either the integrated supOU process or the log returns of a supOU SV model.
\begin{proposition} [Moment function for the {integrated} supOU processes]\label{th:intsupouident}
Let $X$ be  a supOU process as introduced in Section 2.1 and $V = (V_n)_{n\in\N}$ the corresponding {integrated} supOU process and
$$
f_V(V,\beta)=\left(\begin{array}{c}
f_{\mathbb{E}}(V,\beta) \\ f_{\var}(V,\beta) \\ f_1(V,\beta)\\ f_2(V,\beta)\\ \vdots 
\end{array} \right)
$$
where
\begin{align*}
f_\mathbb{E}(V,{\beta}) &= V_1 + \frac{\Delta \mu}{B(\alpha_\pi -1)}  \\ f_{\var}(V,{\beta}) &= V_1^2-\bigg{(}   \frac{\Delta \mu}{B(\alpha_\pi -1)} \bigg{)}^2 +\sigma^2 \frac{(1-B\Delta)^{3-\alpha_\pi}-1-\Delta B (\alpha_\pi-3)}{B^3(\alpha_\pi-1)(\alpha_\pi-2)(\alpha_\pi-3)} \\
f_h(V,{\beta})&= V_1V_{1+h}-\bigg{(}  \frac{\Delta \mu}{B(\alpha_\pi -1)}\bigg{ )}^2 + \frac{\sigma^2(f_{h+1}-2f_{h}+f_{h-1})}{2B^3(\alpha_\pi-1)(\alpha_\pi-2)(\alpha_\pi-3)}\;.
\end{align*}
Then the parameter vector $\beta_0$ is asymptotically identifiable.
\end{proposition}
\begin{proof}
 Taking expectations of $f_V(V,\beta)$ it is again elementary to see that the identifiability is equivalent to showing unique identifiability from the stationary expectation, variance and the  stationary autocovariance function.

Let $\alpha_\pi \neq 2,3$. From \cite{BarndorffNielsen.2011}, Example 3.1 and Prop. 3.5 (i), we get

\begin{align*}
\lim_{h\rightarrow\infty}\frac{\cov(V_{1+h},V_1)}{\cov(V_{1+2h},V_1)} = \lim_{h\rightarrow\infty}\frac{-\frac{\Delta^{3-\alpha_\pi}}{2B(\alpha_\pi - 1)} \int\limits_0^\infty{x^2 \nu(dx)}\ (-Bh)^{1-\alpha_\pi}}{-\frac{\Delta^{3-\alpha_\pi}}{2B(\alpha_\pi - 1)} \int\limits_0^\infty{x^2\nu(dx)}\ (-2Bh)^{1-\alpha_\pi}}= \left(\frac{1}{2}\right)^{1-\alpha_\pi}\;.
\end{align*}
This yields a unique $\alpha_\pi$. Using again \cite{BarndorffNielsen.2011}, Example 3.1 and Prop. 3.5 (i), we also obtain
\begin{align*}
\rho(h) \sim \tilde{\rho}_{B,\alpha}(h) \mathrel{\mathop:}=\frac{\Delta^{3-\alpha} (\alpha-2)(\alpha-3) h^{1-\alpha} (-B)^{3-\alpha}}{2 ((1-B\Delta)^{3-\alpha} -1 -\Delta B(\alpha-3))}\;.
\end{align*}
The derivative of $\tilde{\rho}_{B,\alpha}(h)$ with respect to $B$ is a monotone function in $B$. Hence there can only be one $B<0$ such that $\E [f_V(Y_t,\beta)] = 0$. The uniqueness of $\mu$ and $\sigma^2$ follow from $f_\mathbb{E}(V,{\beta})=0$ and $f_{\var}(V,\beta)=0$, respectively. 
The remaining cases $\alpha_\pi =2$ and $\alpha_\pi =3$ can be treated similarly.
\end{proof}

Due to (\ref{984}) and (\ref{481}) we are able to deduce the moment function for the supOU SV model easily.

\begin{corollary} [Moment function for the supOU stochastic volatility model]
Let $X,$\linebreak $\Sigma$ be  a supOU SV model as introduced in Section 2.3, $Y=(Y_n)_{n\in\mathbb{N}}$ equidistant log returns observed on a grid with size $\Delta>0$ and
$$
f_{SV}(Y,\beta)=\left(\begin{array}{c}
f_{\var}(Y,\beta) \\ f_{\operatorname{var2}}(Y,\beta) \\ f_1(Y,\beta)\\f_2(Y,\beta)\\  \vdots 
\end{array} \right)
$$
where
\begin{align*}
f_{\var}(Y,{\beta}) &= Y_1^2+\frac{\Delta \mu}{B(\alpha_\pi -1)}\\  f_{\operatorname{var2}}(Y,{\beta}) &= Y_1^4+ 3 \sigma^2 \frac{(1-B\Delta)^{3-\alpha_\pi}-1-\Delta B (\alpha_\pi-3)}{B^3(\alpha_\pi-1)(\alpha_\pi-2)(\alpha_\pi-3)}              -3\left(\frac{\Delta \mu}{B(\alpha_\pi -1)}\right)^2, \\
f_h(Y,{\beta}) &= Y^2_1Y^2_{1+h}-\bigg{(}  \frac{\Delta \mu}{B(\alpha_\pi -1)}\bigg{ )}^2 + \frac{\sigma^2(f_{h+1}-2f_{h}+f_{h-1})}{2B^3(\alpha_\pi-1)(\alpha_\pi-2)(\alpha_\pi-3)}.
\end{align*}
Then the parameter vector $\beta_0$ is asymptotically identifiable.
\end{corollary}

\begin{proof} Analogous to the proof of Proposition \ref{th:intsupouident}.
\end{proof}

From Propositions 3.5 and 3.6 one conjectures that for reasonably big  $m$ the model is identifiable and that thus the GMM estimators are consistent (cf. Theorem \ref{Th1}). Unfortunately, it seems very hard to prove this. Thus for large values of $m$ in practice the procedure should give consistent estimators. Actually, it may well be that (non-asymptotic) identifiability comparable to Proposition \ref{th:identsupou} is true, but given the highly involved form of the autocovariance function proving this appears out of reach. In the simulated examples in Section 4 we see that a very moderate high choice of  $m$, which corresponds to the highest used order of the autocovariance function, gives good estimation results. The choice of $m$ is  discussed in Section \ref{sec:m}.

\section{Illustrative examples}

\subsection{Set-up and methodology of the simulation study}\label{sec:m}

In this section we illustrate that the moment estimators are working well concentrating on the supOU process itself and the 
stochastic volatility model. In applications the stochastic volatility model seems the most relevant and the presence of the additional Brownian noise should imply that the estimation is more difficult than for the supOU process itself. 

We assume the semiparametric framework of Section \ref{sec:longmemgam} and that the underlying L\'{e}vy  process $L$ of the supOU process $X$ is a compound Poisson process. 
Actually,  we take a compound Poisson process with rate $0.1$ and $\Gamma(3,20)$-distributed positive jumps for the simulations. The choice of the parameters is motivated by looking at simulated paths and inspecting whether they have a reasonable shape for daily log returns of financial data and its volatility process.

Applying the L\'{e}vy-It\^{o} decomposition to the supOU process $X$ (see e.g. Theorem 2.2 of \cite{BarndorffNielsen.2011}) yields for a general underlying compound Poisson process
\begin{align*}
X_t =\int_{\mathbb{R}}\int_{\mathbb{R}_{-}}\int_{-\infty}^{t} e^{A(t-s)}x\mu(dx,dA,ds)
\end{align*}
where $\mu$ is a Poisson random measure. From that representation it follows that $X$ can be written as
\begin{align}
X_{t} = \sum_{i\geq1,\ \tau_{i}\leq t\ }{e^{A_{i}(t - \tau_{i})}U_{i}}
				+ \sum_{i=1\ }^{\infty}{e^{A_{-i}(t + \tau_{-i})}U_{-i}}\;, \label{eq:sumrep}
\end{align}
where
\begin{align*}
	\tau_{i} \mathrel{\mathop:}= \sum_{j = 1}^{i}{T_{j}} \textrm{ and }
	\tau_{-i} \mathrel{\mathop:}= \sum_{j = 1}^{i}{T_{-j}}\qquad \forall\ i \in \mathbb{N}\;,
\end{align*}
and $(T_{i})_{i \in \mathbb{Z}\backslash \{0\}}$, $(U_{i})_{i \in \mathbb{Z}\backslash \{0\}}$ and $(A_{i})_{i \in \mathbb{Z}\backslash \{0\}}$ are independent sequences of iid distributed random variables with $T_i\sim \exp(\nu(\mathbb{R}))$, $U_{i}\sim \frac{1}{\nu(\mathbb{R})}\nu$ and $A_{i}\sim \pi$. 
In the concrete specification used in our simulated examples we have $T_i\sim \exp(0.1)$, $U_{i}\sim \Gamma(3,20)$ and $A_{i}\sim -B\Gamma(\alpha_\pi,1)$.

Now we are able to simulate the introduced stochastic processes easily and to illustrate our moment estimators. In the following we simulate the processes 1000 times (independently) with 10000 observations (on a unit grid) in each run. 
 Clearly the infinite sum in \eqref{eq:sumrep} can only be obtained approximately. We decided to ignore all jumps before time $-2000$, i.e. in the end we simulate the L\'evy basis exactly on $\bbr^-\times [-2000,10000]$ and set it to zero outside this area. This allows us to simulate the volatility process, i.e. the supOU process, (up to the truncation error of the infinite sum) exactly on $[0,10000]$. These exact values were then used in a standard Euler scheme with grid size $1/20$ to simulate the log returns in the stochastic volatility model over unit intervals.

Afterwards we calculate the GMM estimators for each of the 1000 observed paths, i.e. we solve the optimization problem (\ref{mini}) separately for each path. We estimate the parameters both based on the values of the supOU process itself as well as of the log-returns -- both observed on a unit grid. To weight the moments appropriately we use the 2-step iterated GMM Estimation as described in \cite{Hall.2005}, Section 3.6. This means that the weighting matrix $I$ equals the identity matrix in the first step and in the second step the weighting matrix $I$ is an approximation of $S^{-1}$ where
\begin{align}
S = \lim_{n\rightarrow \infty} \var\bigg{(} \frac{1}{\sqrt{n}}\sum_{t=1}^{n} f(X^{(m)}_t, \beta_1)\bigg{)}
\label{331}
 \end{align}
and $\beta_1$ is the estimation result of the first step. Actually, we take the very simple estimator
\begin{align*}
\hat{S} := \frac{1}{n} \sum_{t=1}^{n}f(X^{(m)}_t,\beta_1) f(X^{(m)}_t,\beta_1)'  
\end{align*}
which performs quite well in our studies, but could in principle be improved by using estimators taking autocorrelation effects into account. 

In the estimation of the supOU SV model based on the log returns we used the mean, the second moment and the first five lags of the ``autocovariance'' function (actually of $h\mapsto E(Y_t^2Y_{t+h}^2)$) of the squared logreturns obtained in the simulation. The choice of ``$m$'' is an intricate issue. We want to estimate four parameters. In order to have a proper overidentified system for GMM, we thus need more than four moment conditions. Actually, we have $2+m$ conditions which requires $m\geq3$. So five is very close to the minimum and it is clear that $m$ should be somewhat bigger than 3 to have a ``more overidentified'' system. On the other hand the dimensionality and thus the computational efforts increase with $m$. Moreover, it is folklore for GMM estimators that the finite sample properties get  bad when using too many moment conditions. In our simulation studies it turned out that once in a while $\hat S$ is singular up to numerical precision for $m$ larger than $5$. So we decided to fix $m=5$ and refrain from studying in detail the estimators for other $m$.
Likewise, we used the mean, the variance and the first five lags of the acf when using the observations of the supOU process itself.

In the illustrations below we see that this 2-step iterated GMM estimation works well. Throughout the illustrations we concentrate on two different cases. In the first case we assume a parameter vector $\beta_0=(\mu_0,\sigma^2_0,\alpha_{\pi,0},B_0)=(0.015,0.003,4,\allowbreak-0.1)$ to cover the case of short memory effects and in the second one we assume a parameter vector $\beta_0=(\mu_0,\sigma^2_0,\alpha_{\pi,0},B_0)=(0.015,0.003,1.95,-0.1)$ to cover the case of long memory effects. 

All simulations and estimations have been carried out using R (\cite{Rsoftware}). For the estimations we used the routine \texttt{optim} with the BFGS algorithm after a straightforward variable transformation to obtain an unconstrained optimization problem. The initial values for the first step estimation were taken randomly from a neighbourhood of the true parameters. In the second estimation step the starting values of the optimisation were taken as the estimates of the first step, unless the optimisation algorithm in the first step ended without proper convergence or the obtained estimates were rather far off. In these  cases the initial values for the second step were again taken randomly from a neighbourhood of the true parameters in order to avoid ending up in a far off local minimum. In our simulations  we encountered that non-convergence of the optimisation routine  in the first step is  not uncommon. In the second step non-convergence of the optimisation routine happened almost never, i.e. in at most 6 (out of 1000) cases when using all 10000 observations and in at most 16 (out of 1000) cases when using only 1000 observations. This shows that although the parameter estimators from the first step could be bad, the resulting estimate for the weighting matrix is still  good enough to give a much better behaved optimisation problem in the second step and that using a good weighting matrix is very important. In the upcoming plots of the results of our simulation study the paths where the two step GMM estimator did not converge in the second optimisation step were excluded.

\subsection{Results of the simulation study}
\subsubsection{SupOU processes}

In Figure \ref{figsupOU10K} we see the estimation results for a supOU process using mean,  variance and lags $1,2,4,5$  of the acf using 1000 independent paths with 10000 observations each. The upper set of plots shows the short memory case and the lower set of plots shows the long memory case where the thick line indicates the true parameter values in all histograms. Likewise, Figure \ref{figsupOU10Kqq} shows normal QQ-plots of the obtained parameter estimators. Obviously the estimation procedure works very well and the estimates are in most cases rather evenly distributed around the true parameter values, although one can also spot some mild skewness in several plots. In the histograms for $\alpha_\pi$ in the short memory case and for $\mu$ in the long memory case we see a mild bias. Actually, when ``zooming into the histograms'' one can see some small bias in most parameter estimates. It is noteworthy that the estimator for $\alpha_\pi$ tends to be too low in the short memory case, but in most cases it stays above the ``long memory threshold'' $2$.

The QQ-Plots for the estimators of $\mu$ and $\sigma^2$ are very remarkable. They indicate clearly an asymptotic normality of the estimators -- both in the short and (maybe surprisingly) long memory case. The tails of the estimators for $\alpha_\pi$ and $B$ are systemically deviating from the ones of the normal distribution. But again one does not see a real difference between the short and long memory case. Moreover, the deviation is rather small so that asymptotically a normal distribution (or another one not really far from it) may still be valid.

\begin{figure}[p]
\center
\includegraphics[height=0.49\textwidth,angle=270]{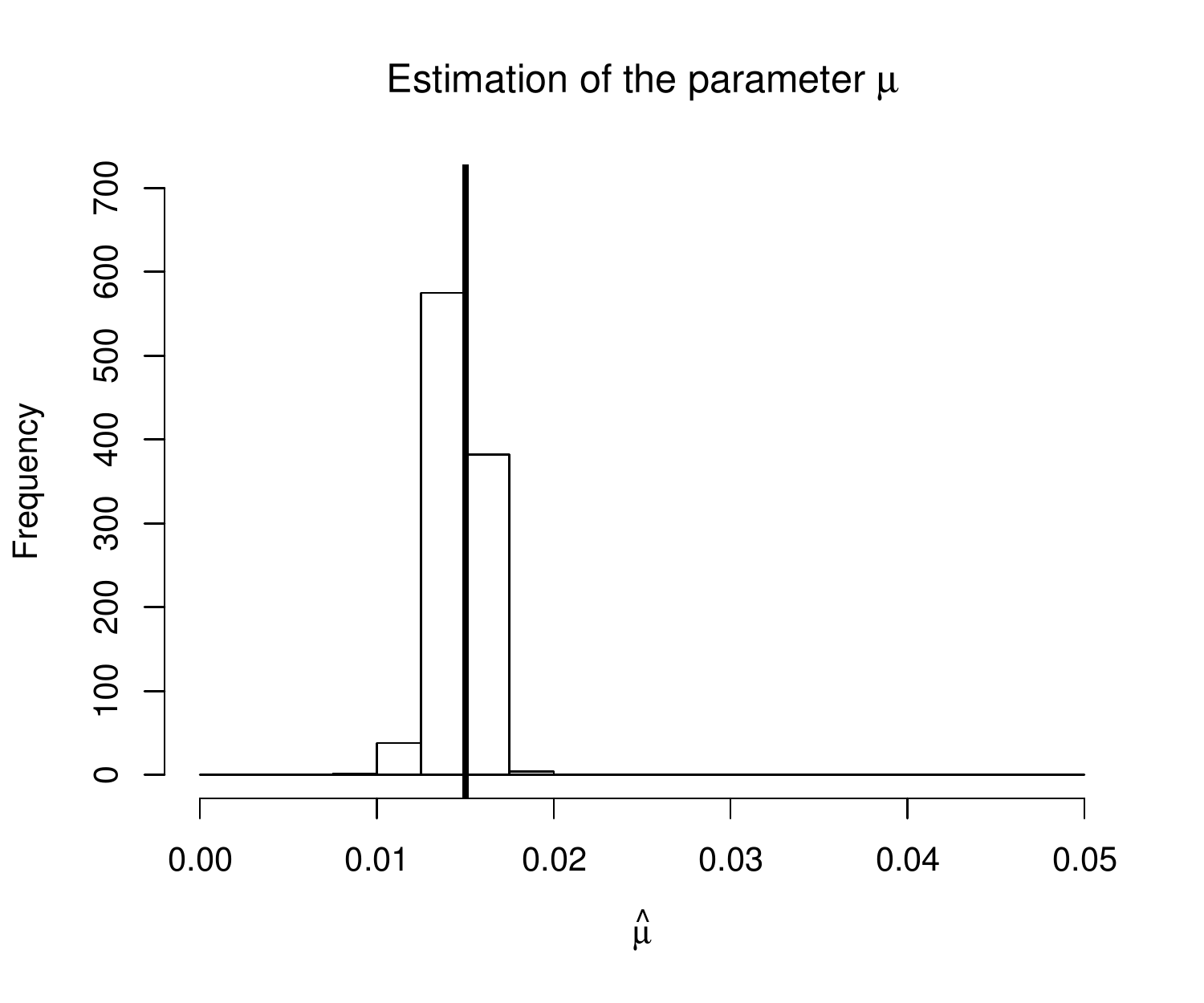}
\includegraphics[height=0.49\textwidth,angle=270]{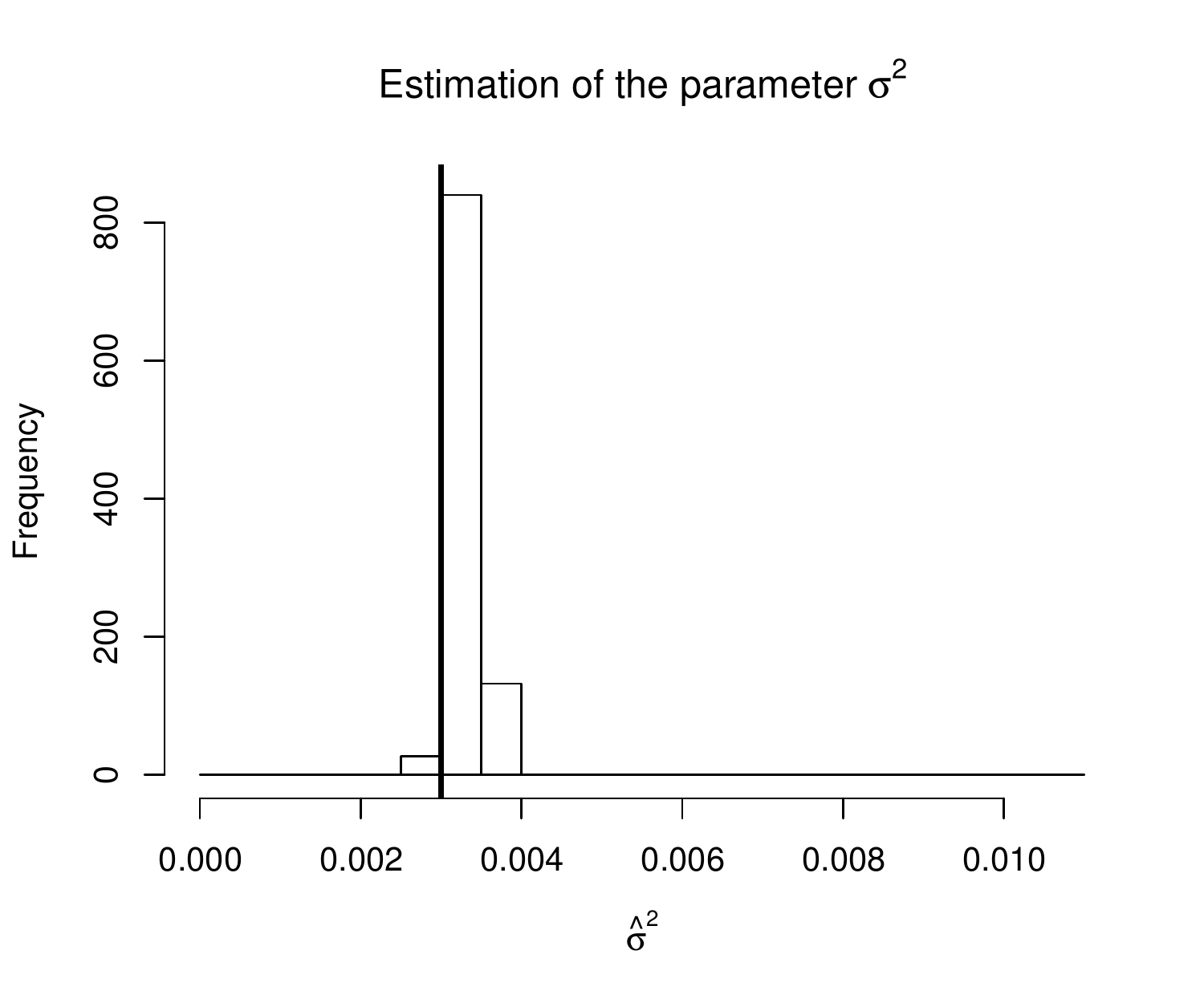}\\
\includegraphics[height=0.49\textwidth,angle=270]{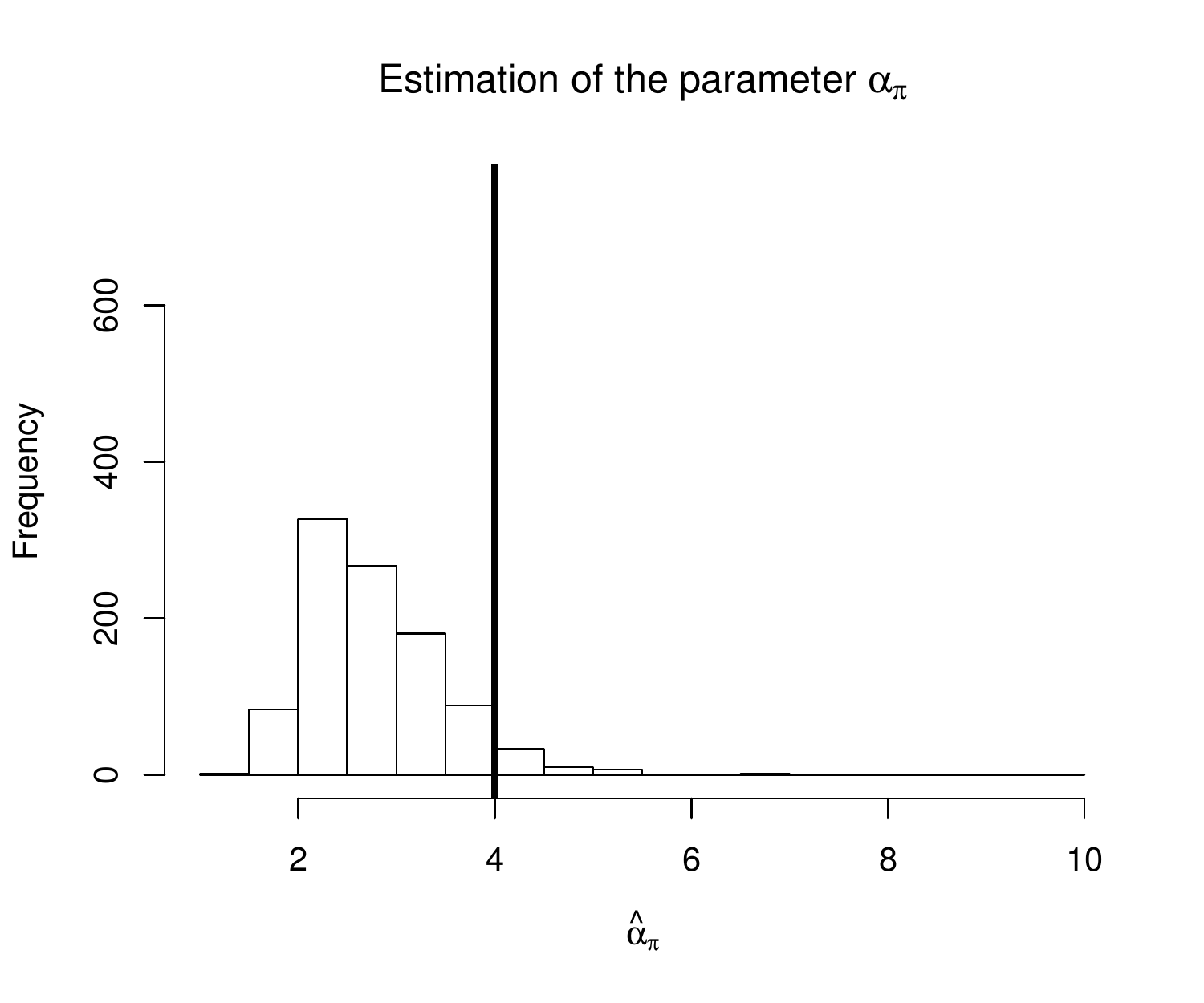}
\includegraphics[height=0.49\textwidth,angle=270]{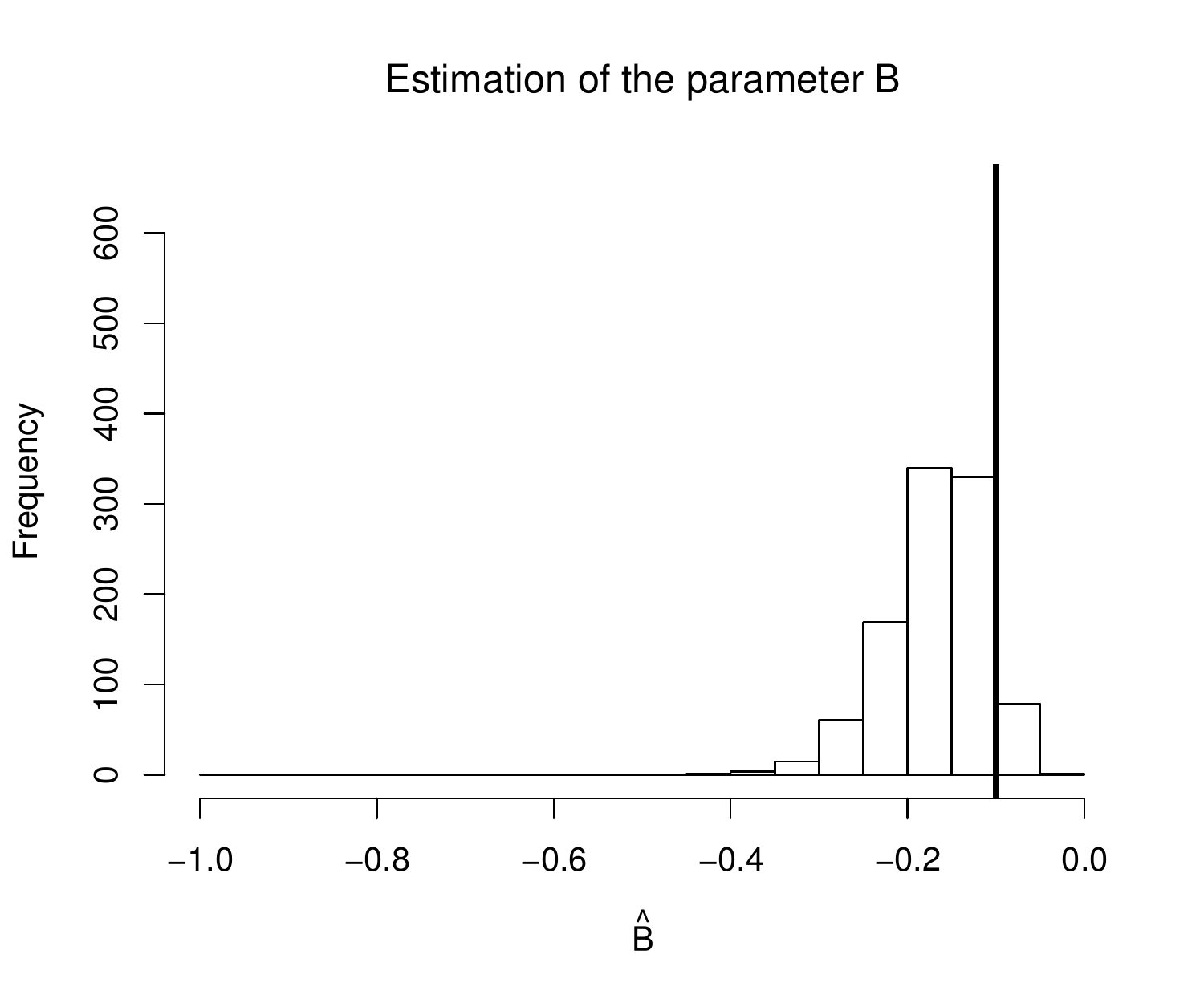}\\[5mm]
\includegraphics[height=0.49\textwidth,angle=270]{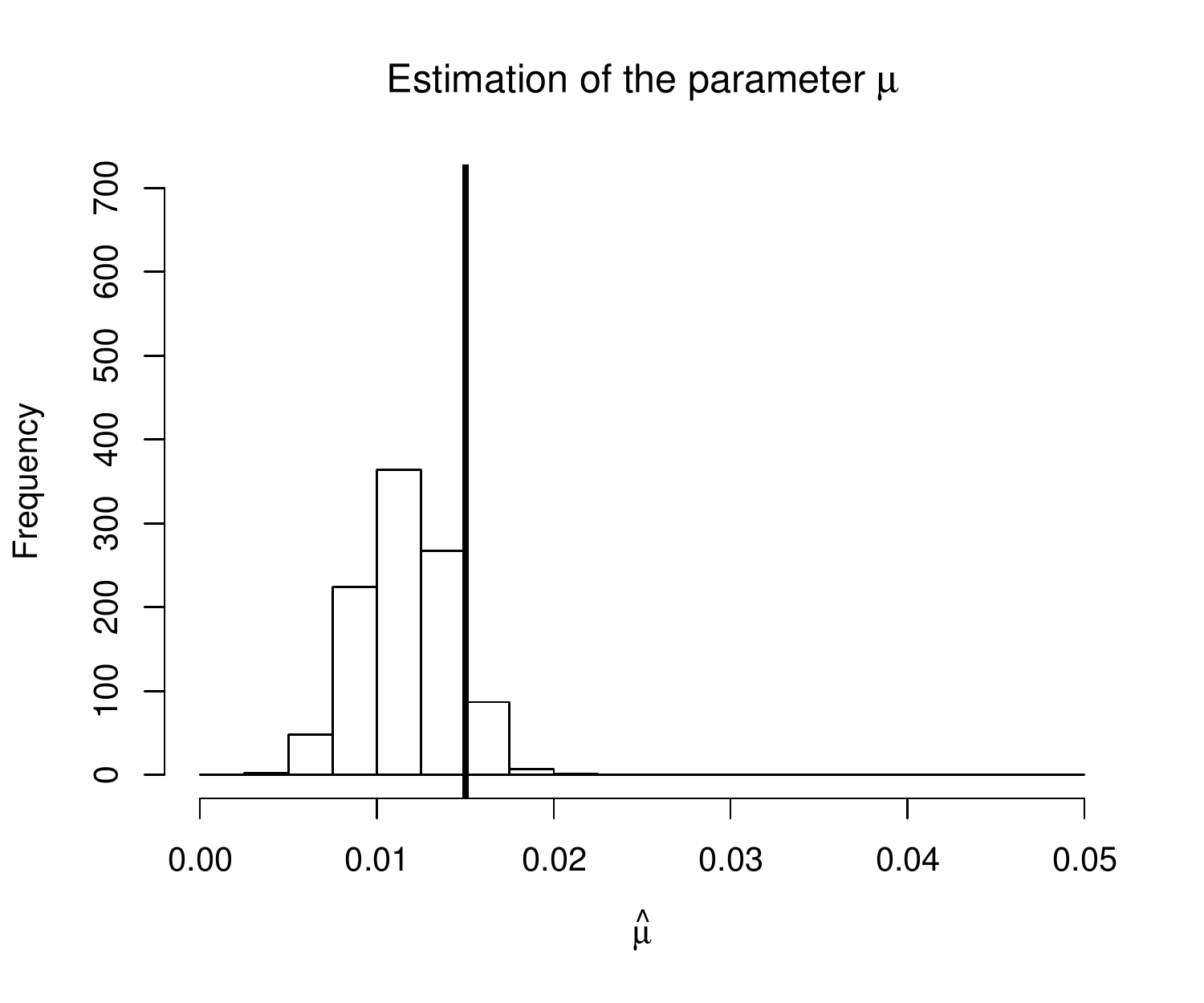}
\includegraphics[height=0.49\textwidth,angle=270]{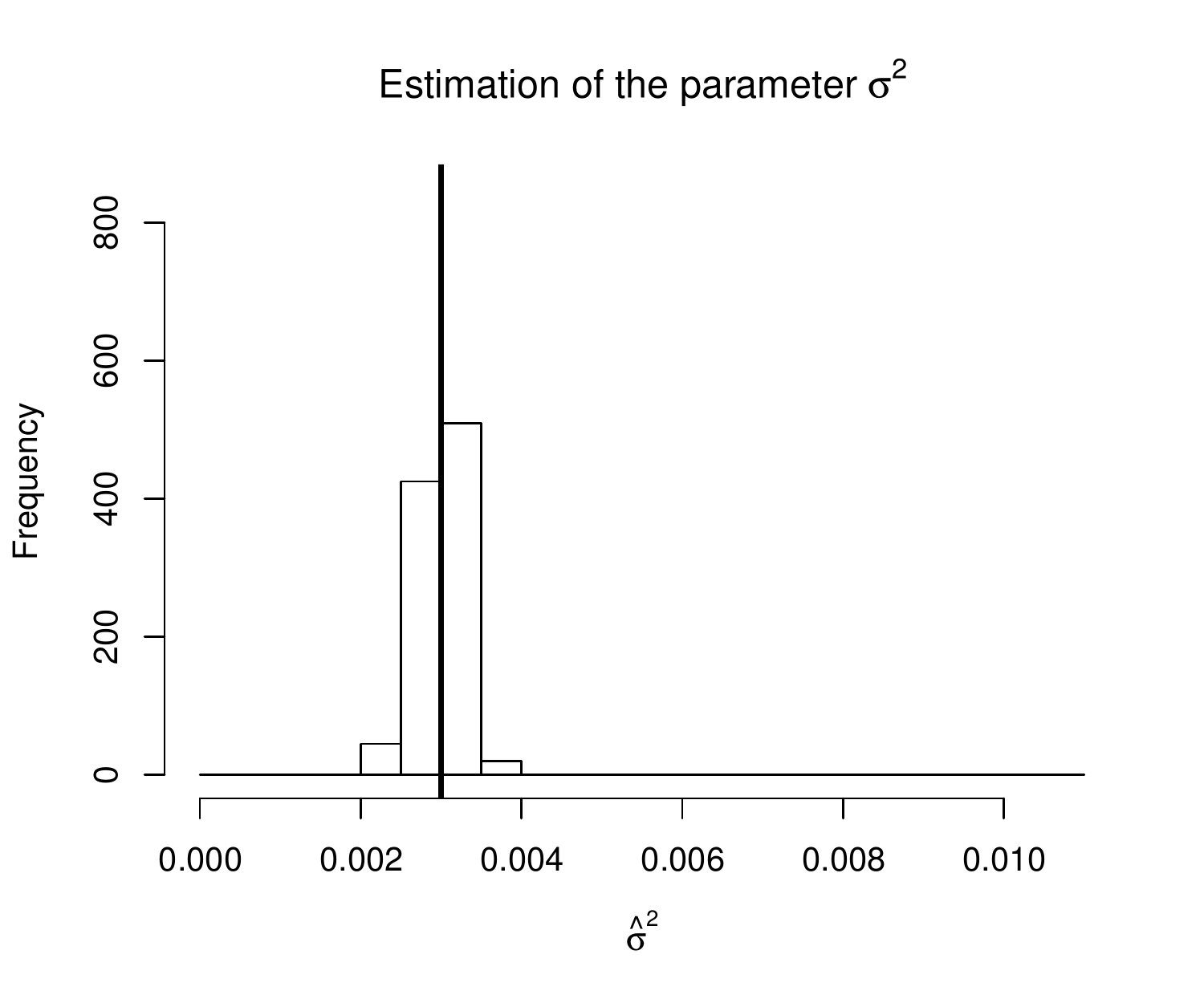}\\
\includegraphics[height=0.49\textwidth,angle=270]{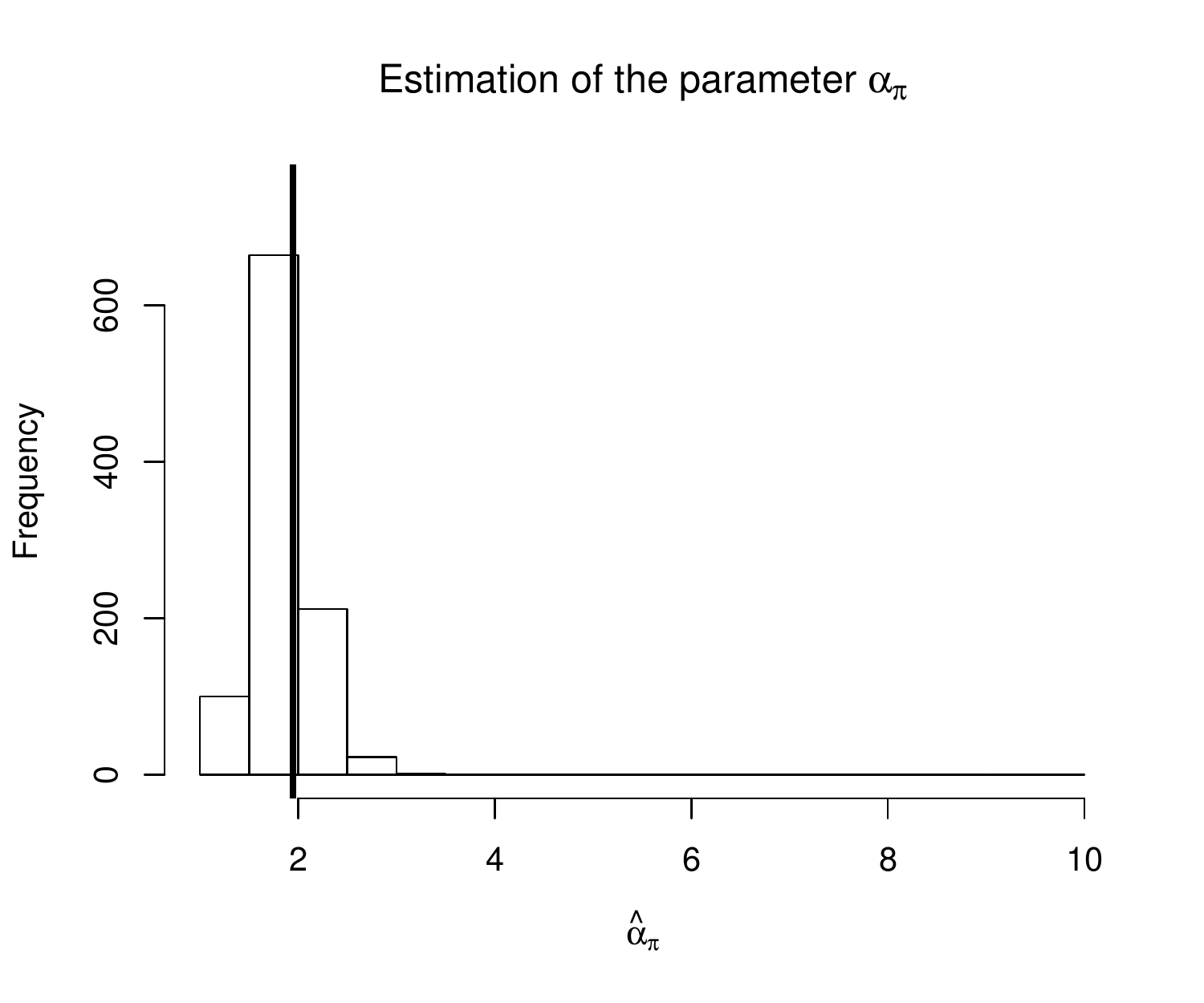}
\includegraphics[height=0.49\textwidth,angle=270]{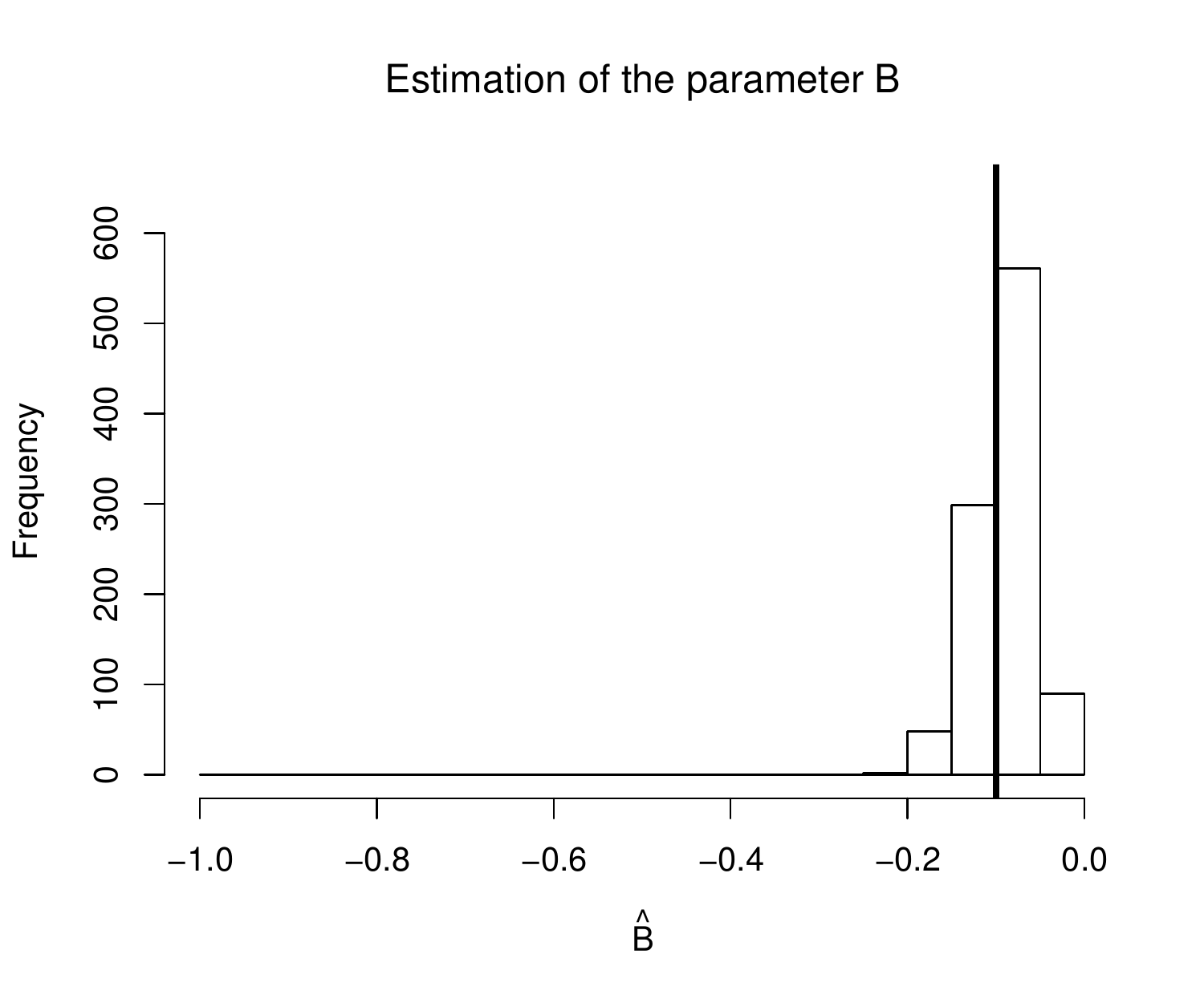}
\caption{Histograms of parameter estimates of 1000 paths of length 10000  of a supOU process with short (upper set of plots) and with long memory (lower set of plots). The true values are indicated by black lines.}
\label{figsupOU10K}
\end{figure}

\begin{figure}[p]
\center
\includegraphics[height=0.49\textwidth,angle=270]{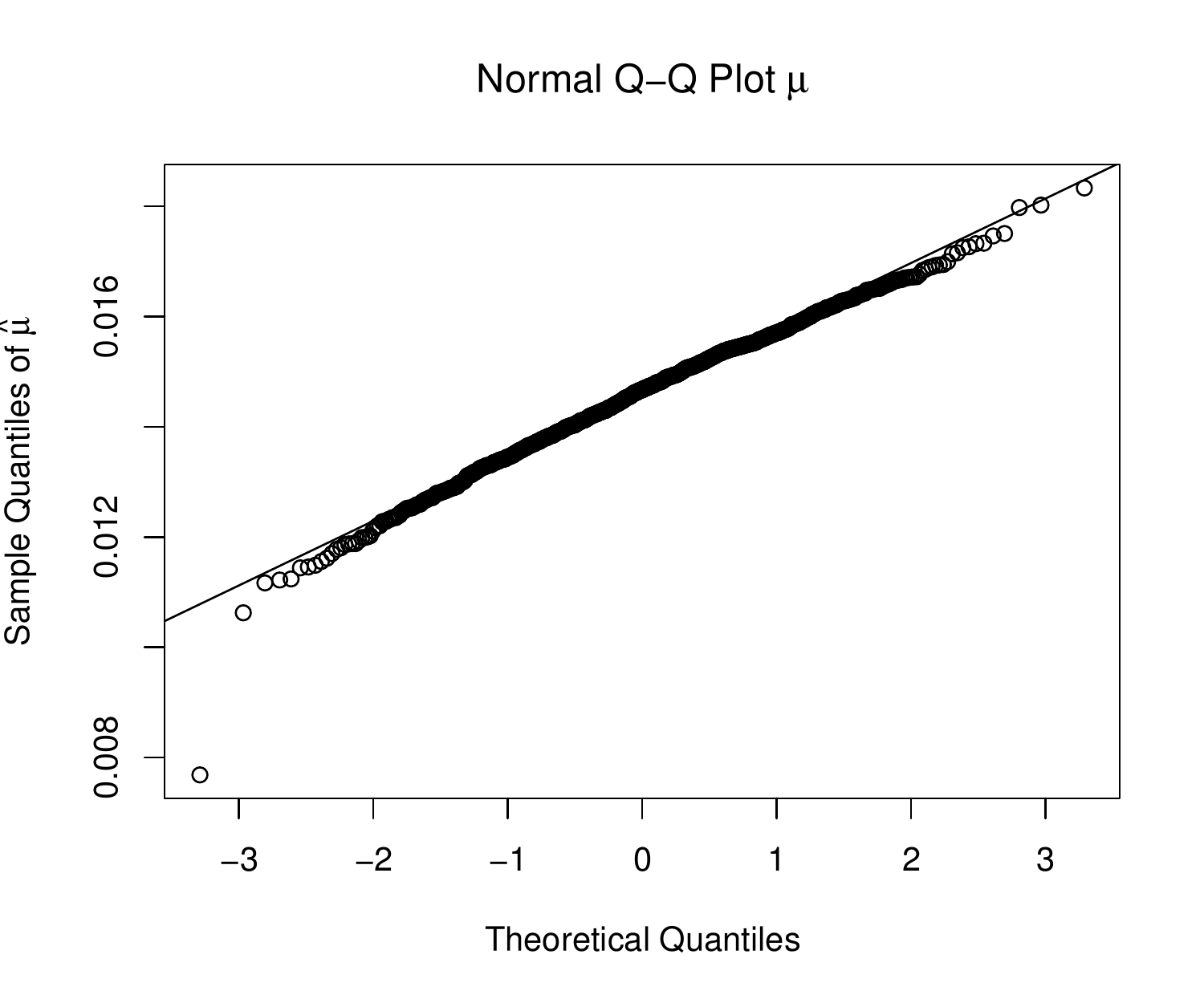}
\includegraphics[height=0.49\textwidth,angle=270]{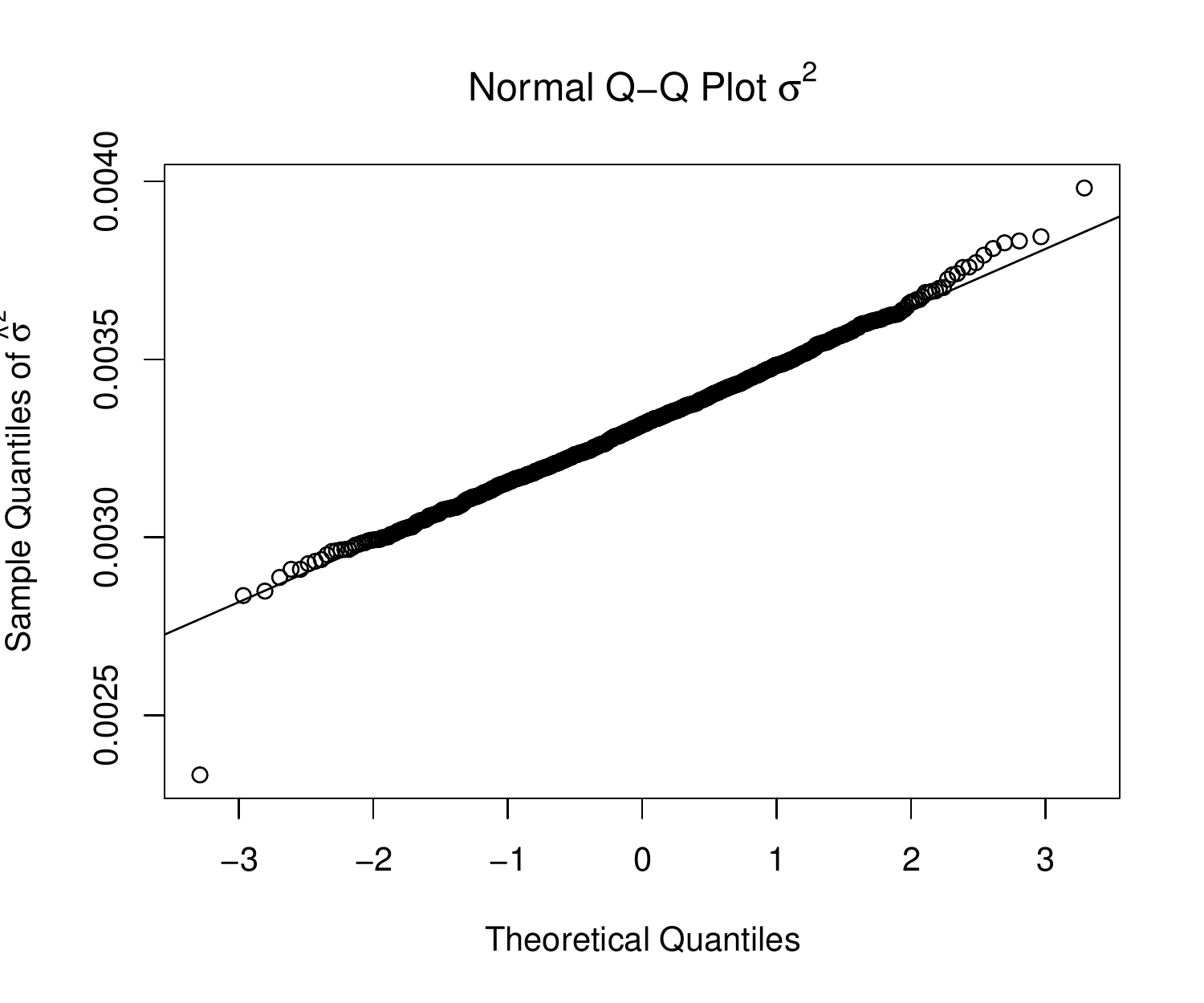}\\
\includegraphics[height=0.49\textwidth,angle=270]{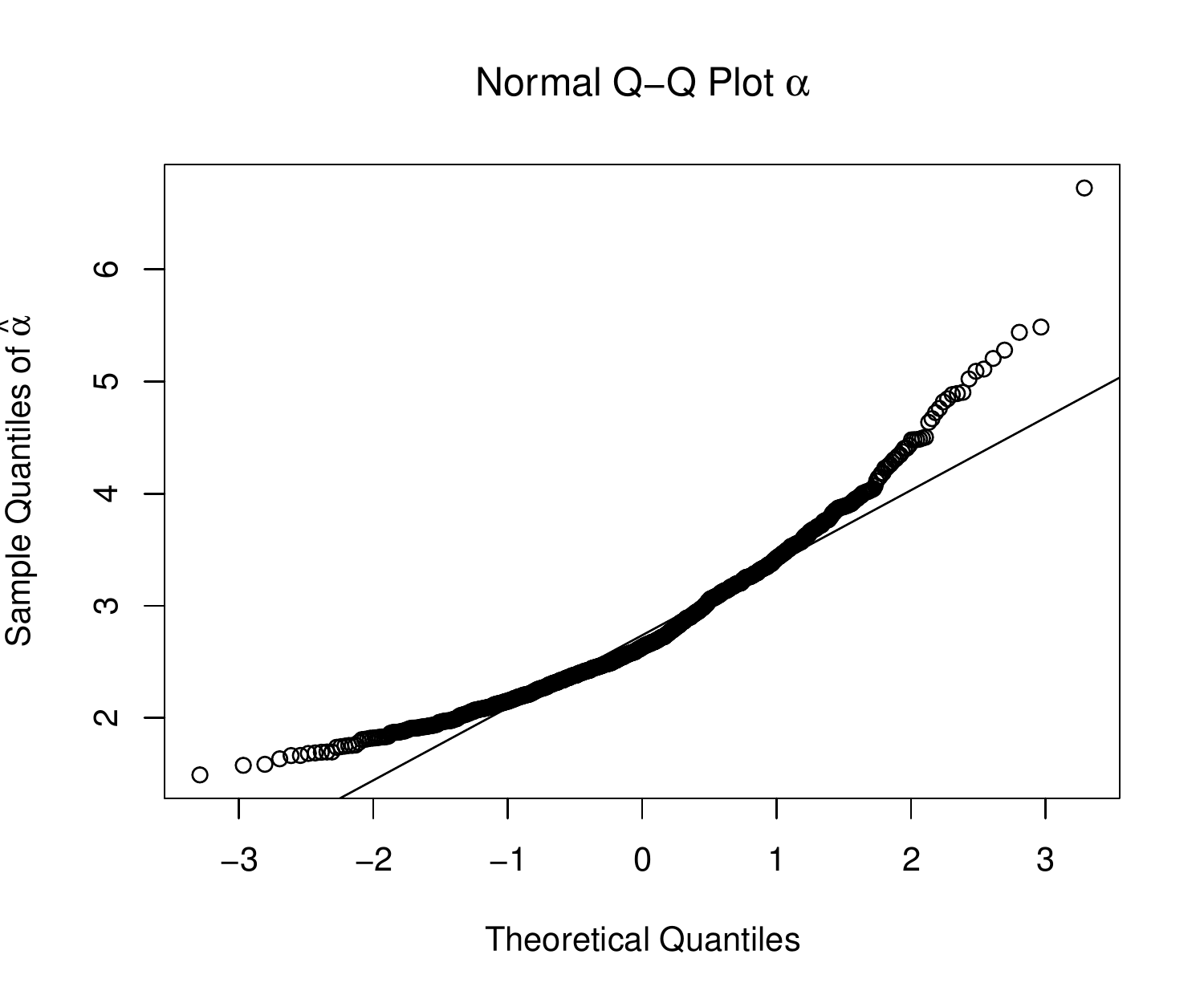}
\includegraphics[height=0.49\textwidth,angle=270]{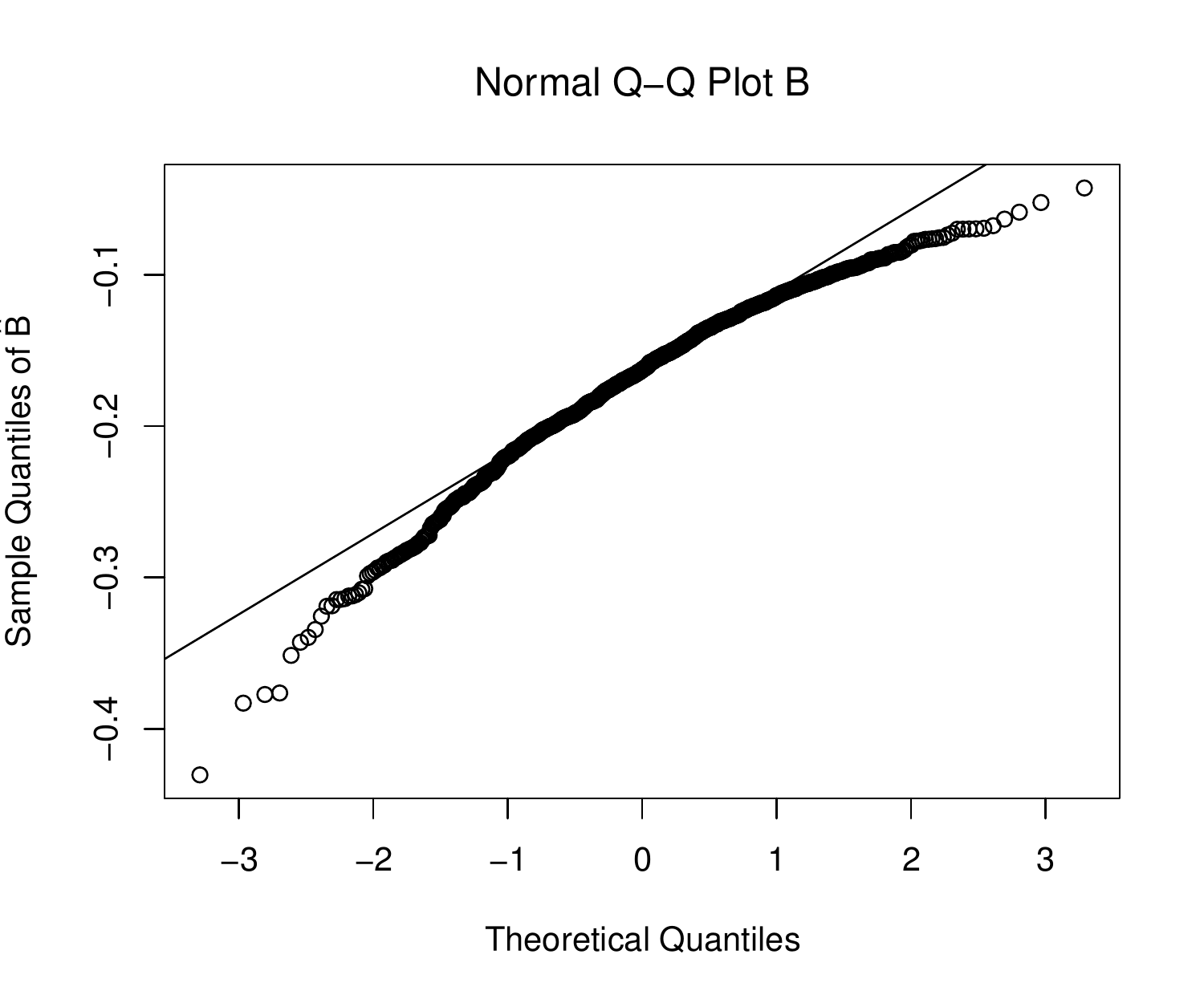}\\[5mm]
\includegraphics[height=0.49\textwidth,angle=270]{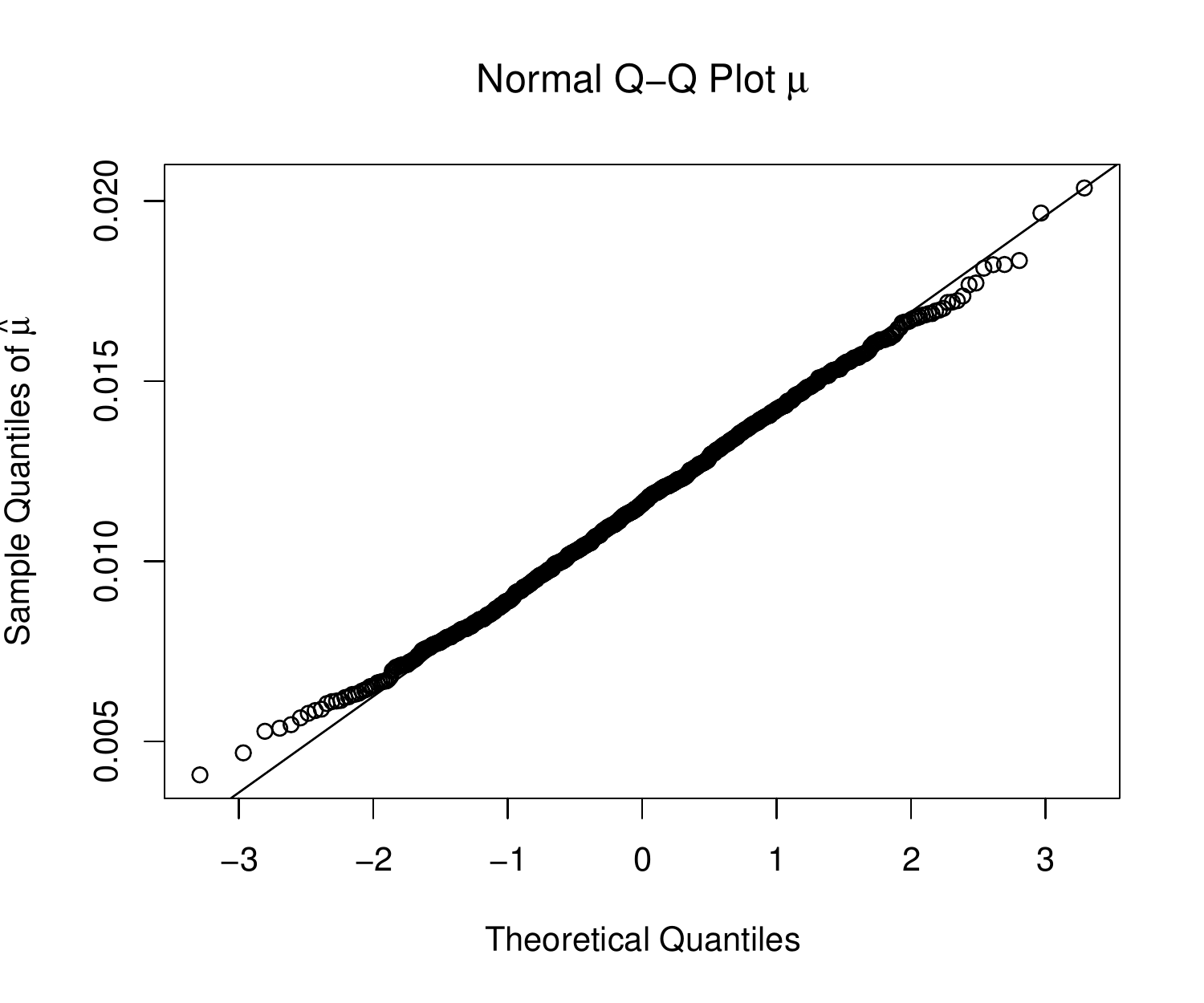}
\includegraphics[height=0.49\textwidth,angle=270]{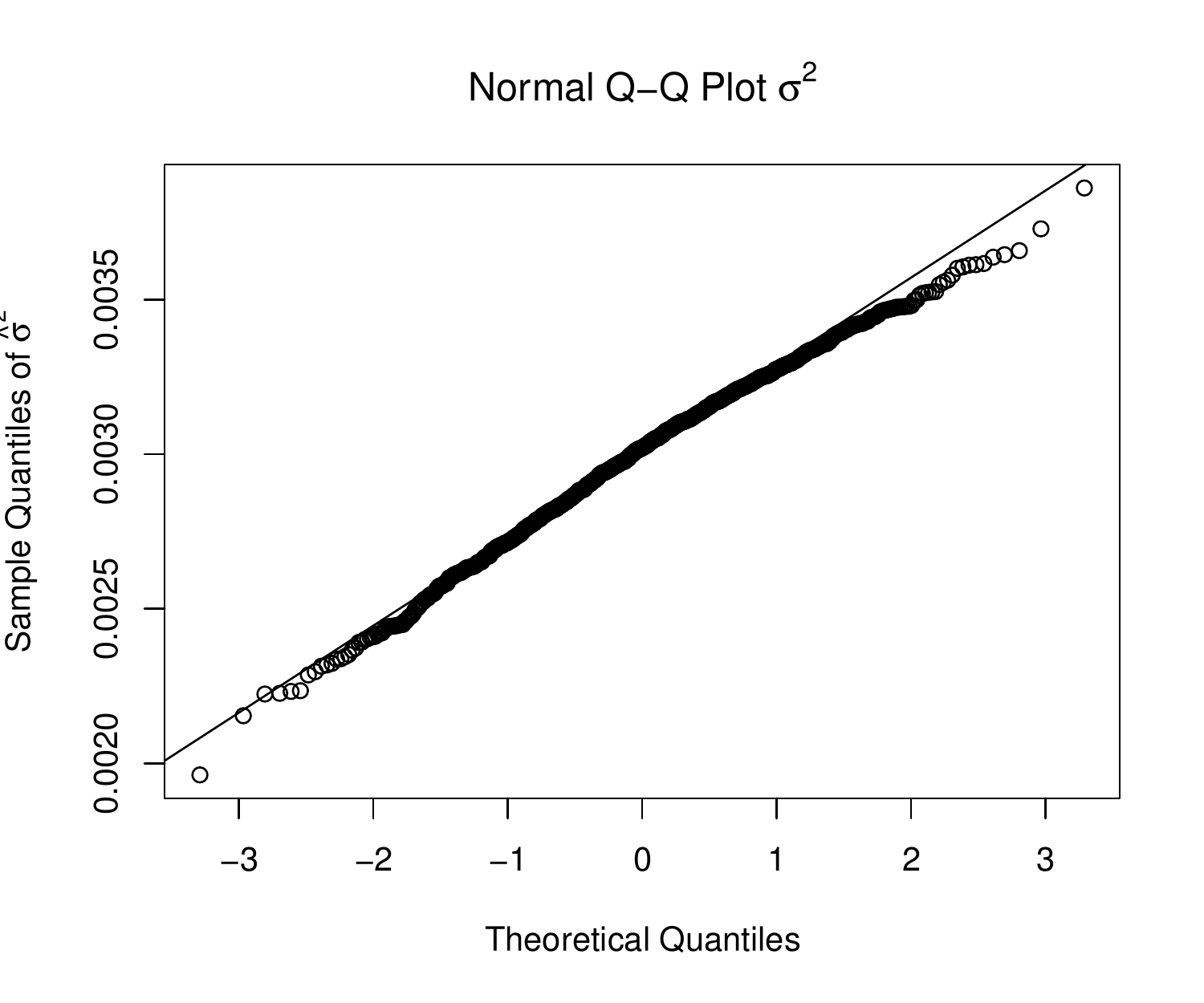}\\
\includegraphics[height=0.49\textwidth,angle=270]{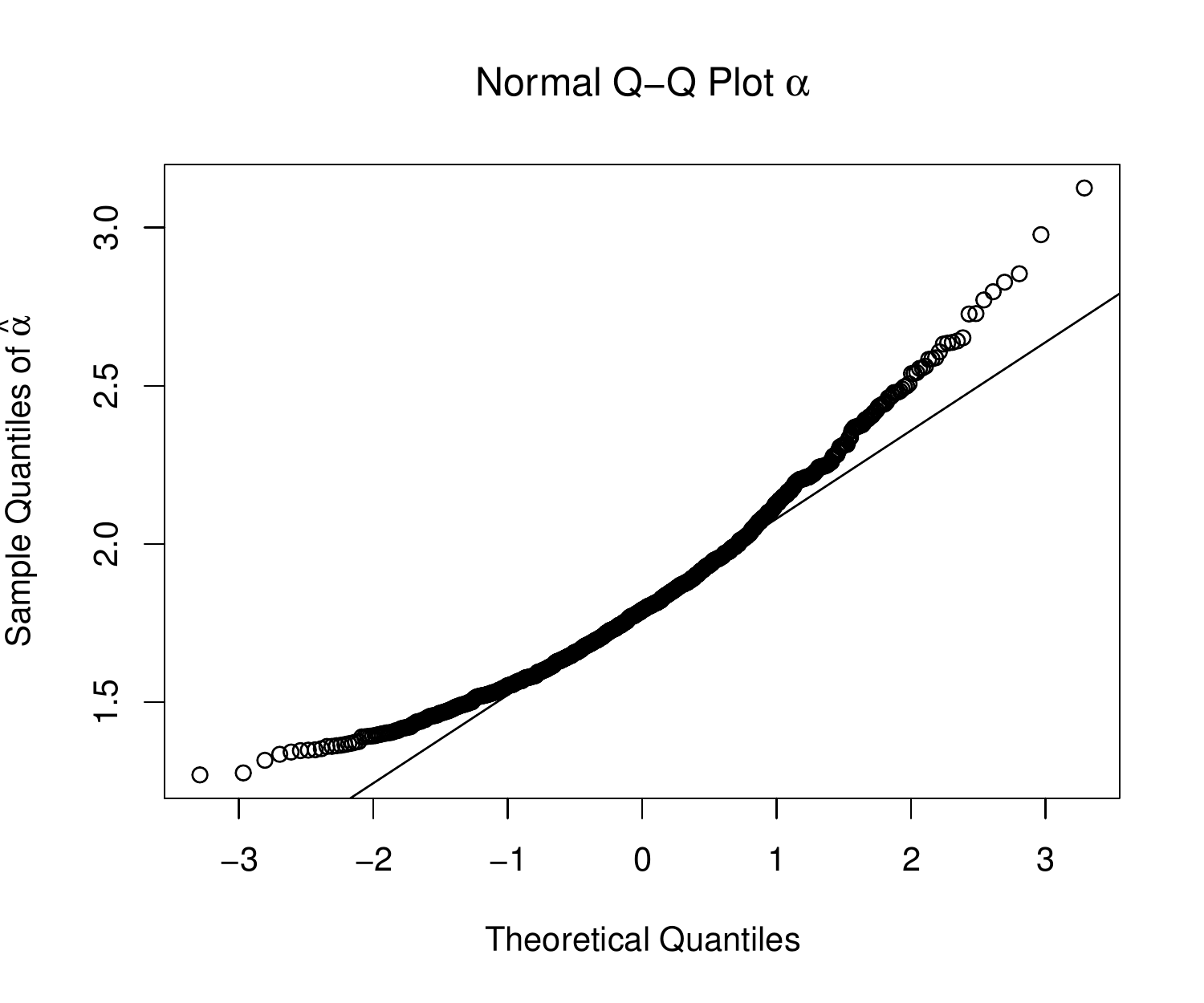}
\includegraphics[height=0.49\textwidth,angle=270]{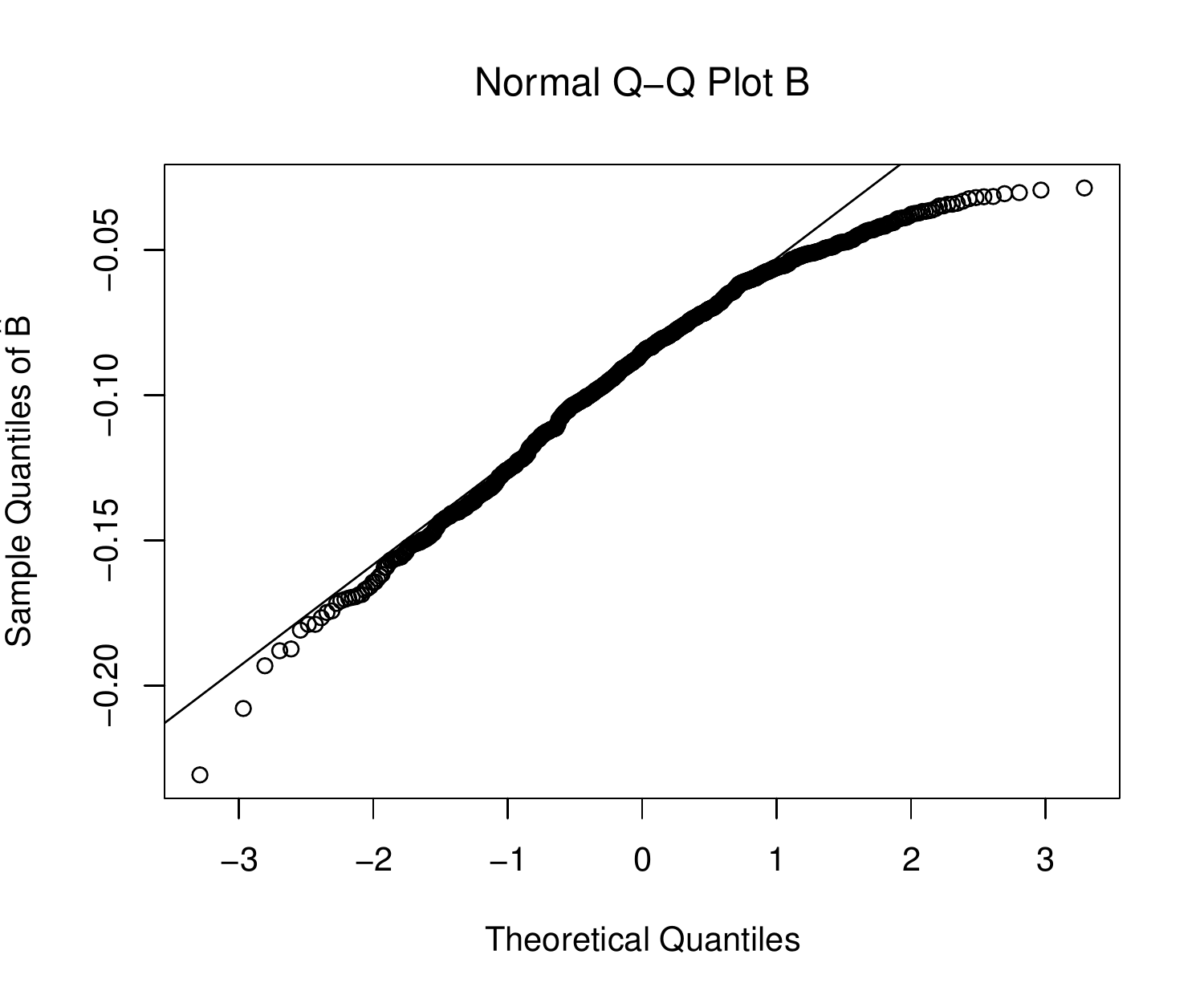}
\caption{Normal QQ-Plots of parameter estimates of 1000 paths of length 10000  of a supOU process with short (upper set of plots) and with long memory (lower set of plots).}
\label{figsupOU10Kqq}
\end{figure}

\begin{figure}[p]
\center
\includegraphics[height=0.49\textwidth,angle=270]{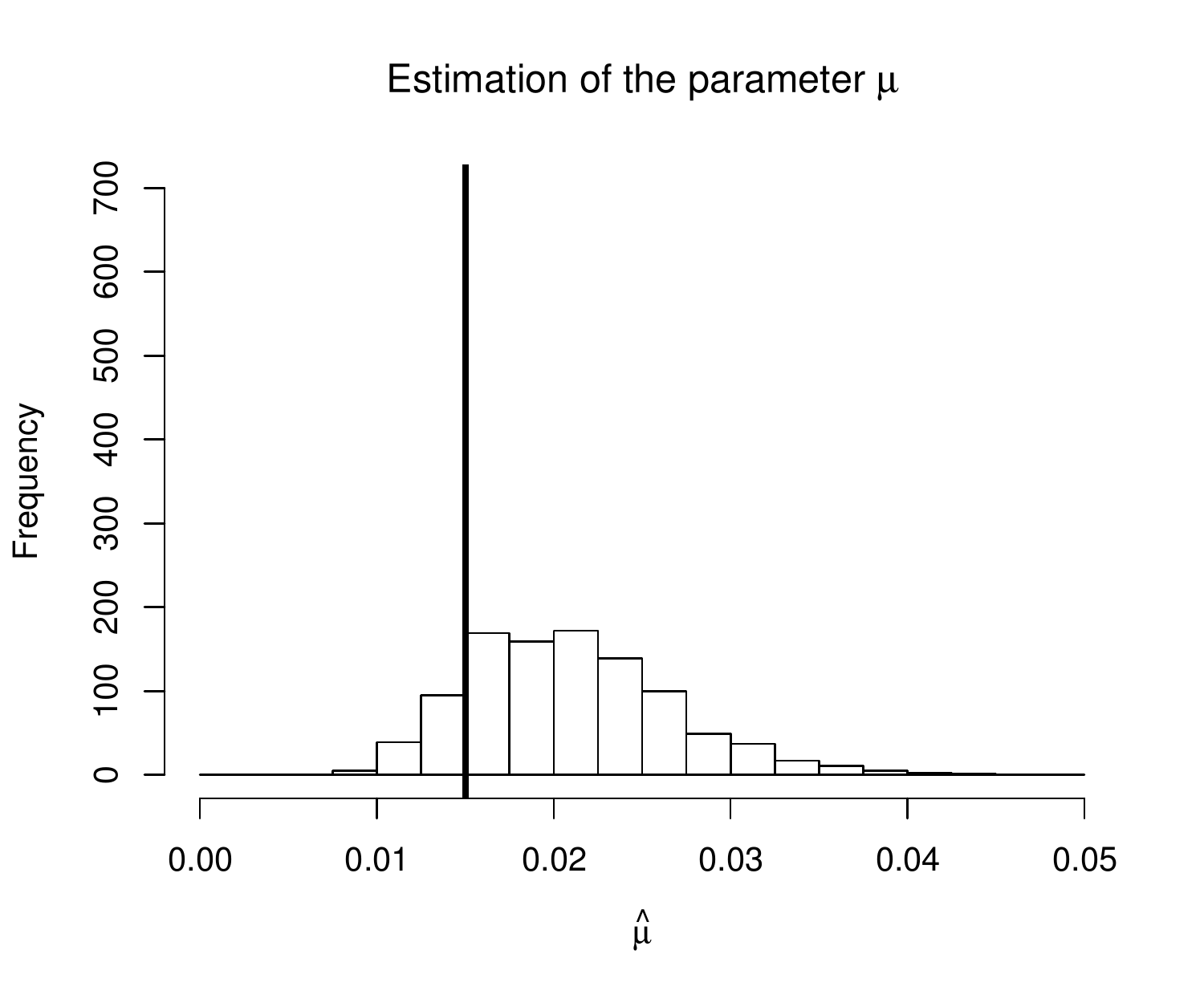}
\includegraphics[height=0.49\textwidth,angle=270]{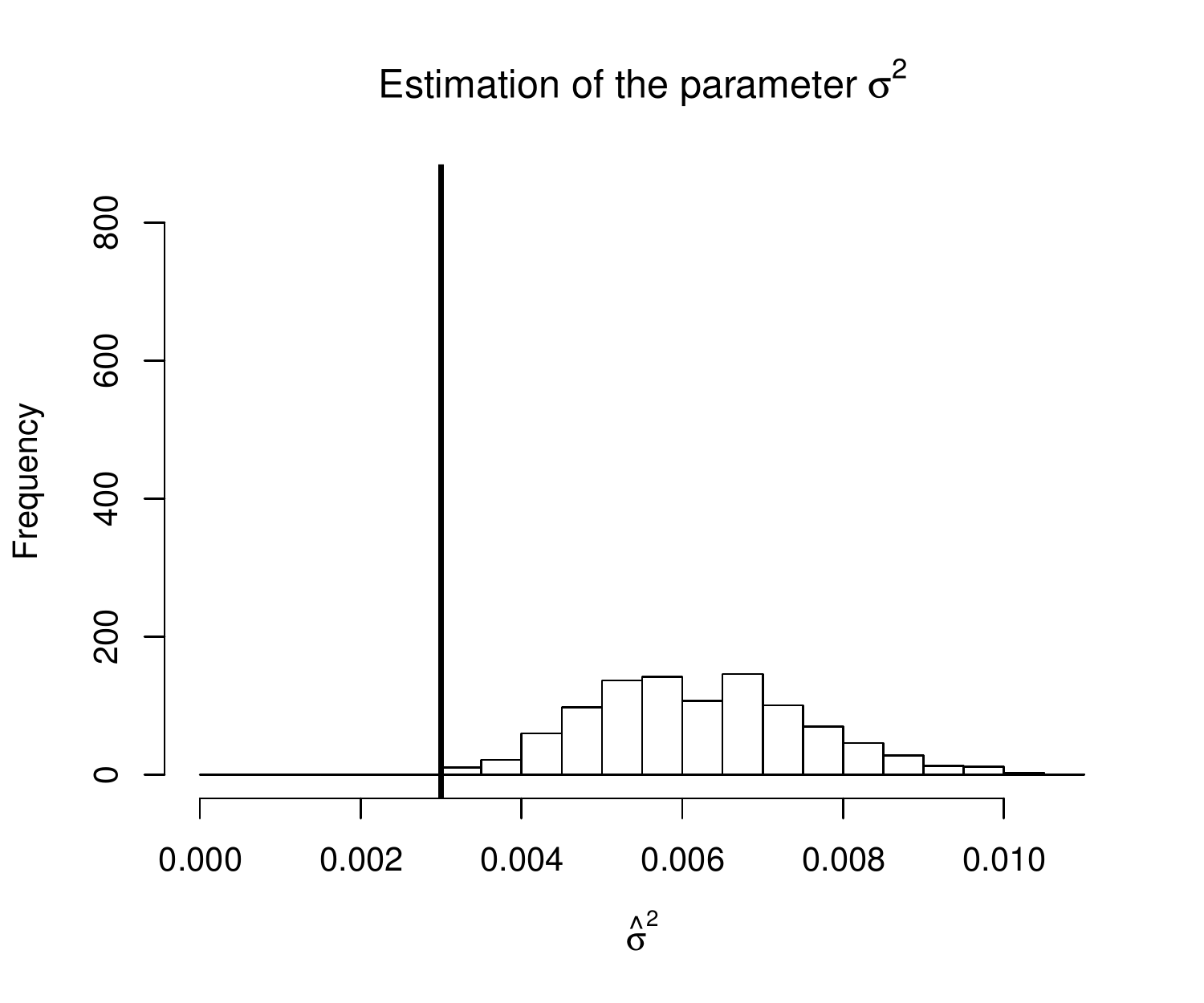}\\
\includegraphics[height=0.49\textwidth,angle=270]{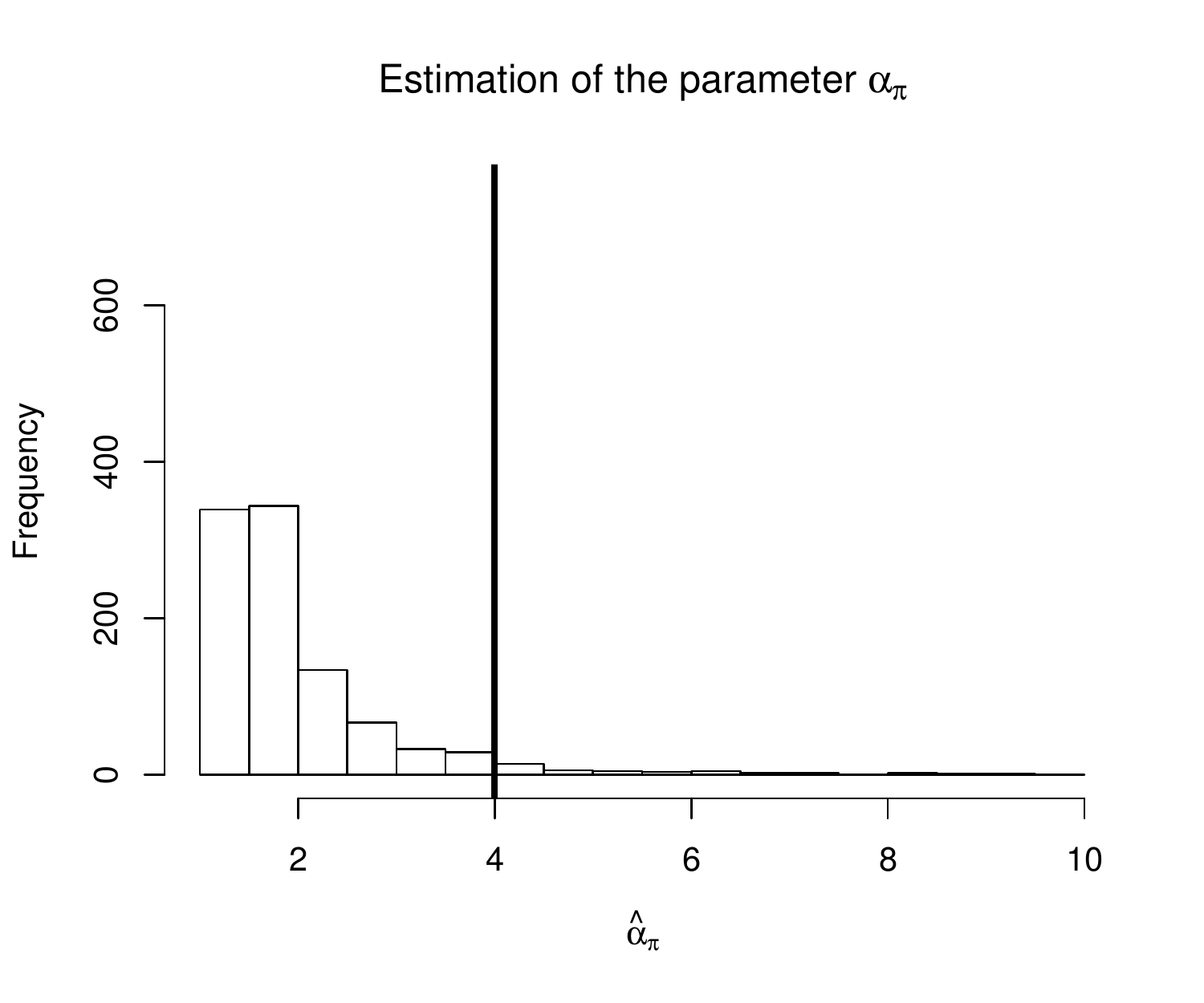}
\includegraphics[height=0.49\textwidth,angle=270]{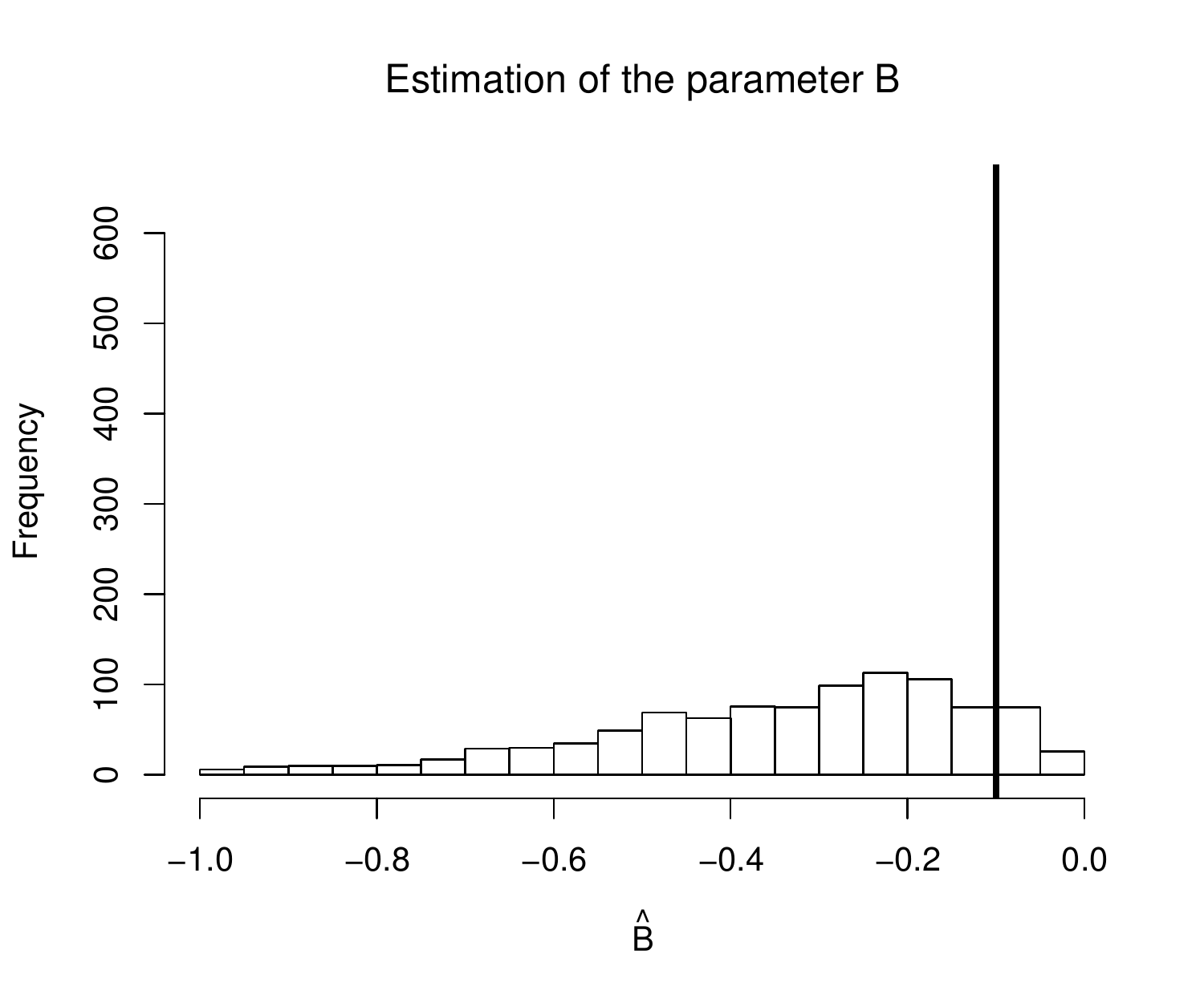}\\[5mm]
\includegraphics[height=0.49\textwidth,angle=270]{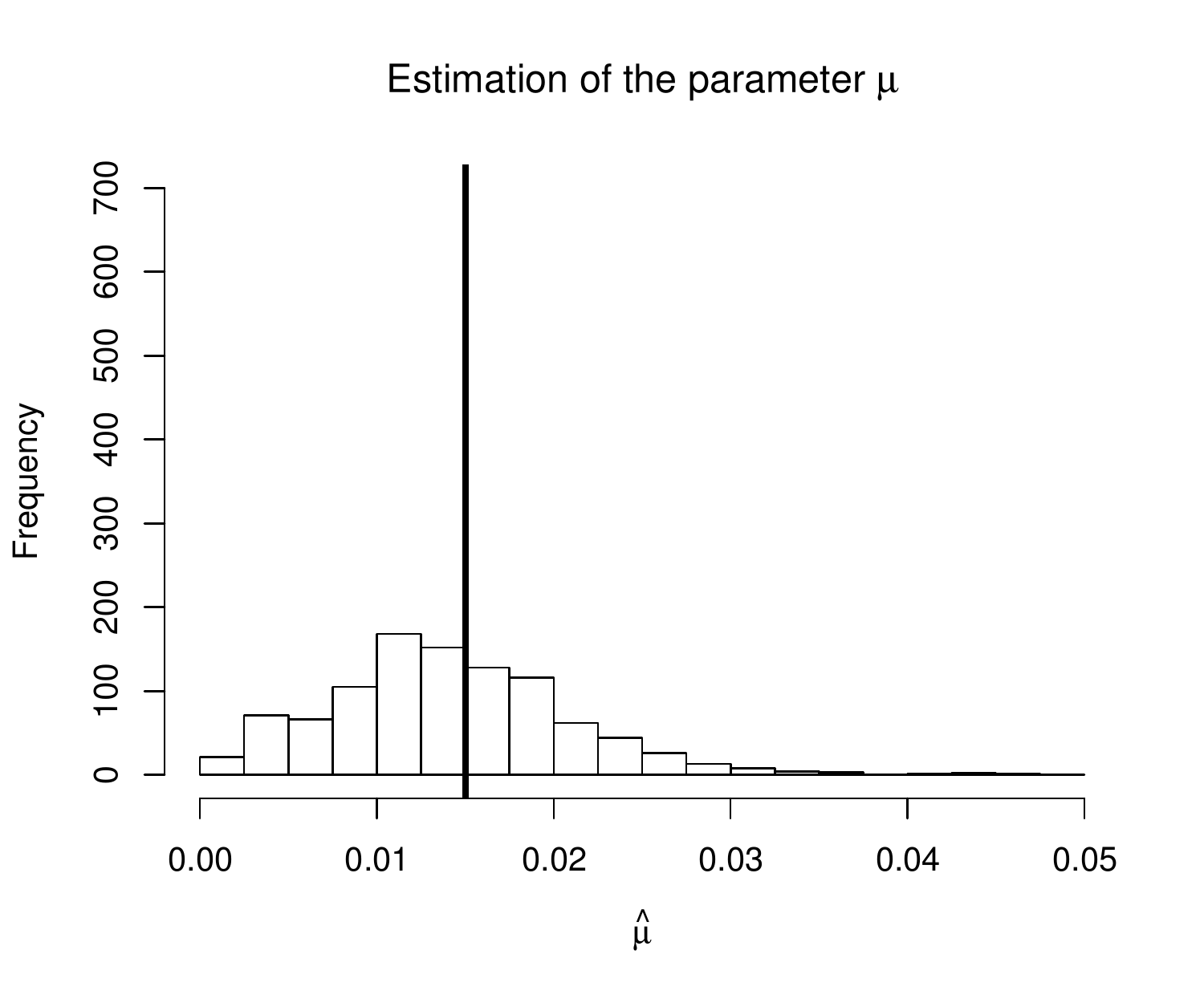}
\includegraphics[height=0.49\textwidth,angle=270]{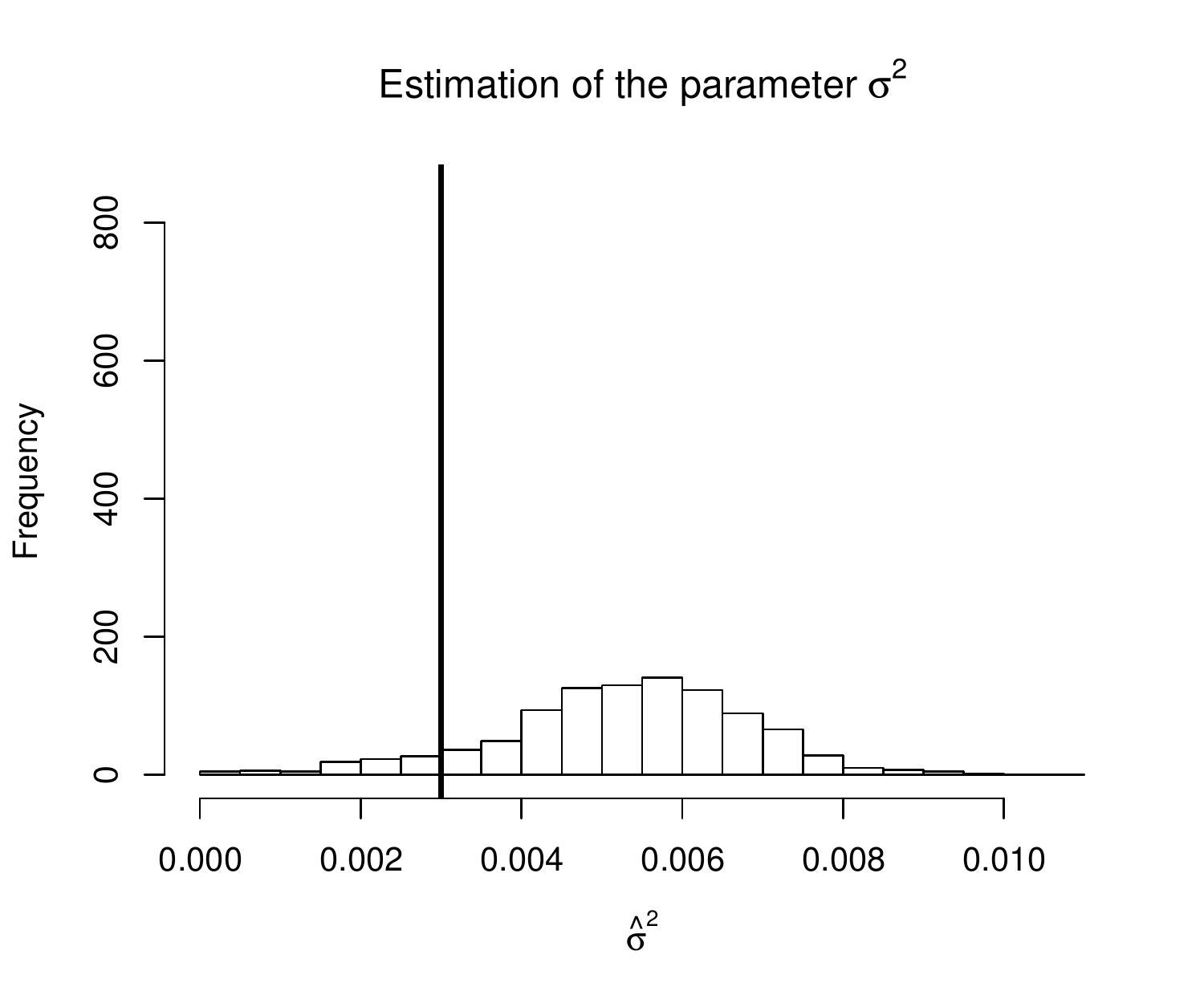}\\
\includegraphics[height=0.49\textwidth,angle=270]{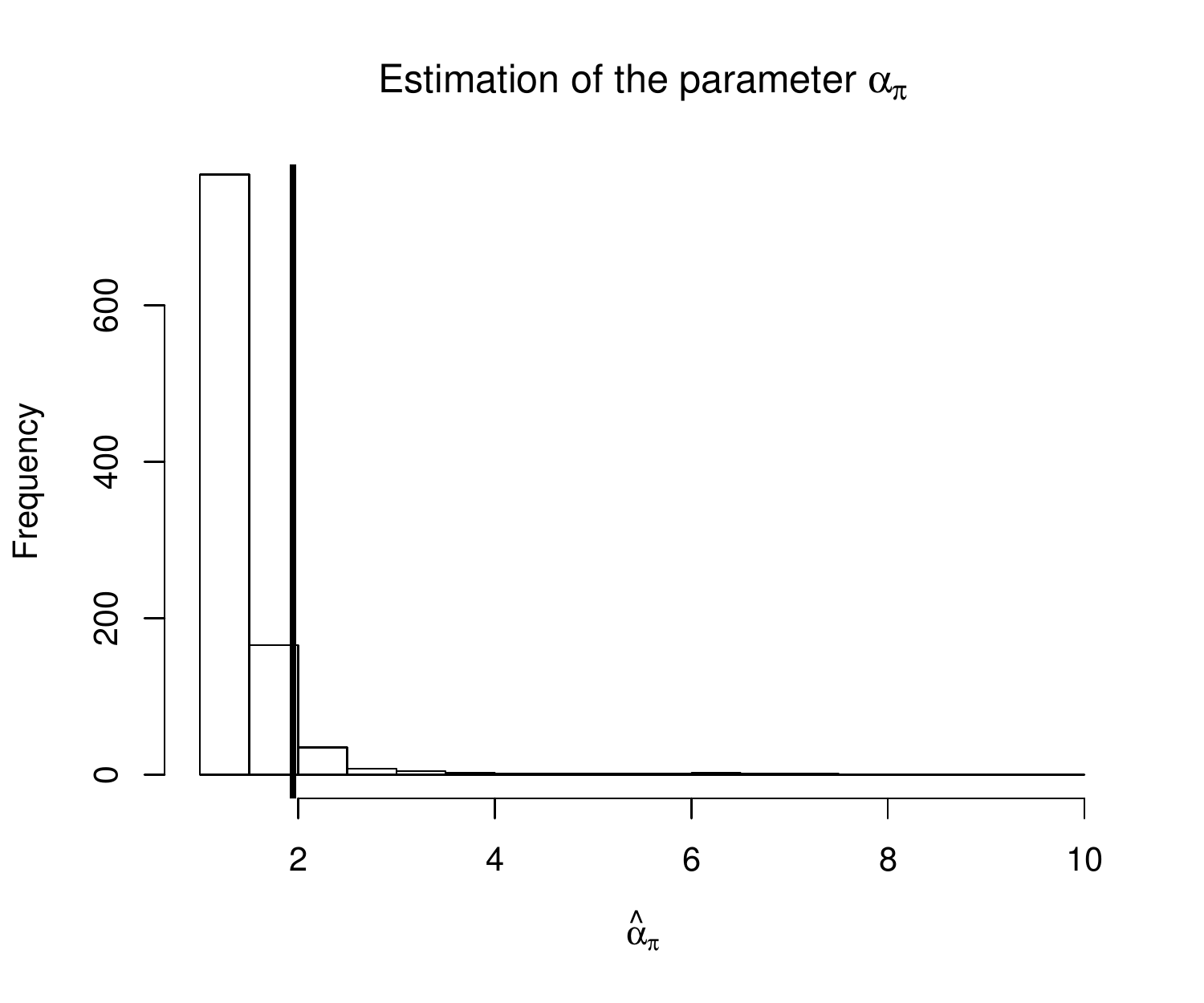}
\includegraphics[height=0.49\textwidth,angle=270]{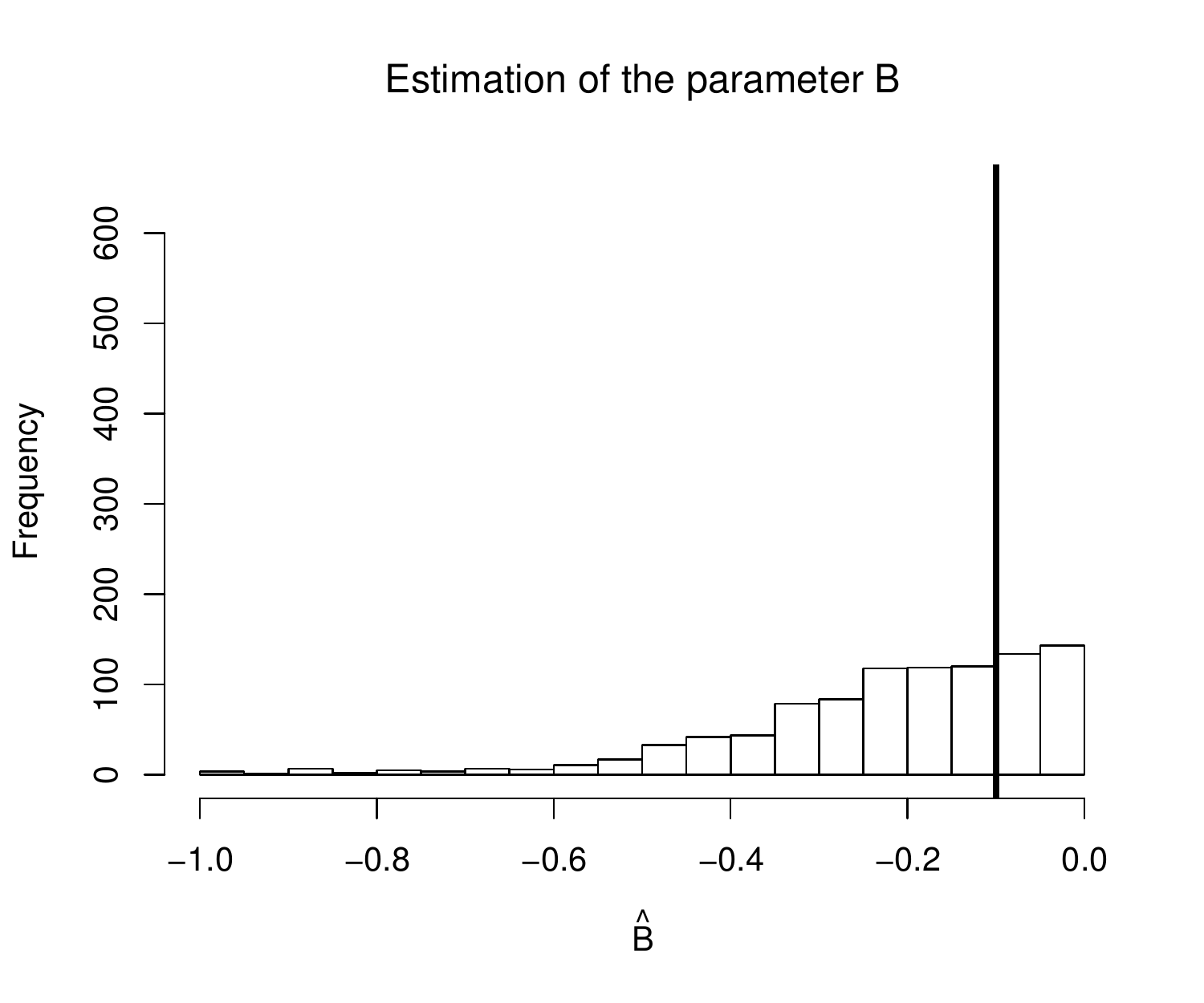}
\caption{Histograms of parameter estimates of 1000 paths of length 1000  of a supOU process with short (upper set of plots) and with long memory (lower set of plots). The true values are indicated by black lines.}
\label{figsupOU1K}
\end{figure}

\begin{figure}[p]
\center
\includegraphics[height=0.49\textwidth,angle=270]{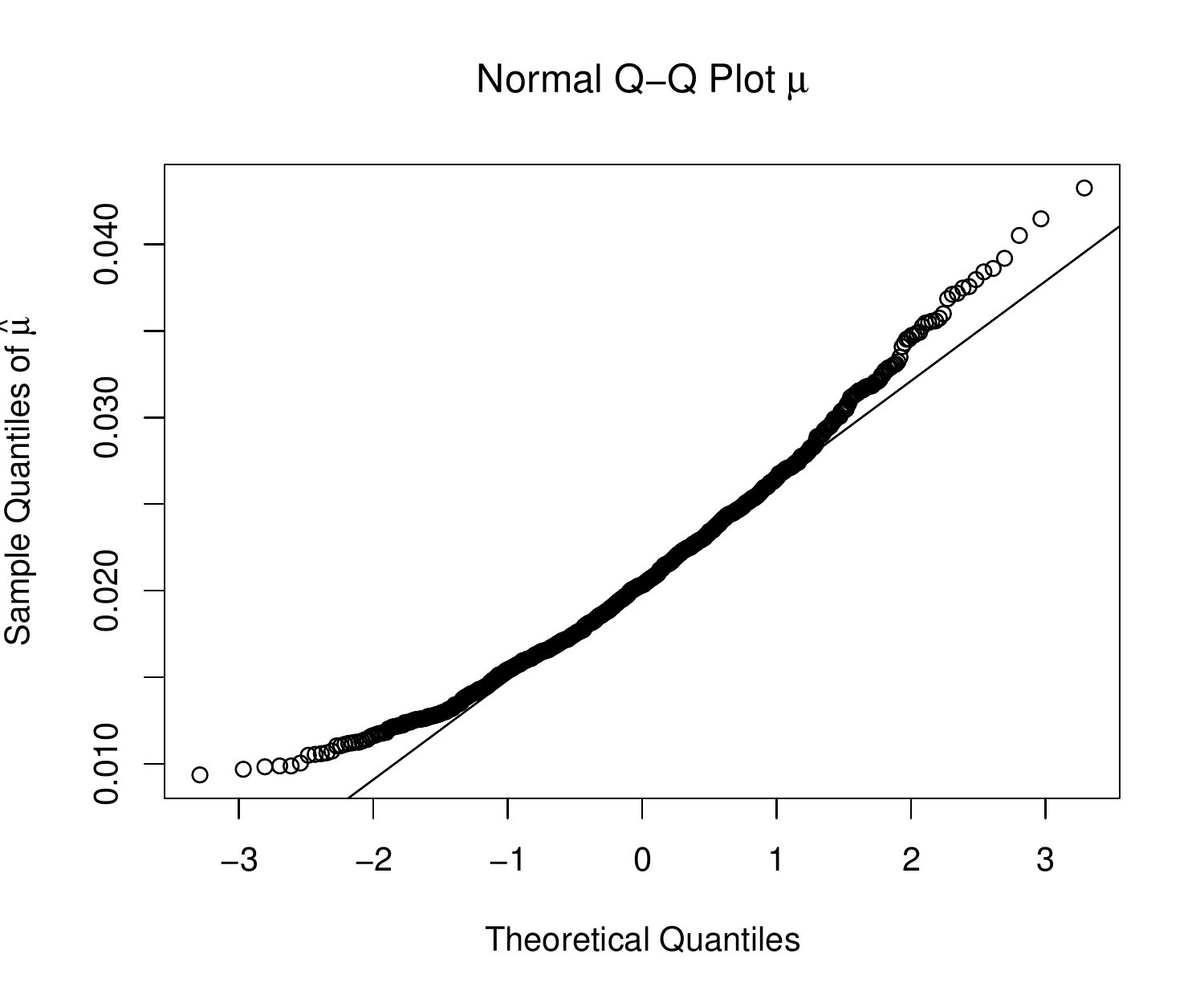}
\includegraphics[height=0.49\textwidth,angle=270]{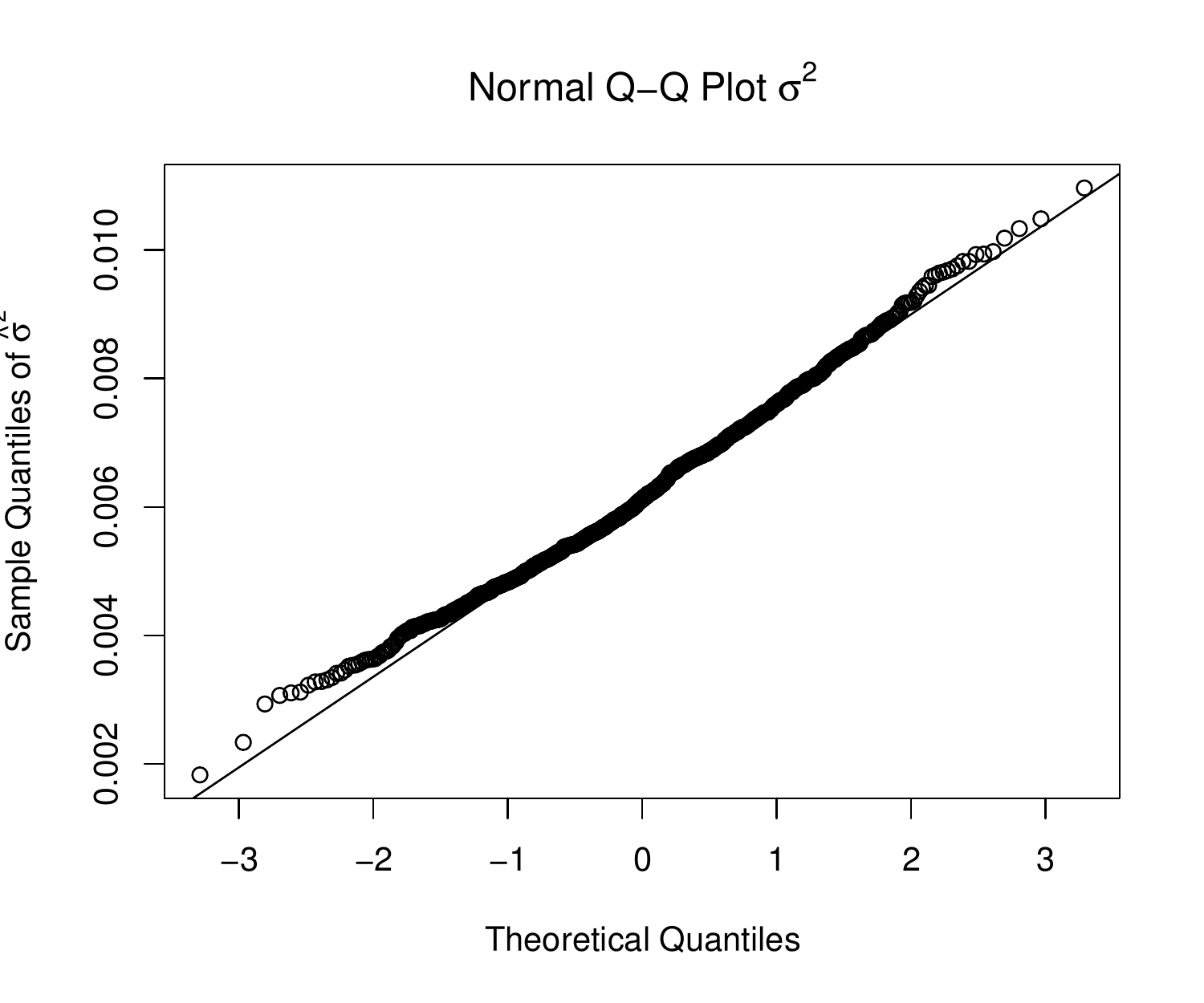}\\
\includegraphics[height=0.49\textwidth,angle=270]{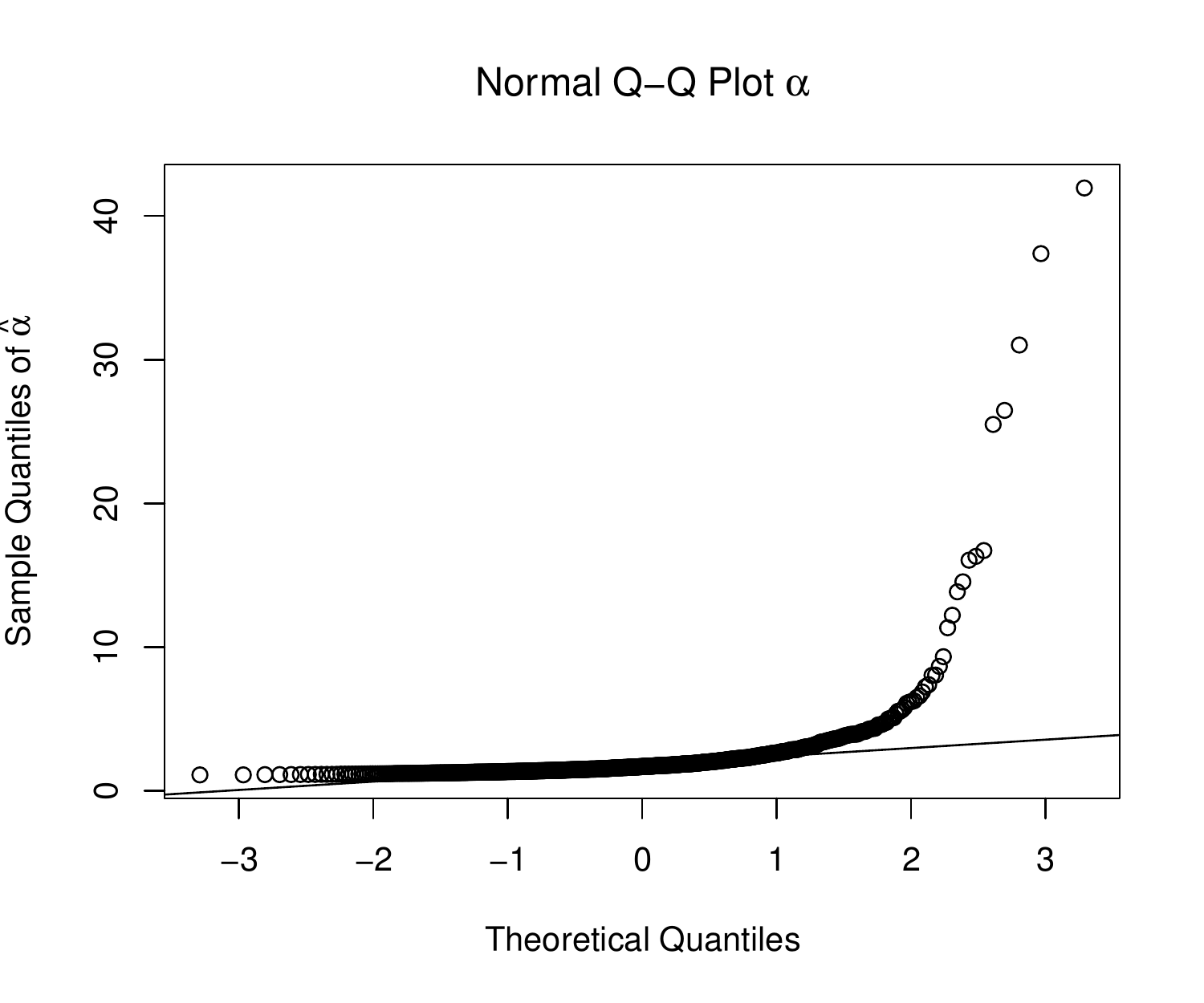}
\includegraphics[height=0.49\textwidth,angle=270]{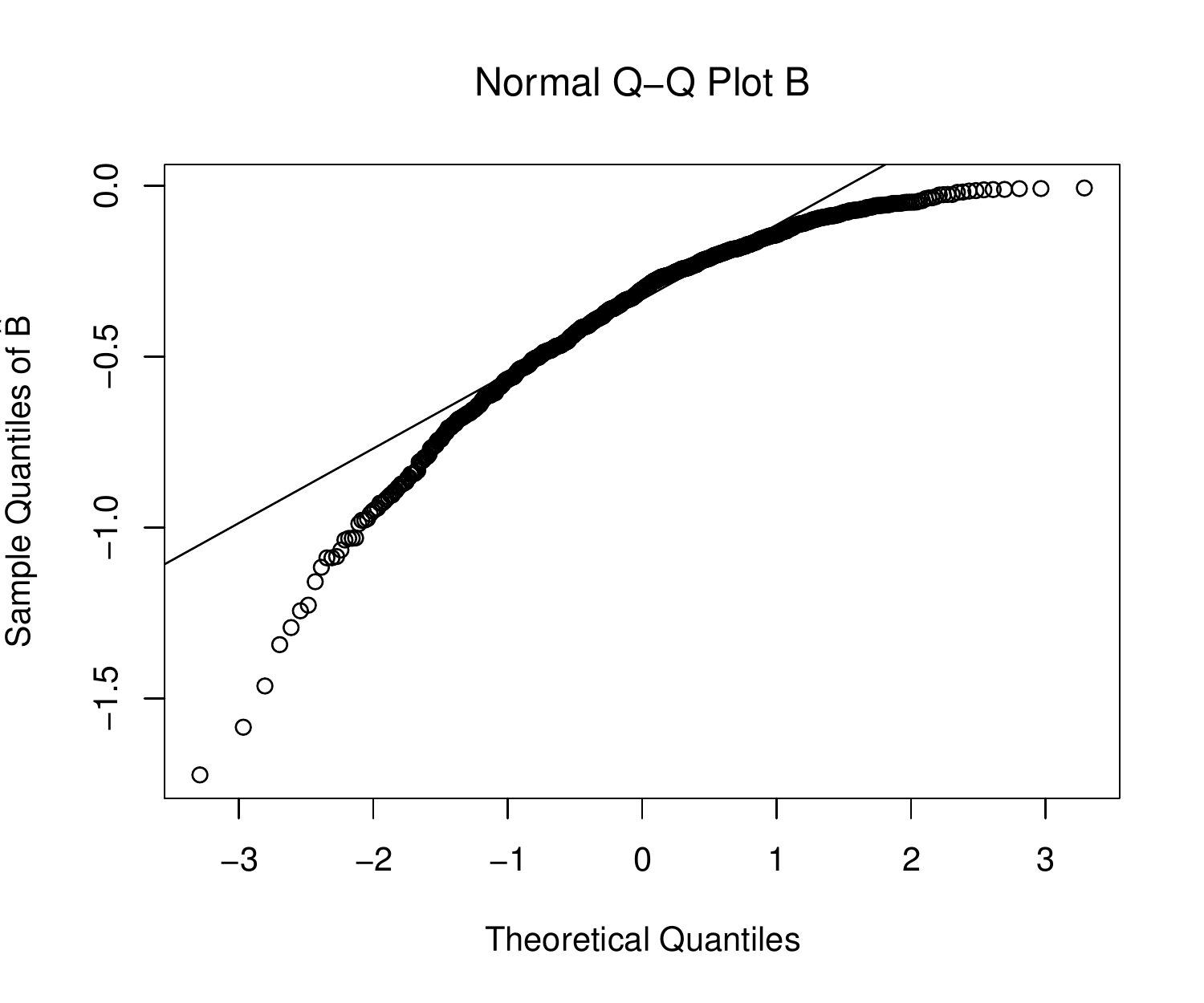}\\[5mm]
\includegraphics[height=0.49\textwidth,angle=270]{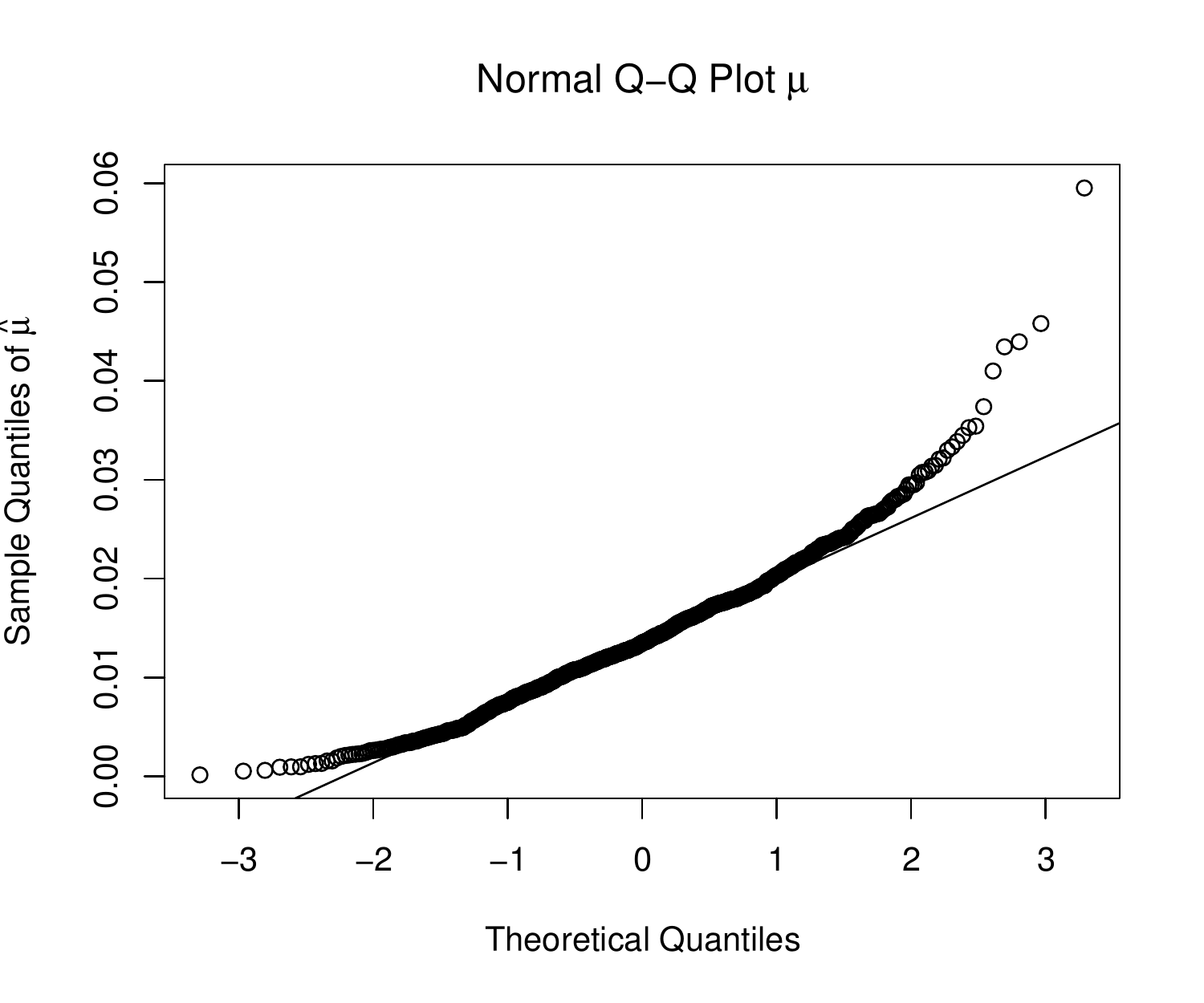}
\includegraphics[height=0.49\textwidth,angle=270]{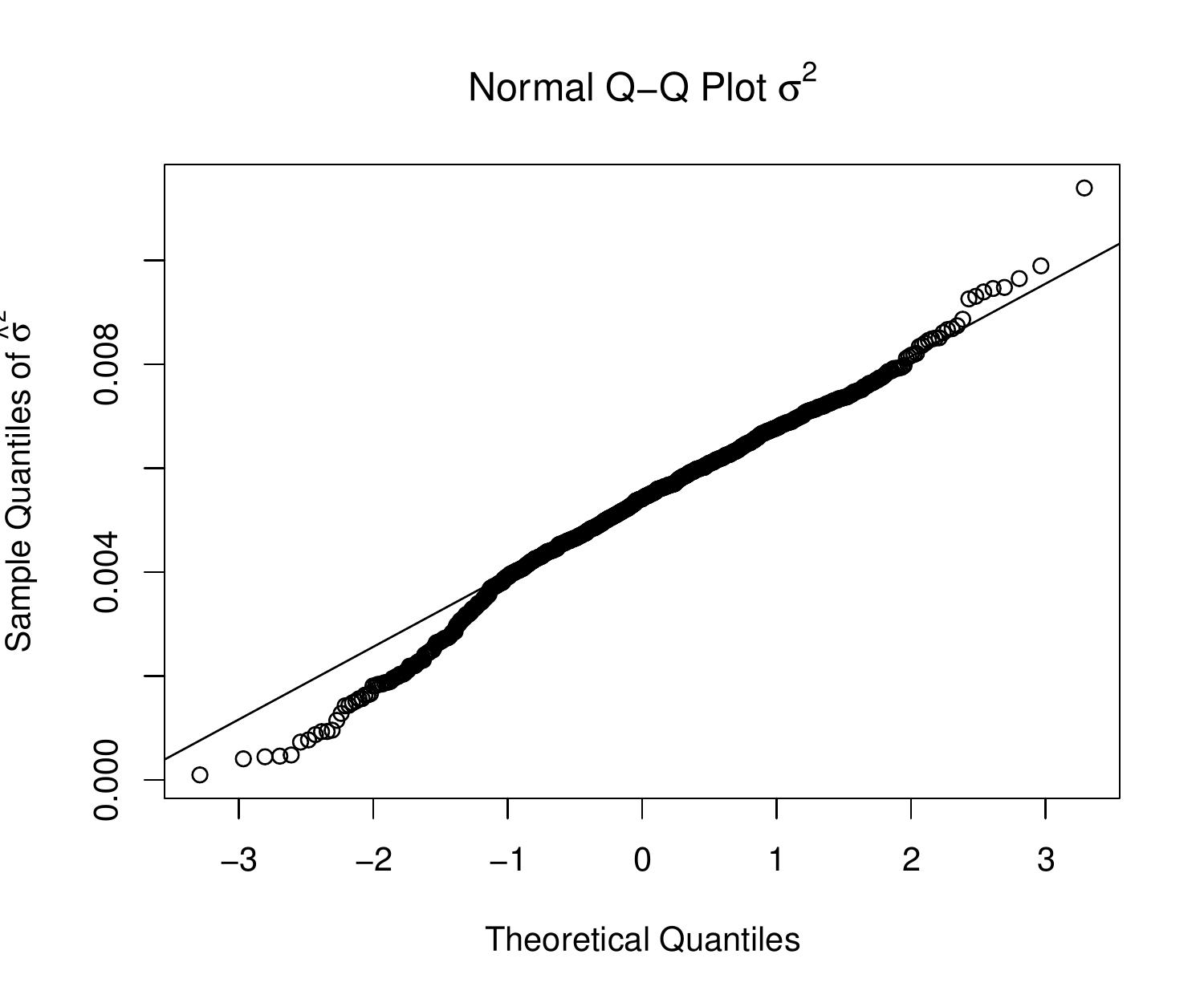}\\
\includegraphics[height=0.49\textwidth,angle=270]{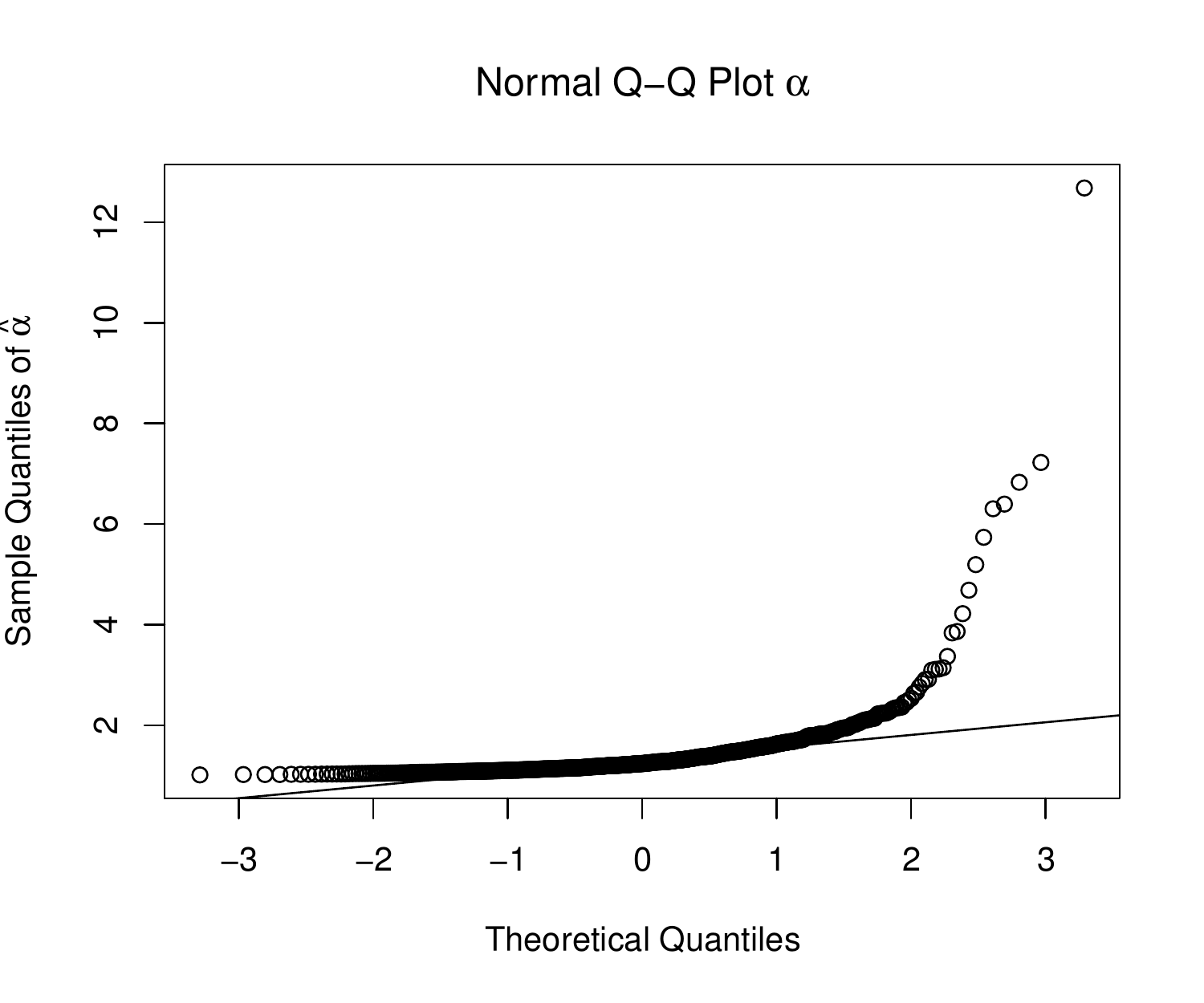}
\includegraphics[height=0.49\textwidth,angle=270]{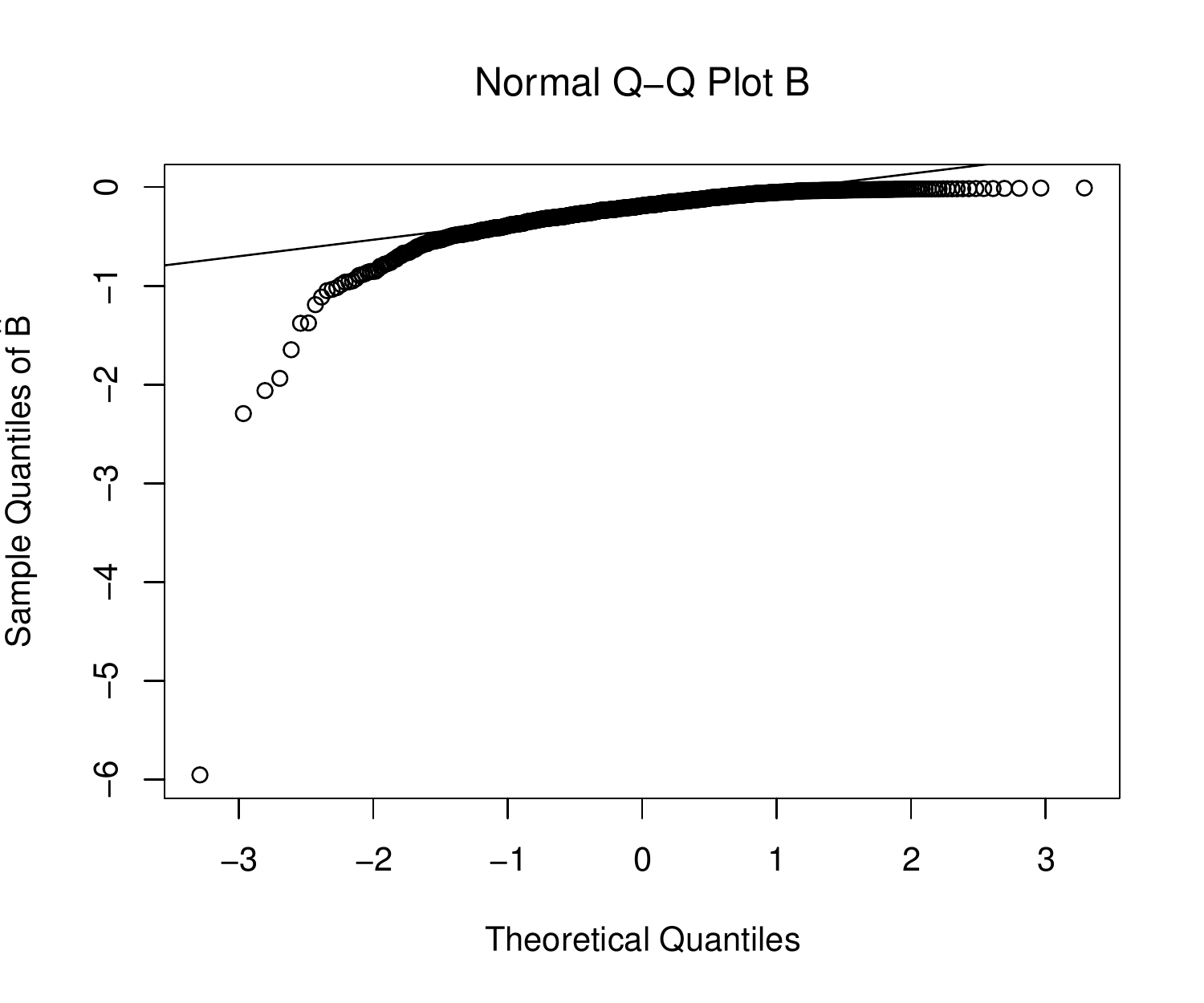}
\caption{Normal QQ-Plots of parameter estimates of 1000 paths of length 1000  of a supOU process with short (upper set of plots) and with long memory (lower set of plots).}
\label{figsupOU1Kqq}
\end{figure}

\begin{figure}[p]
\center
\includegraphics[height=0.49\textwidth,angle=270]{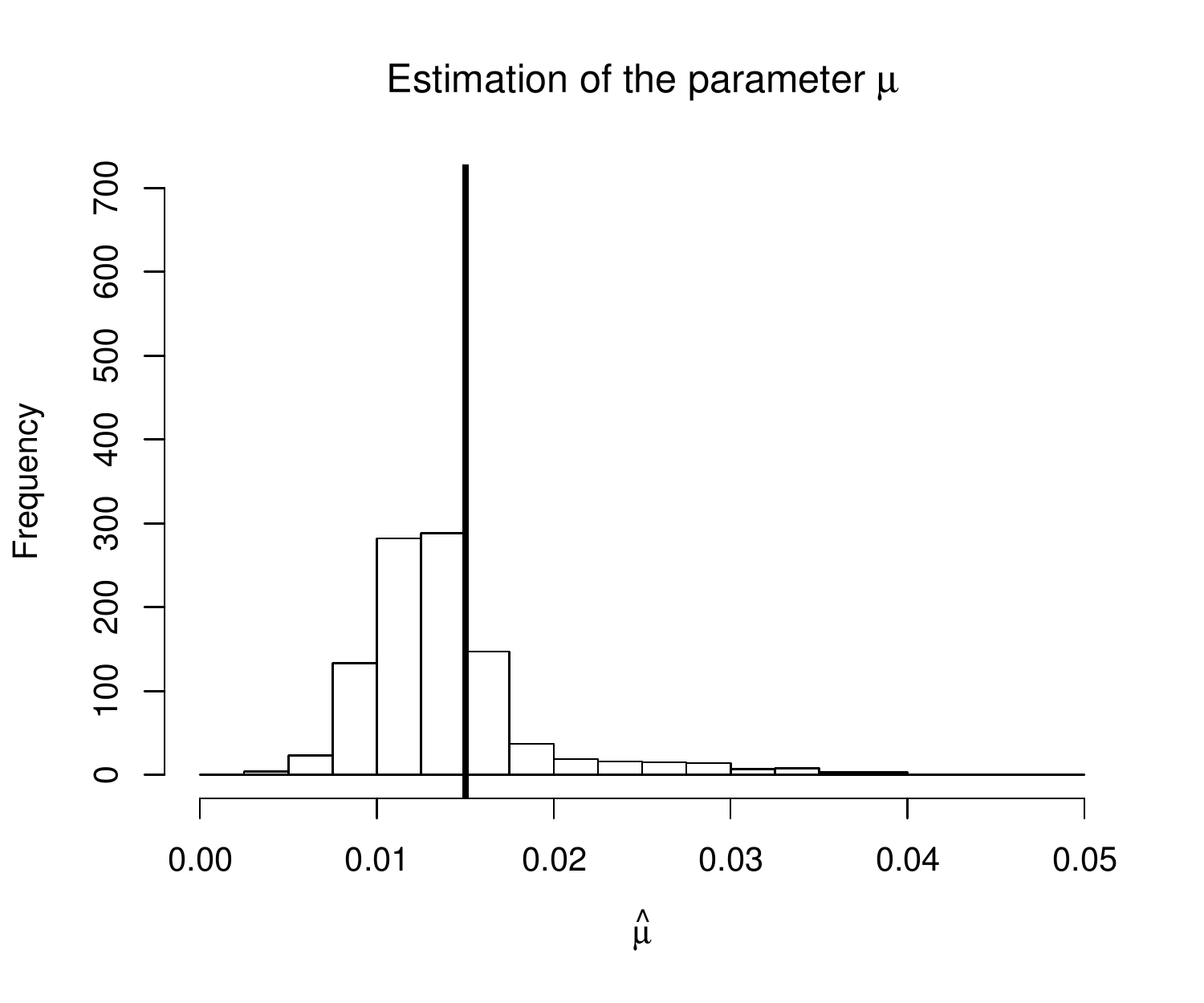}
\includegraphics[height=0.49\textwidth,angle=270]{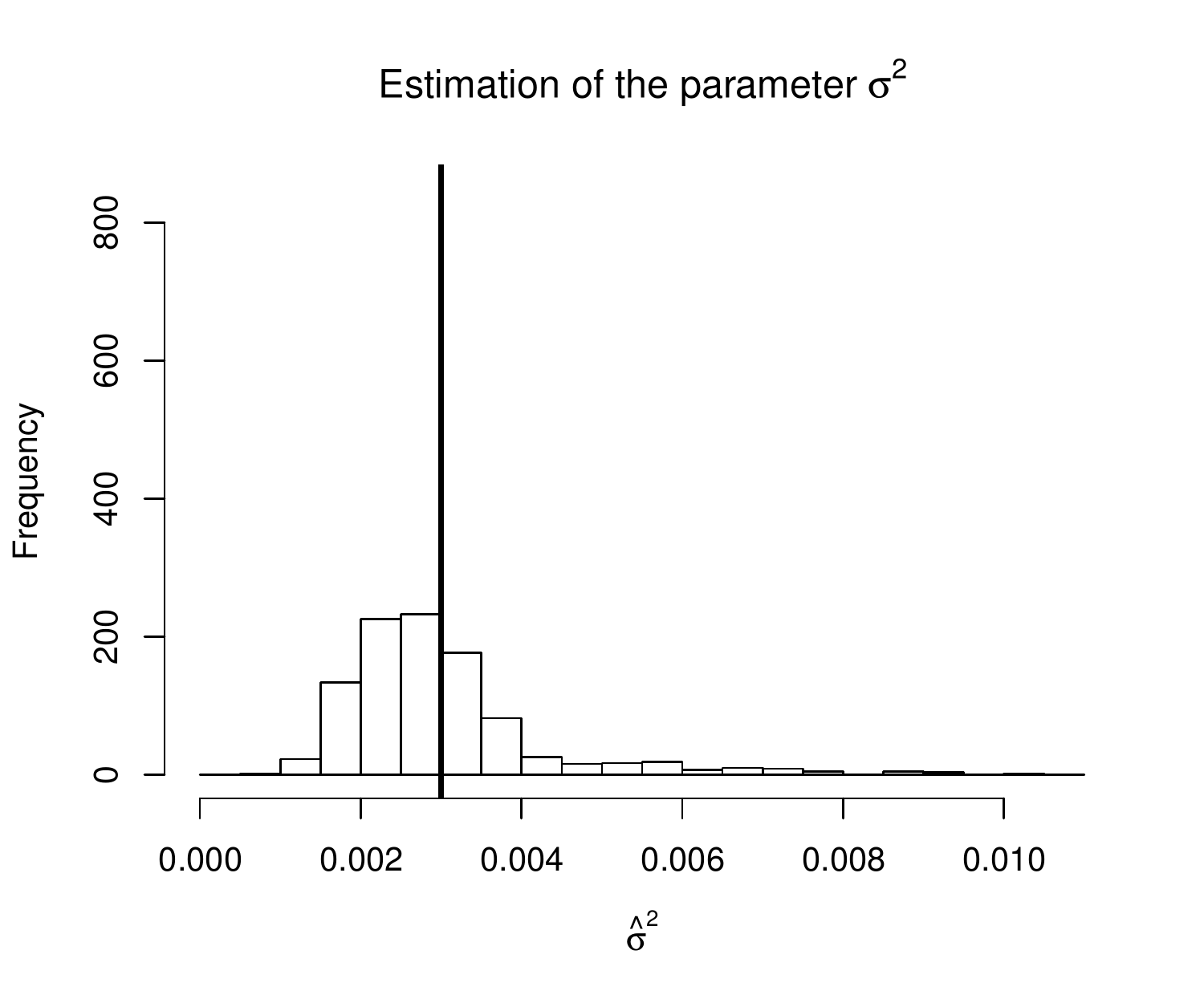}\\
\includegraphics[height=0.49\textwidth,angle=270]{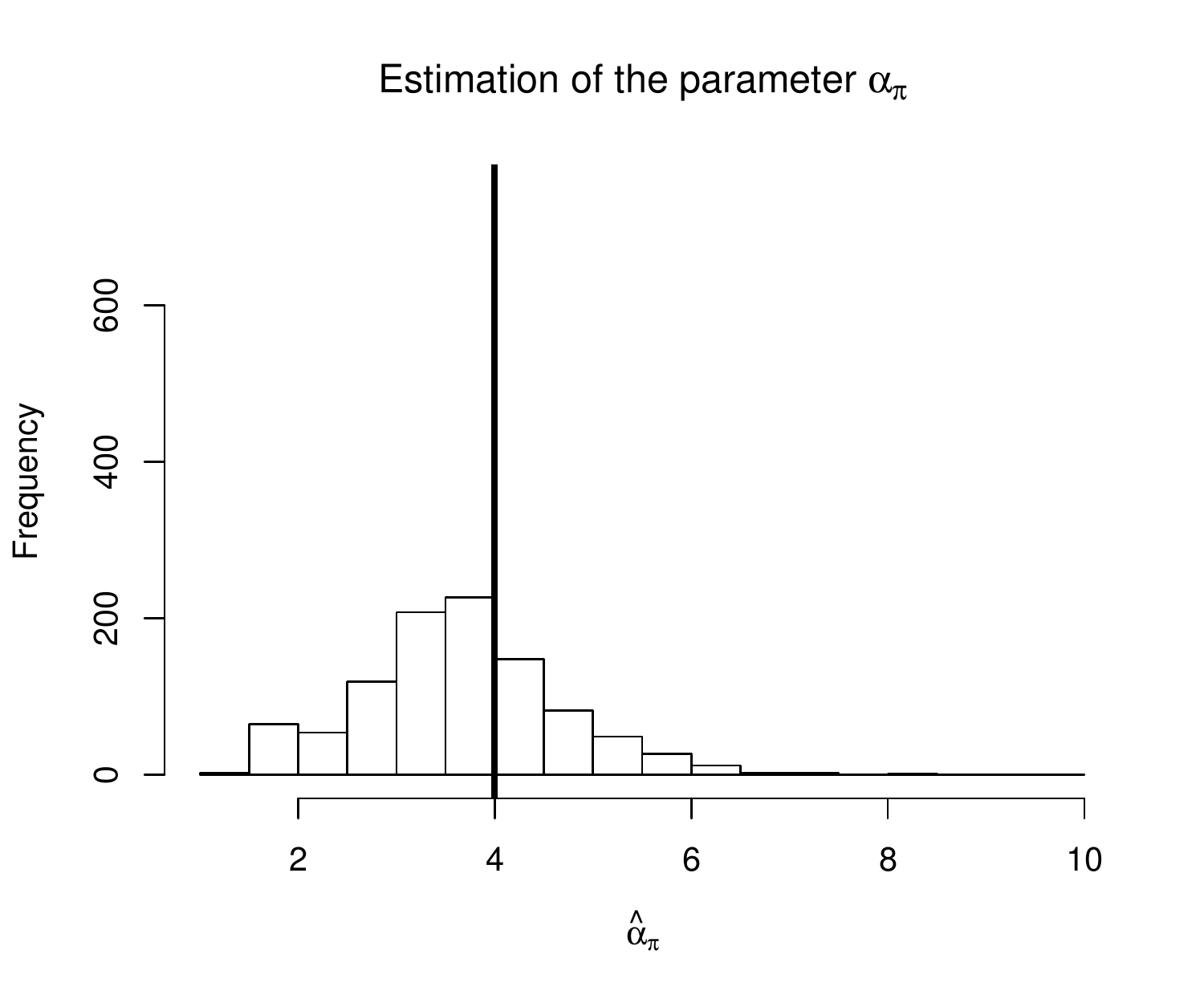}
\includegraphics[height=0.49\textwidth,angle=270]{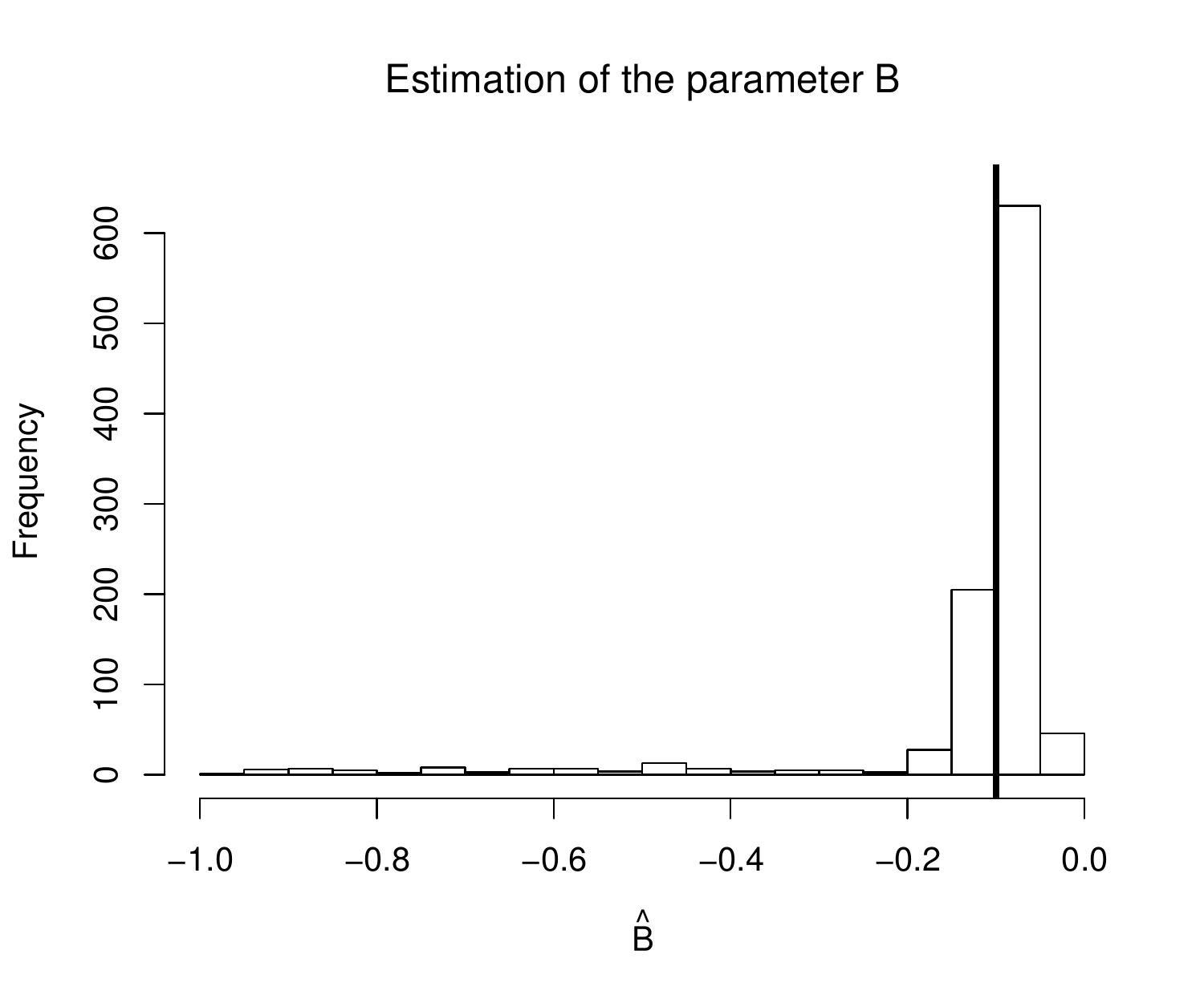}\\[5mm]
\includegraphics[height=0.49\textwidth,angle=270]{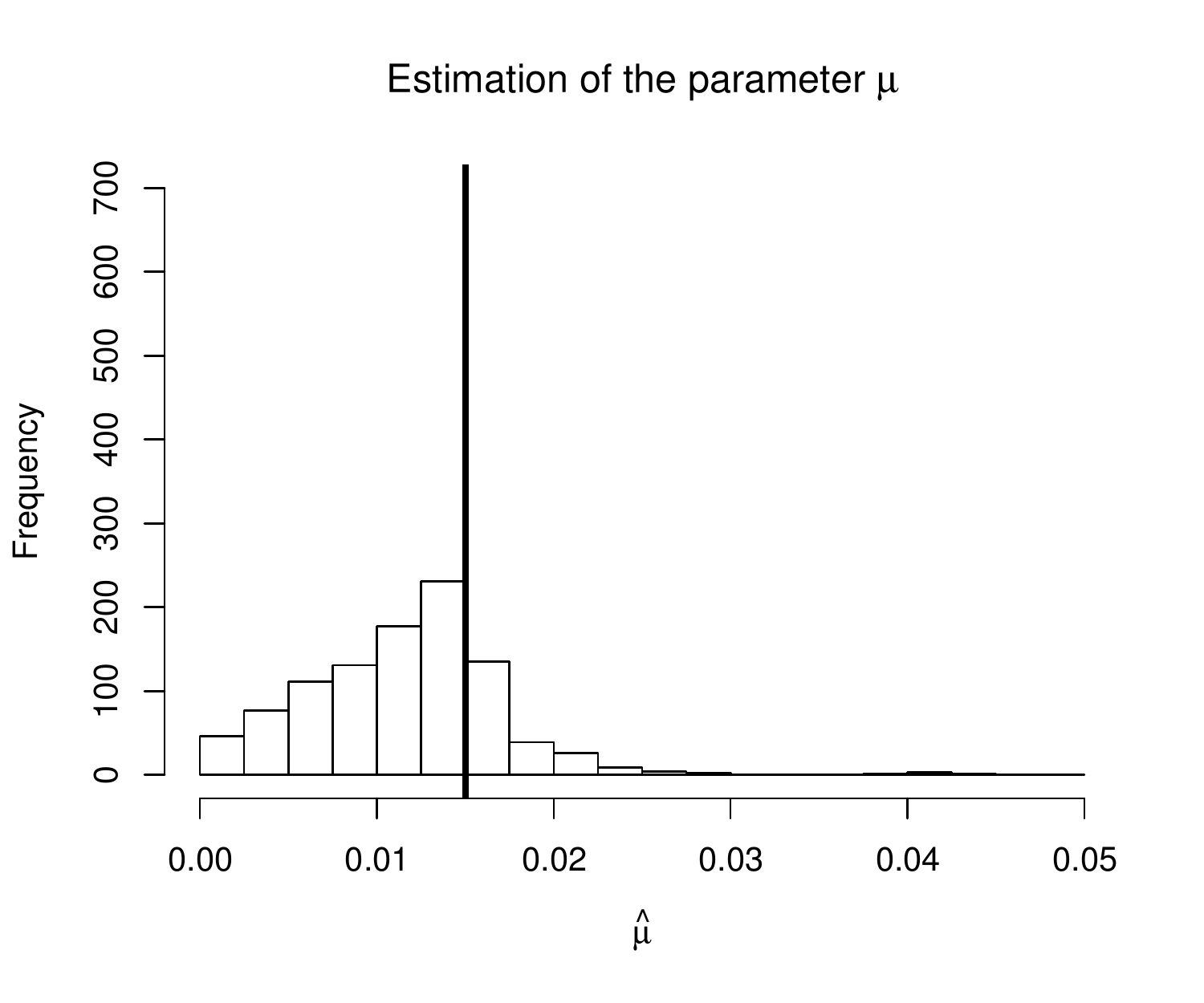}
\includegraphics[height=0.49\textwidth,angle=270]{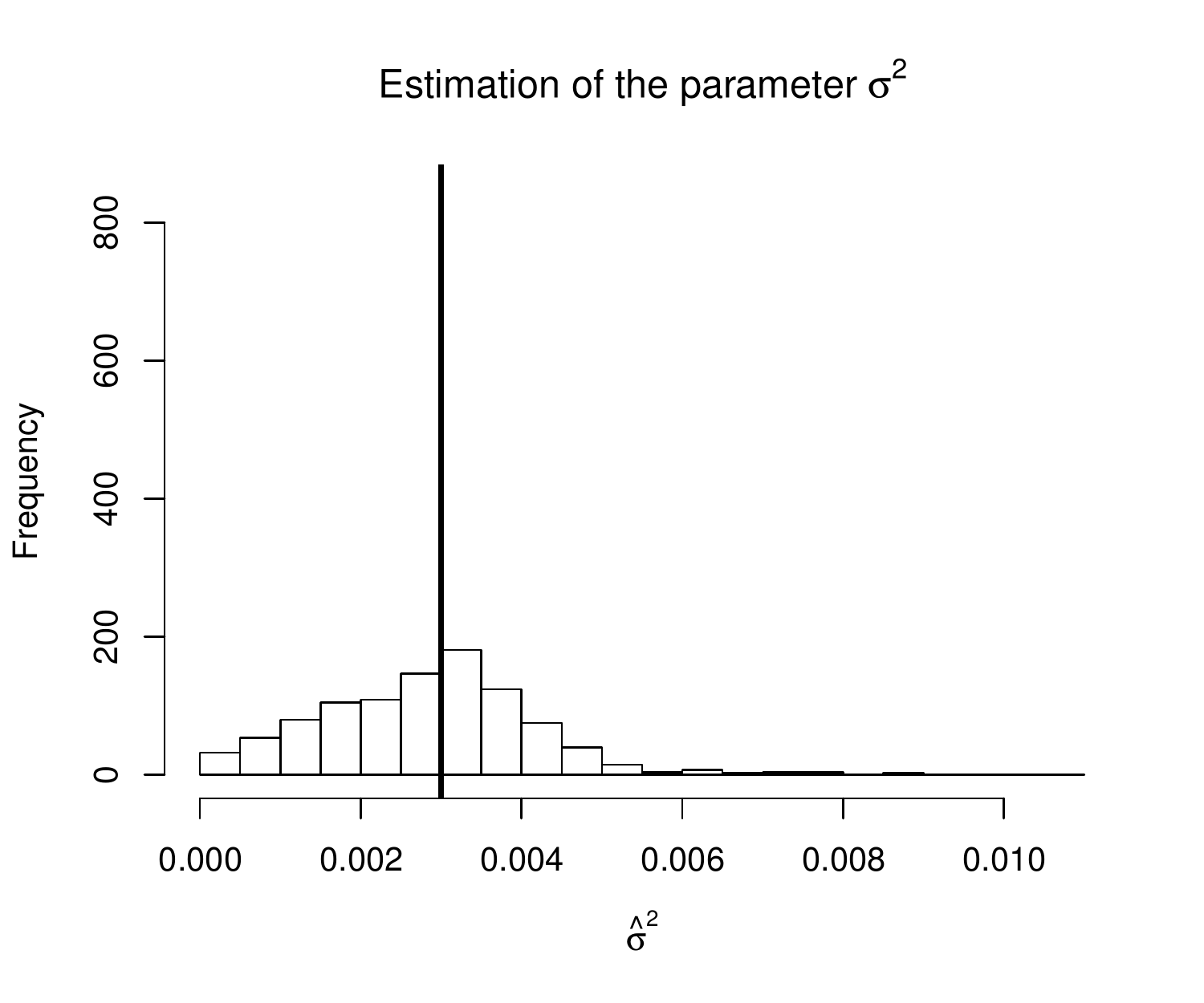}\\
\includegraphics[height=0.49\textwidth,angle=270]{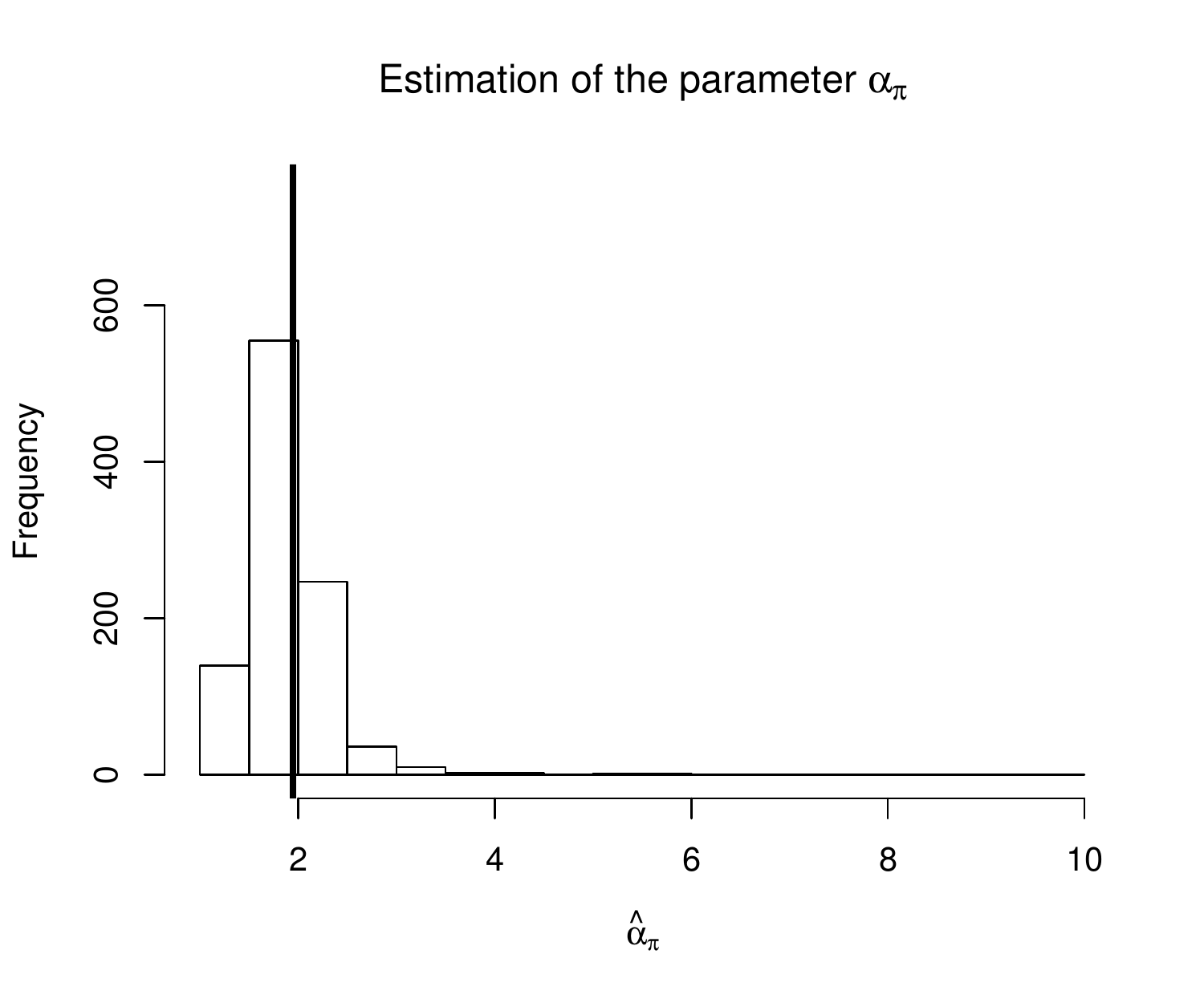}
\includegraphics[height=0.49\textwidth,angle=270]{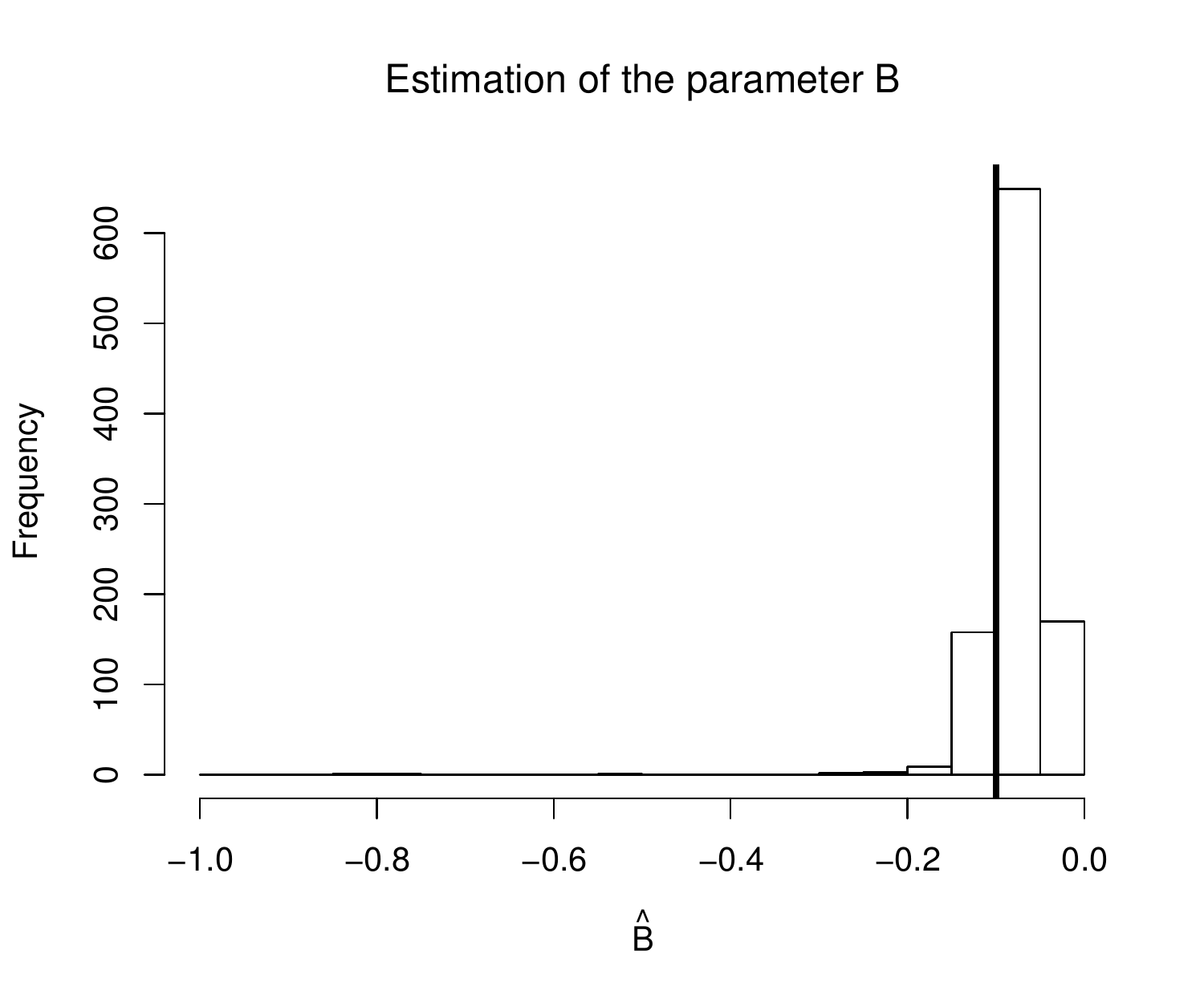}
\caption{Histograms of parameter estimates of 1000 paths of length 10000  of a supOU stochastic volatility model with short (upper set of plots) and with long memory (lower set of plots). The true values are indicated by black lines.}
\label{figSVsupOU10K}
\end{figure}

\begin{figure}[p]
\center
\includegraphics[height=0.49\textwidth,angle=270]{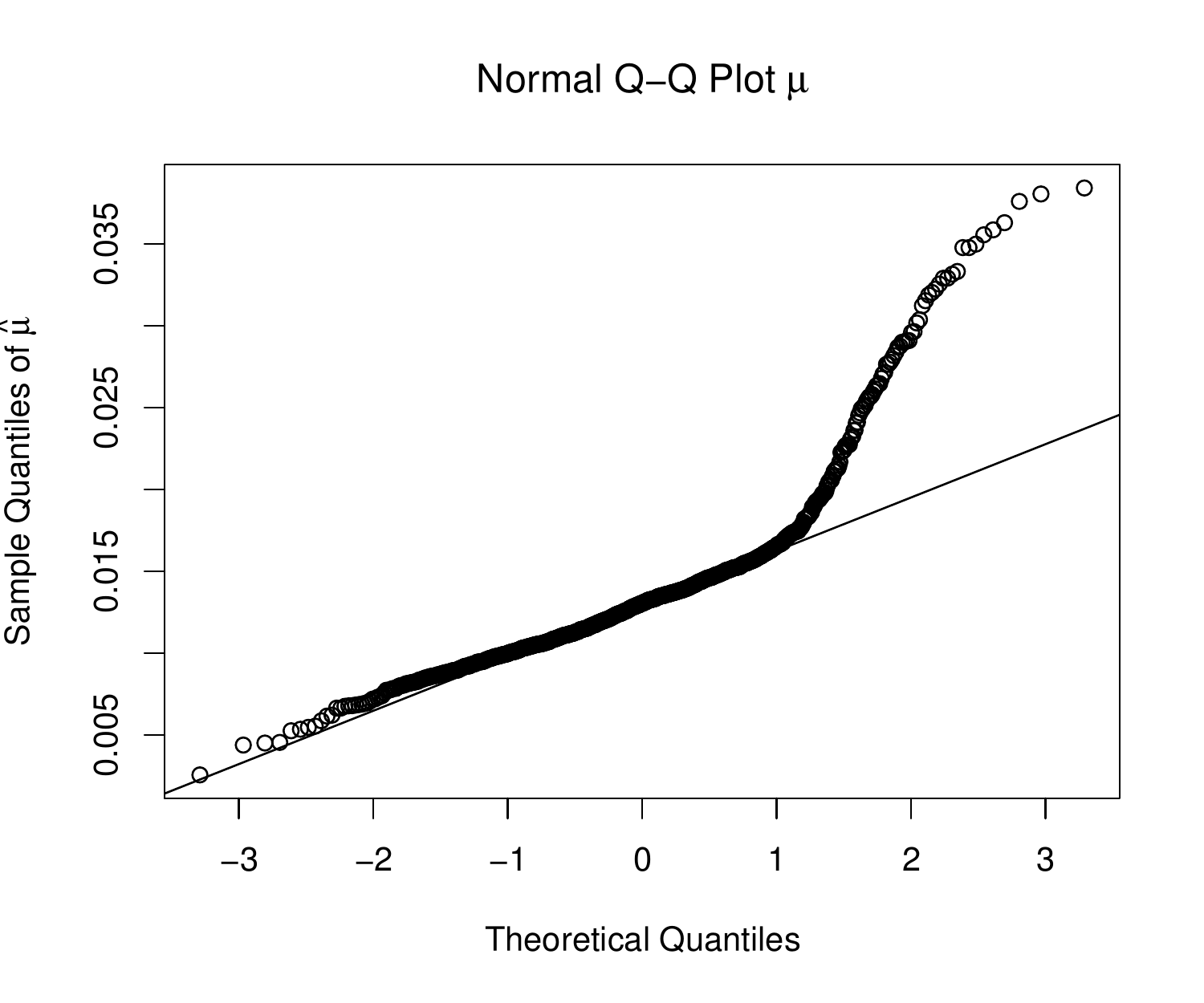}
\includegraphics[height=0.49\textwidth,angle=270]{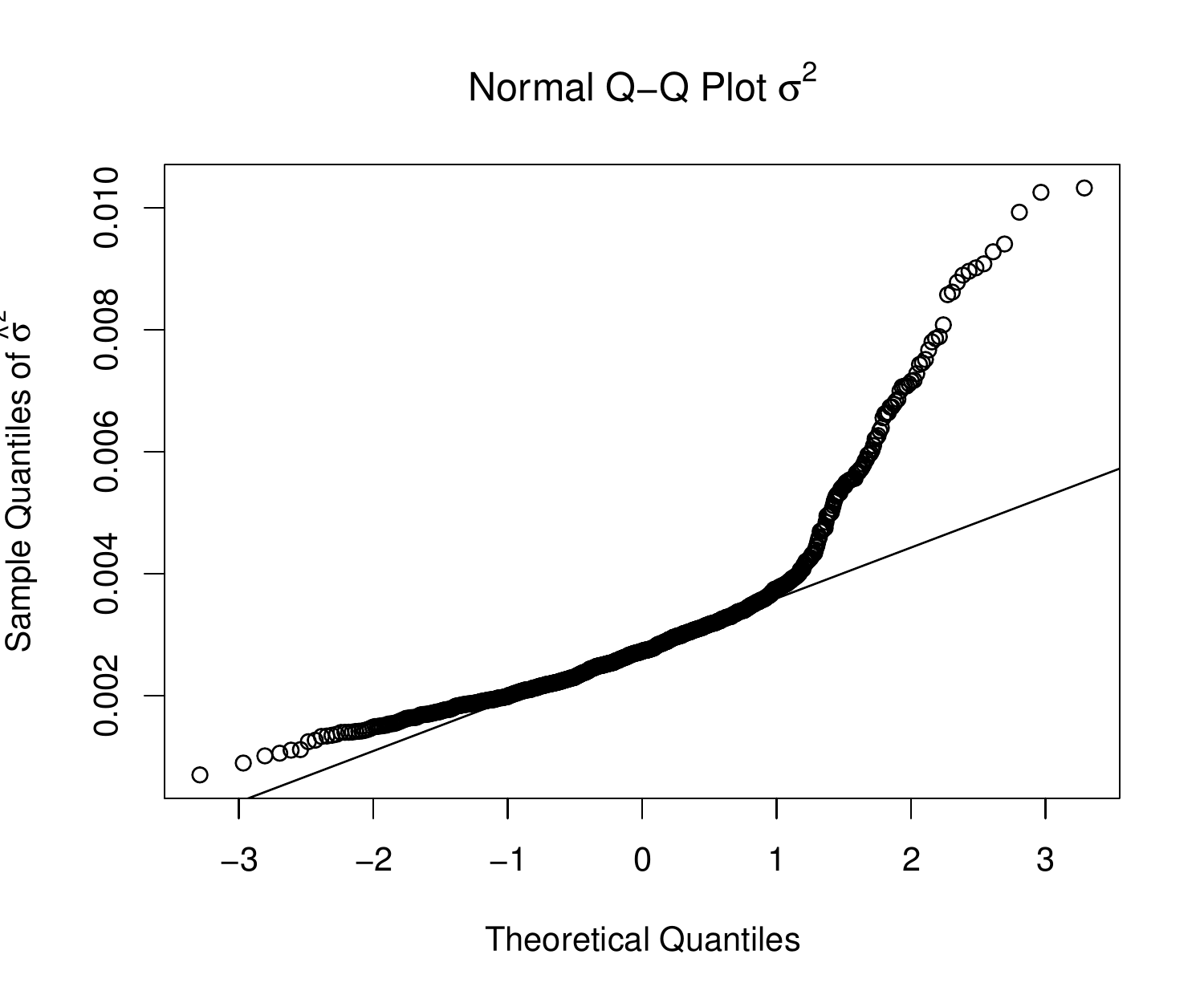}\\
\includegraphics[height=0.49\textwidth,angle=270]{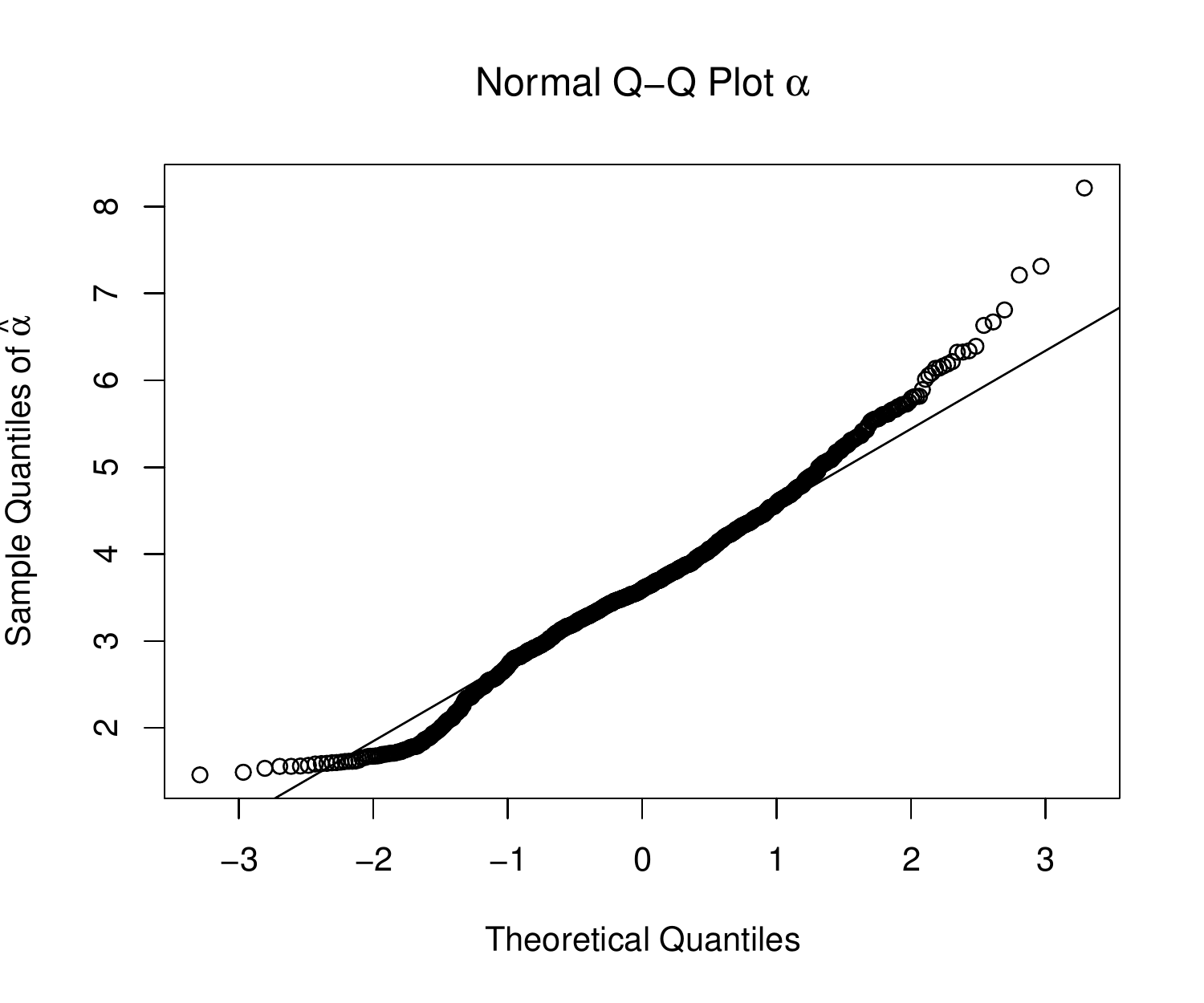}
\includegraphics[height=0.49\textwidth,angle=270]{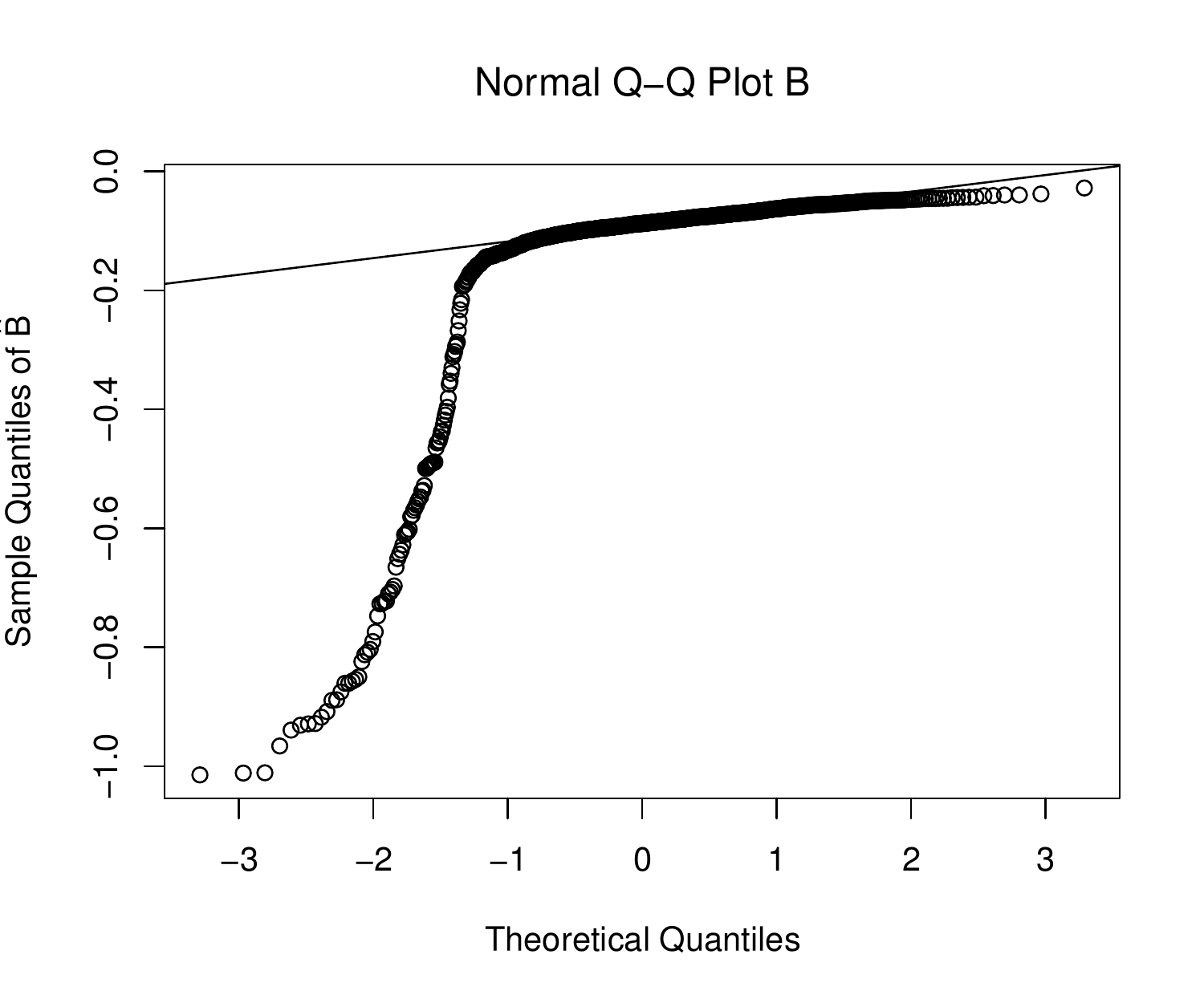}\\[5mm]
\includegraphics[height=0.49\textwidth,angle=270]{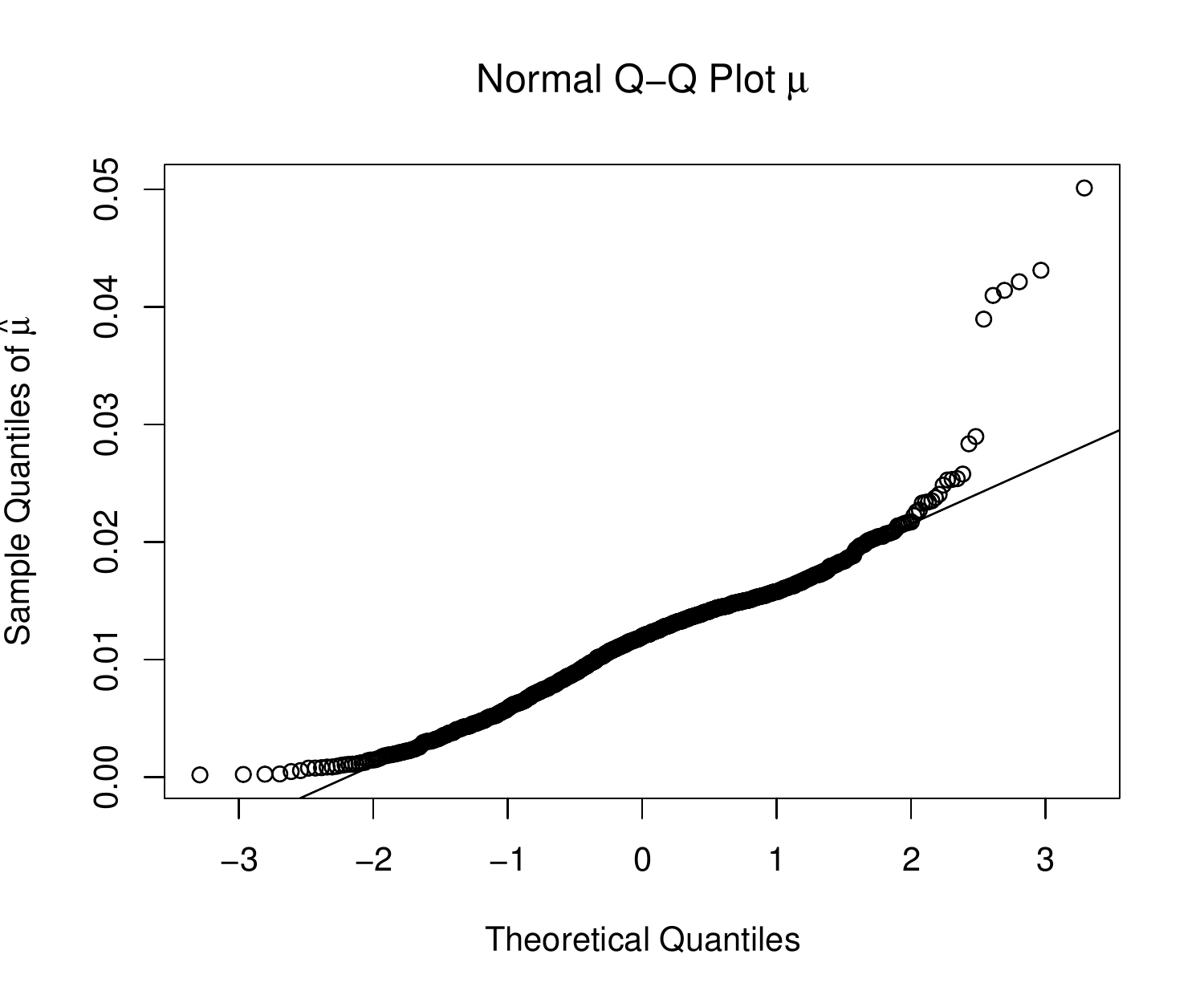}
\includegraphics[height=0.49\textwidth,angle=270]{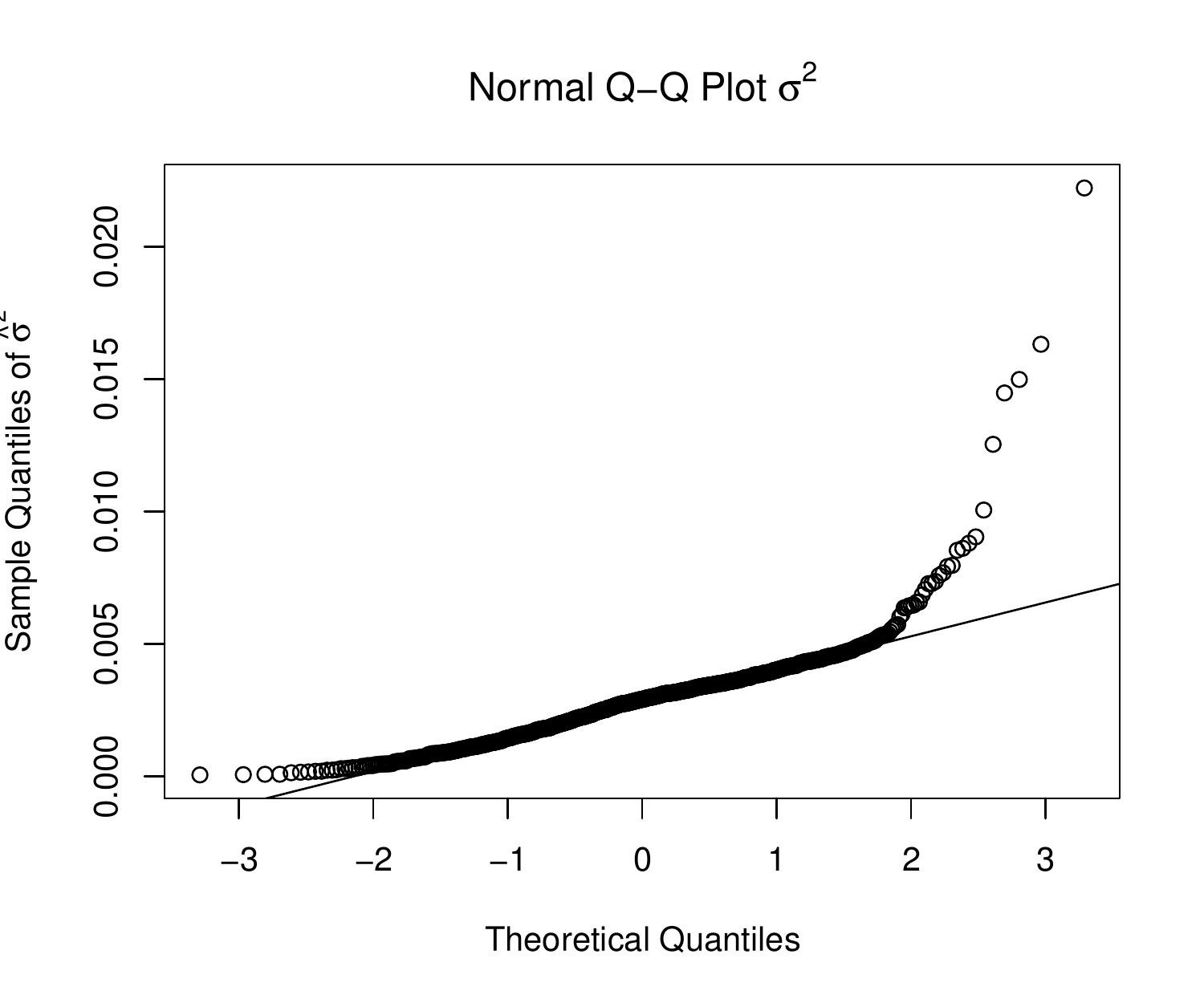}\\
\includegraphics[height=0.49\textwidth,angle=270]{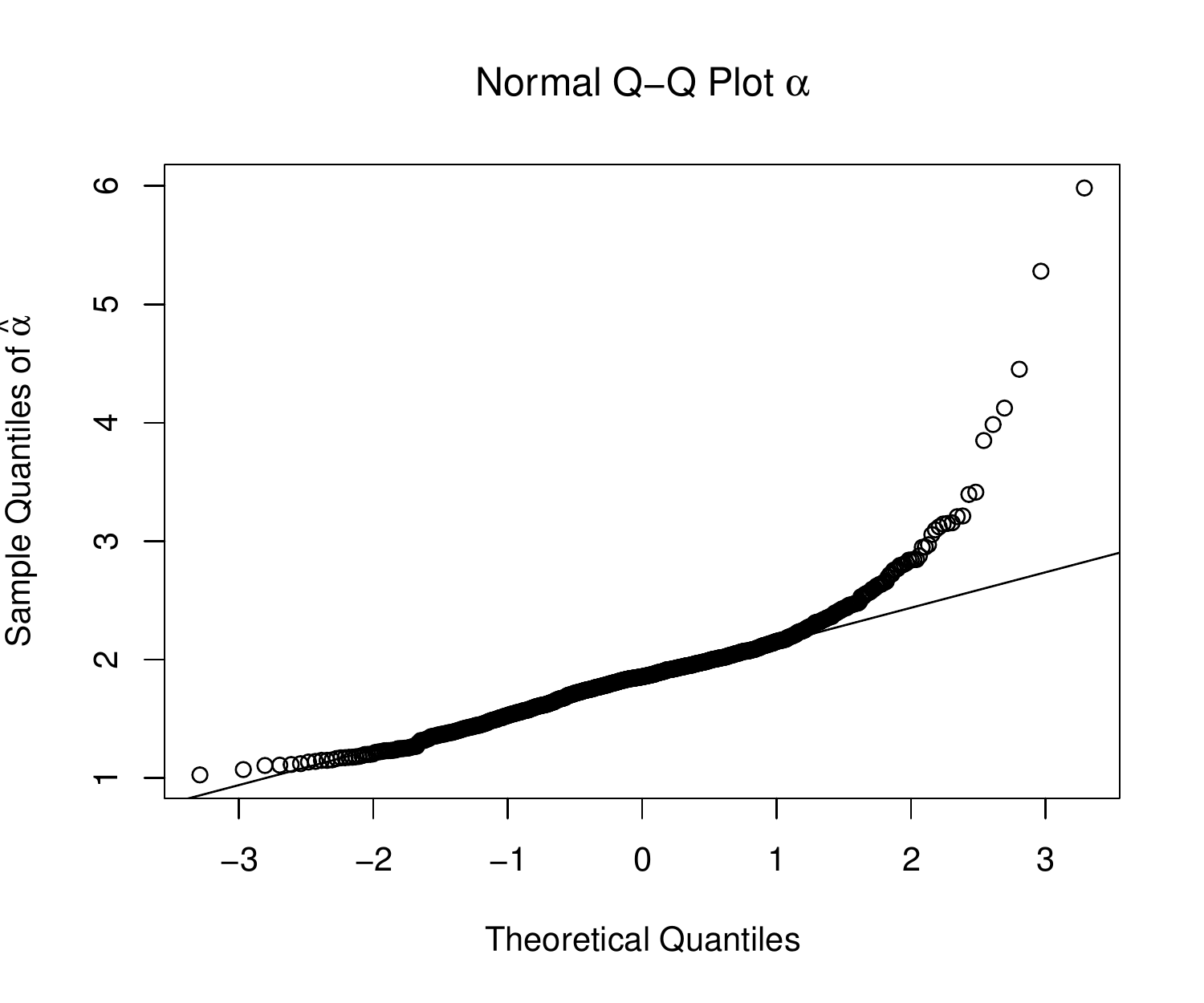}
\includegraphics[height=0.49\textwidth,angle=270]{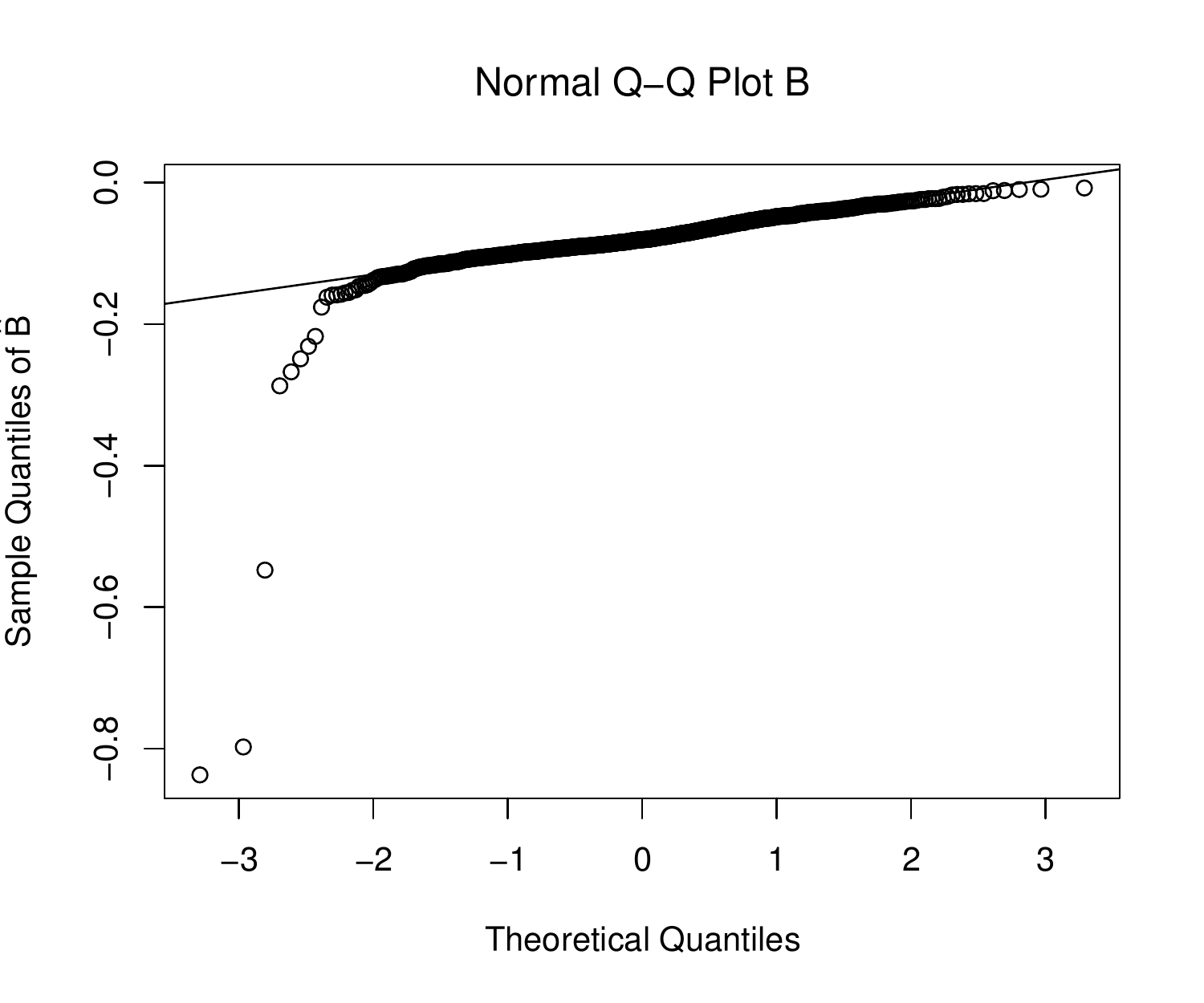}
\caption{Normal QQ-Plots of parameter estimates of 1000 paths of length 10000  of a supOU stochastic volatility model with short (upper set of plots) and with long memory (lower set of plots).}
\label{figSVsupOU10Kqq}
\end{figure}
\begin{figure}[p]
\center
\includegraphics[height=0.49\textwidth,angle=270]{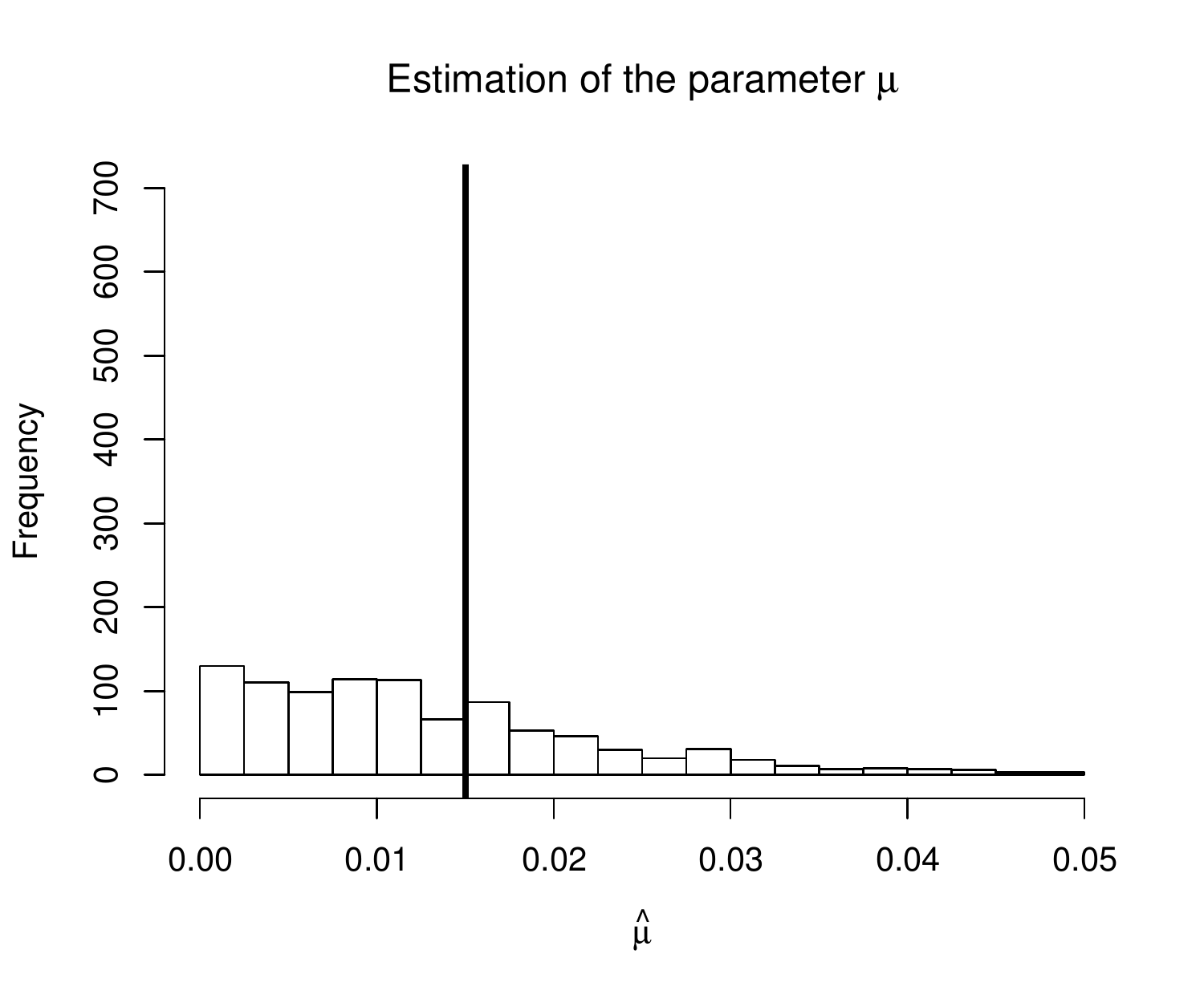}
\includegraphics[height=0.49\textwidth,angle=270]{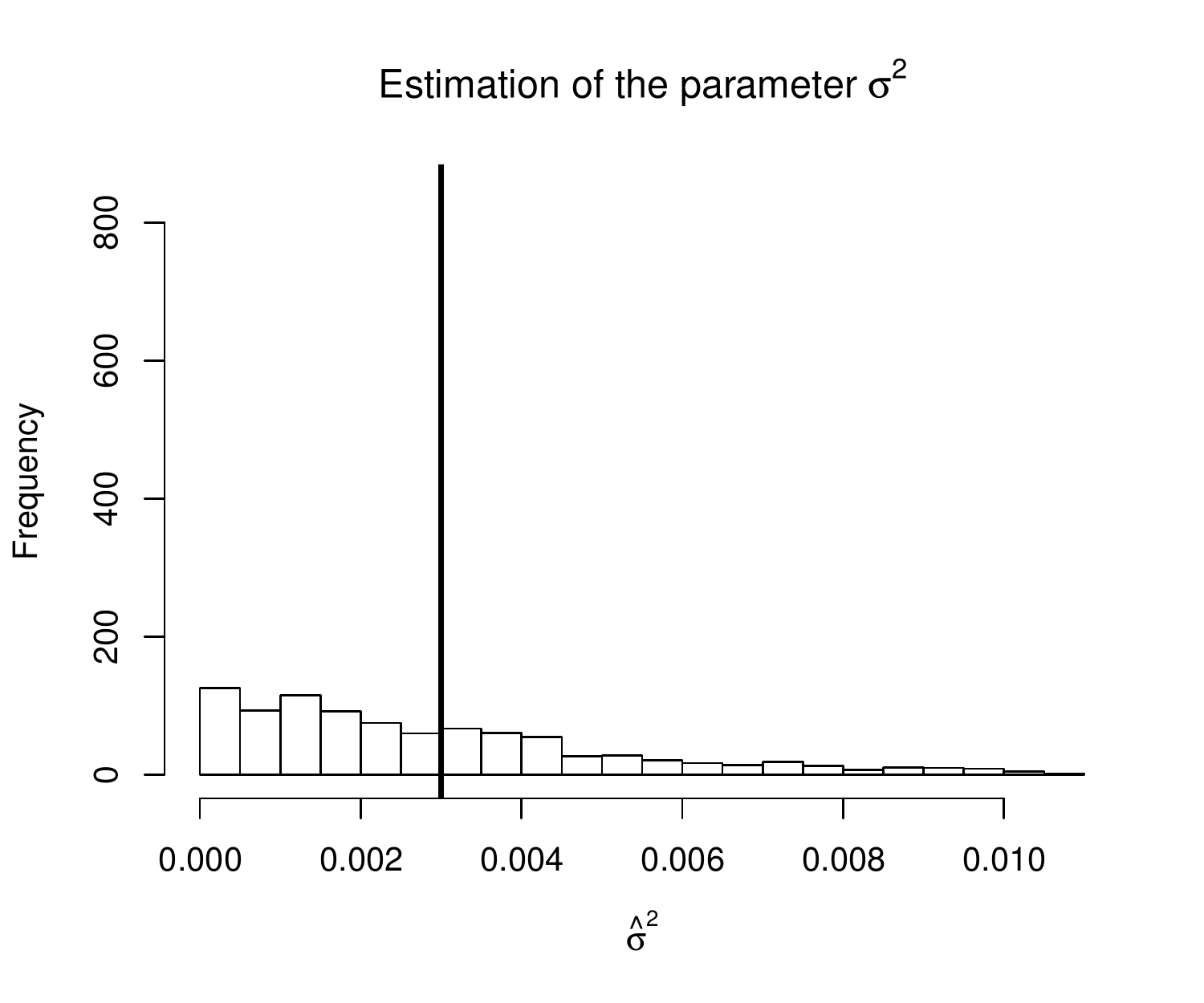}\\
\includegraphics[height=0.49\textwidth,angle=270]{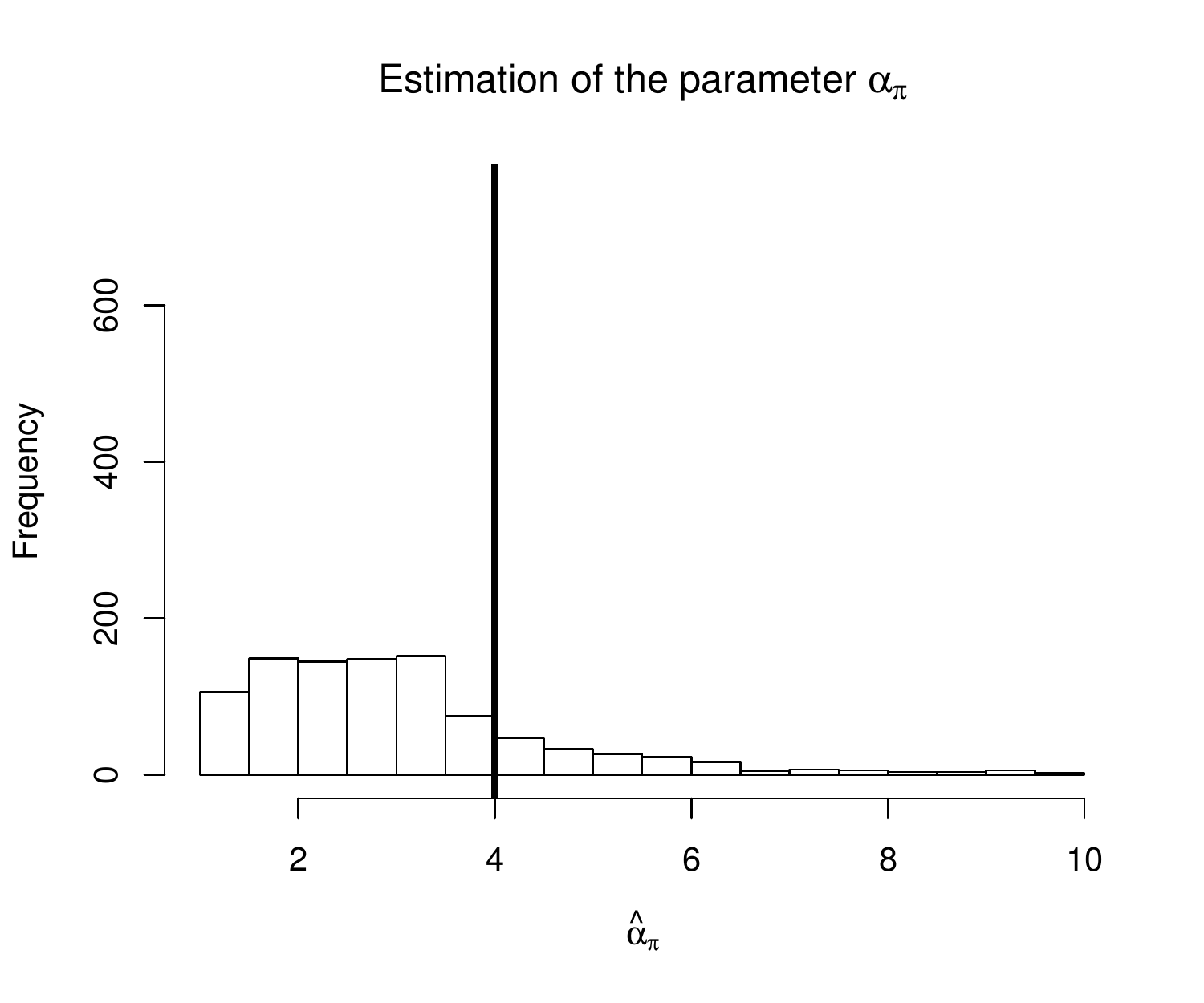}
\includegraphics[height=0.49\textwidth,angle=270]{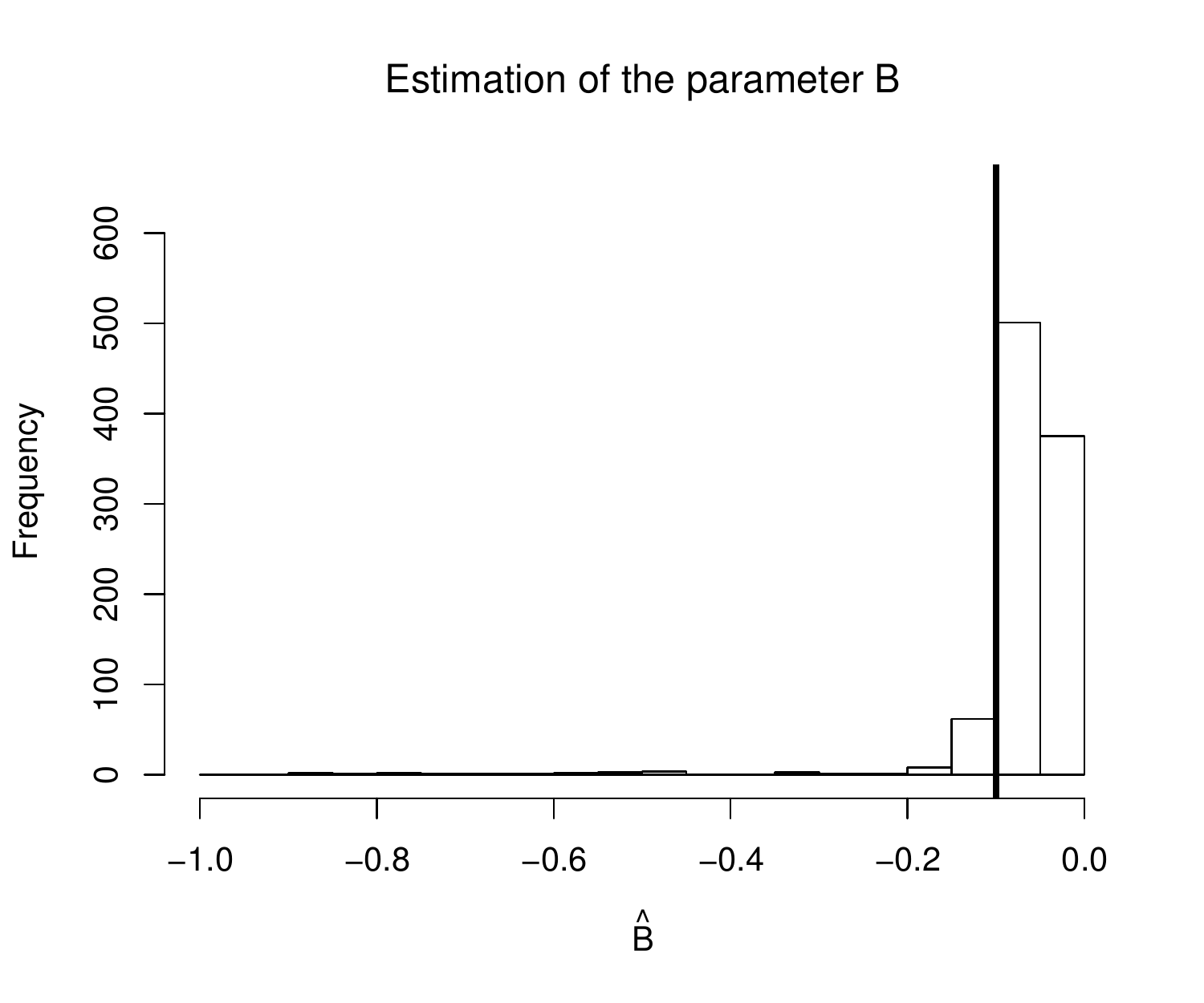}\\[5mm]
\includegraphics[height=0.49\textwidth,angle=270]{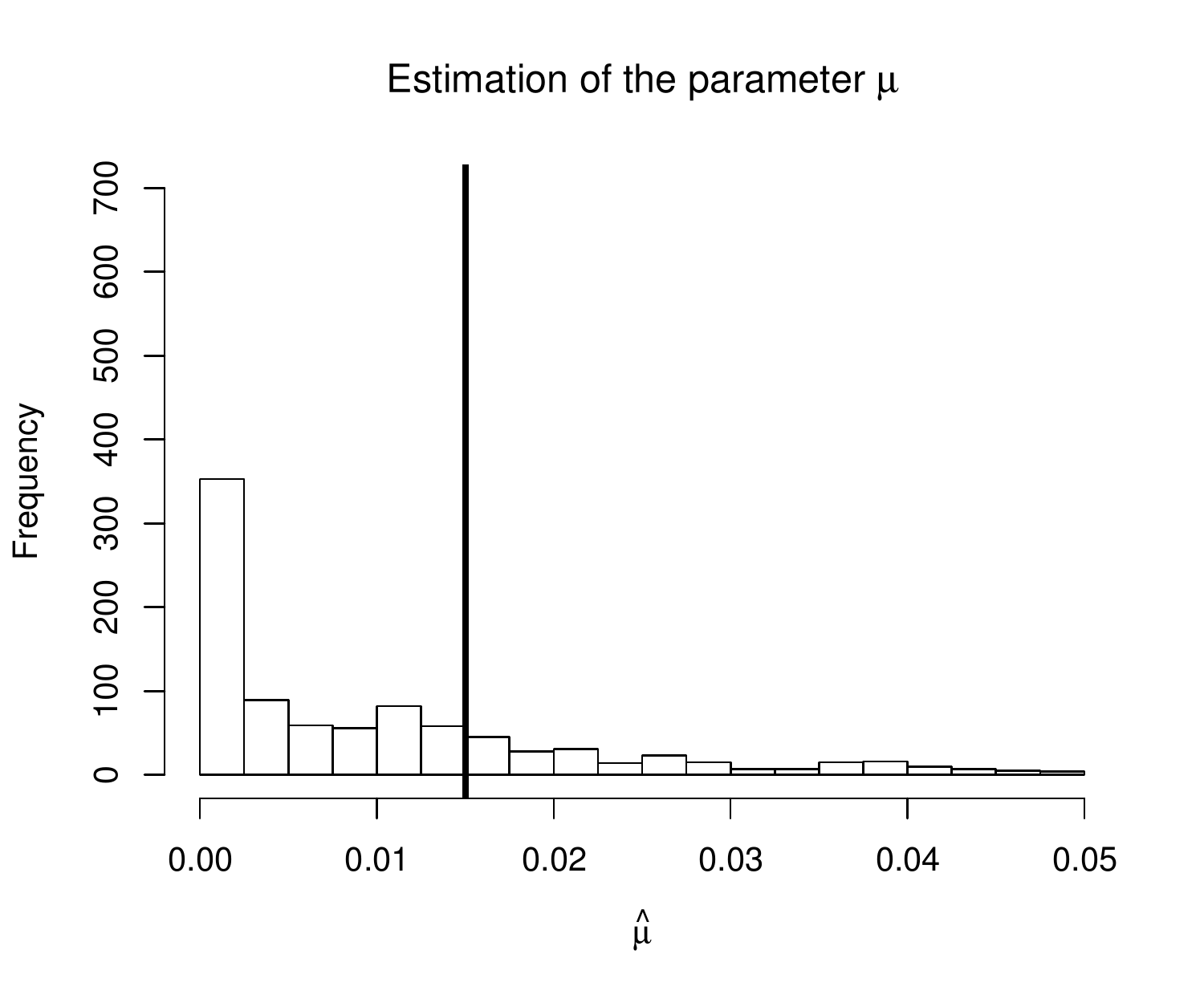}
\includegraphics[height=0.49\textwidth,angle=270]{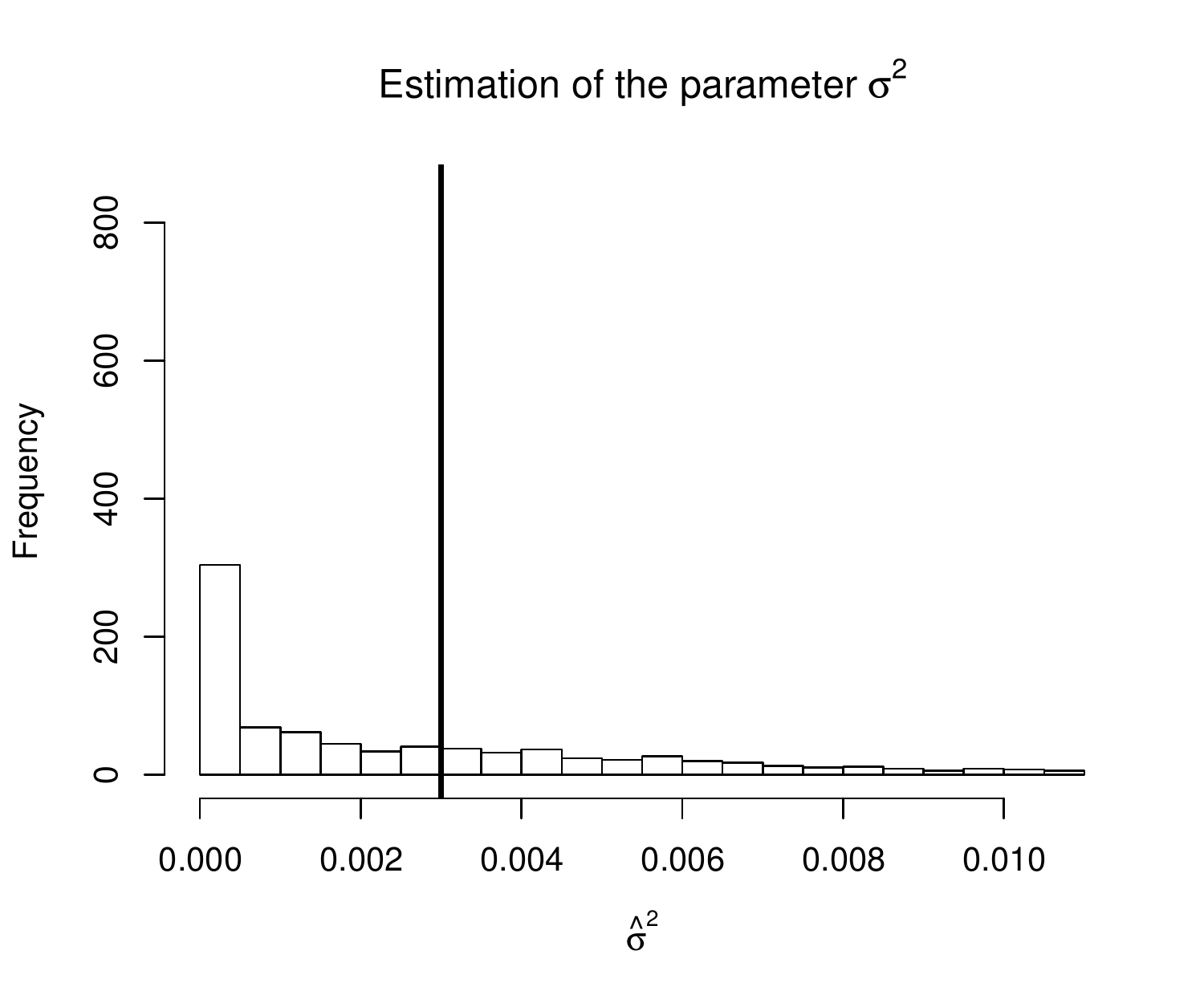}\\
\includegraphics[height=0.49\textwidth,angle=270]{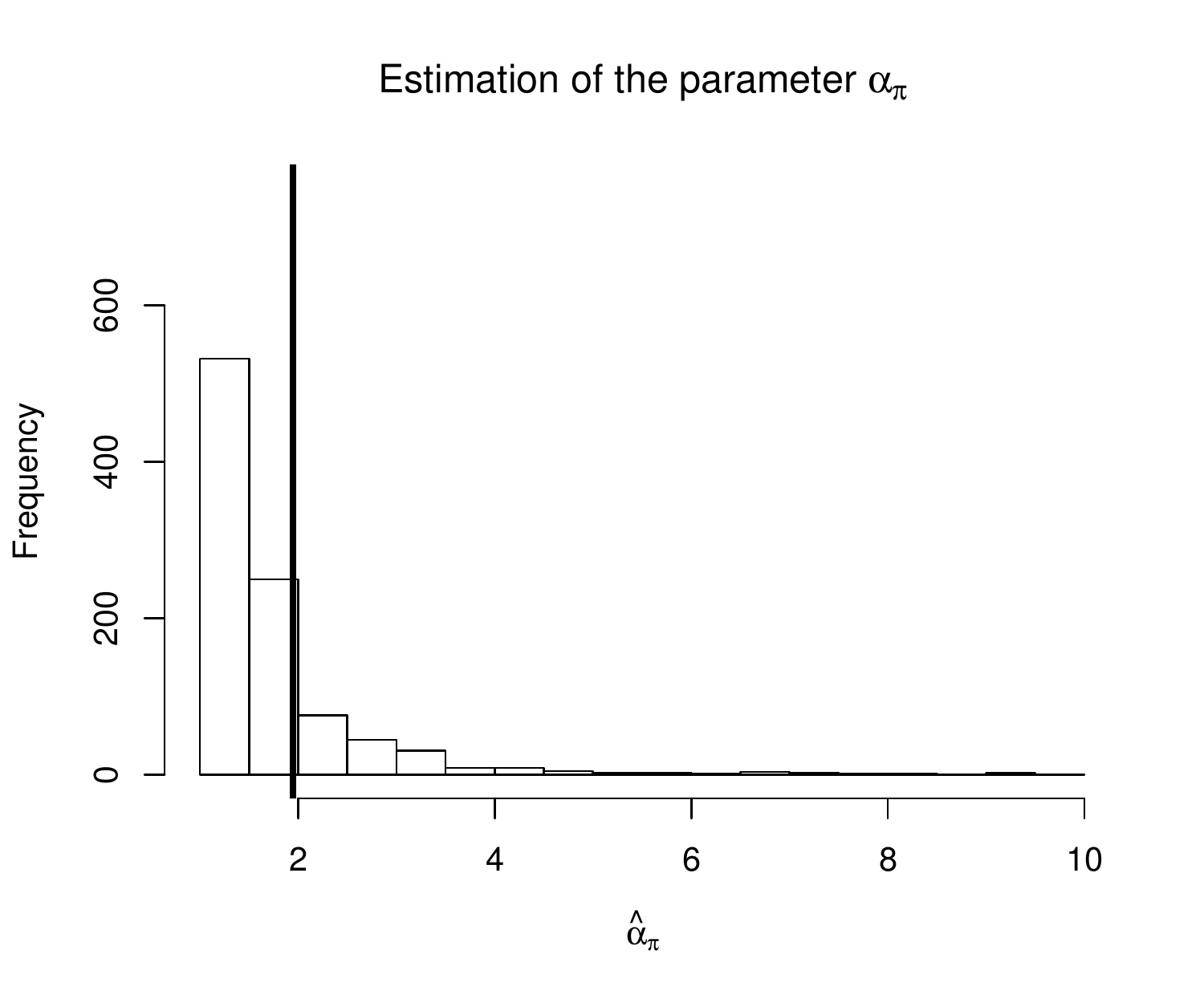}
\includegraphics[height=0.49\textwidth,angle=270]{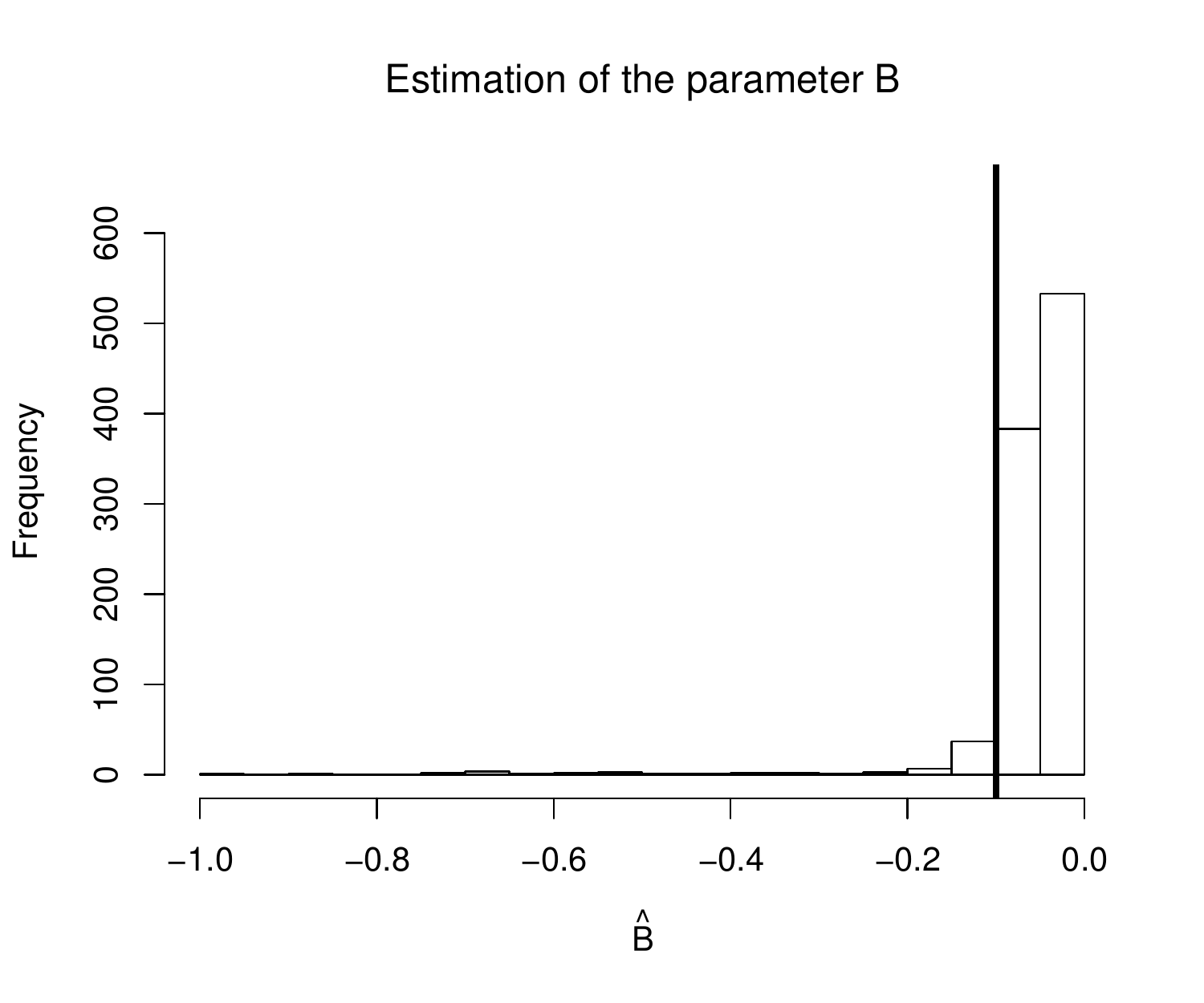}
\caption{Histograms of parameter estimates of 1000 paths of length 1000  of a supOU stochastic volatility model with short (upper set of plots) and with long memory (lower set of plots). The true values are indicated by black lines.}
\label{figSVsupOU1K}
\end{figure}

\begin{figure}[p]
\center
\includegraphics[height=0.49\textwidth,angle=270]{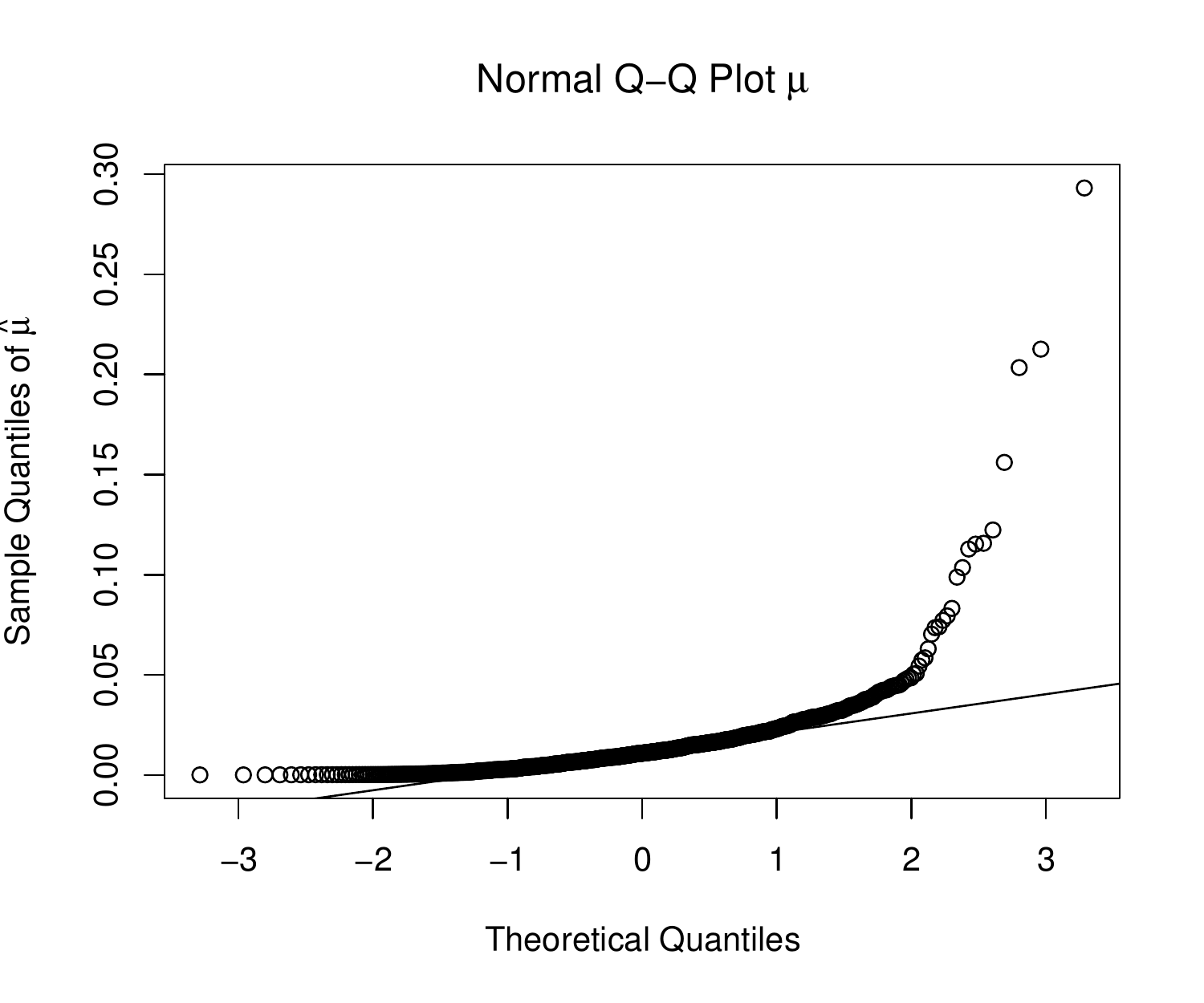}
\includegraphics[height=0.49\textwidth,angle=270]{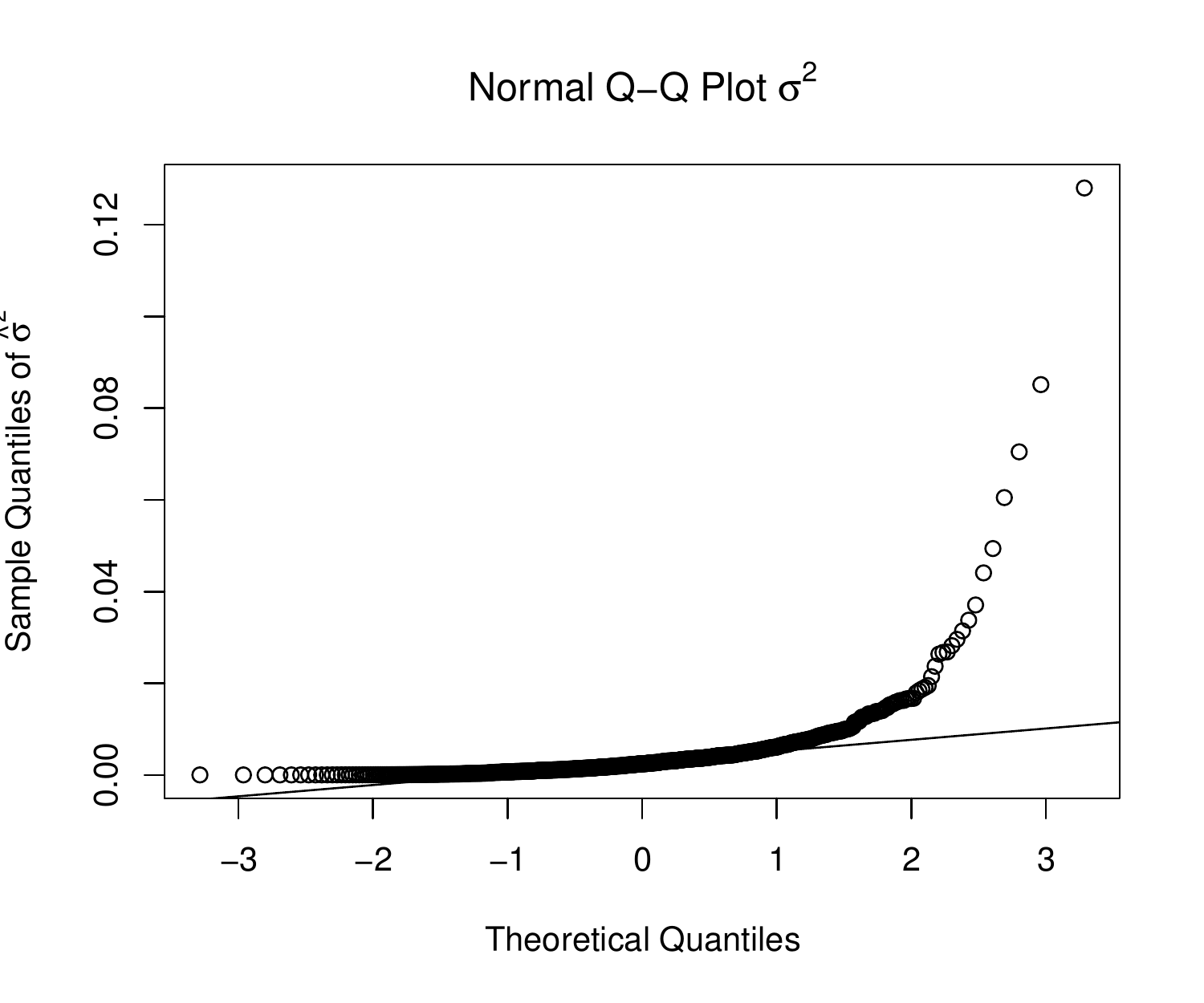}\\
\includegraphics[height=0.49\textwidth,angle=270]{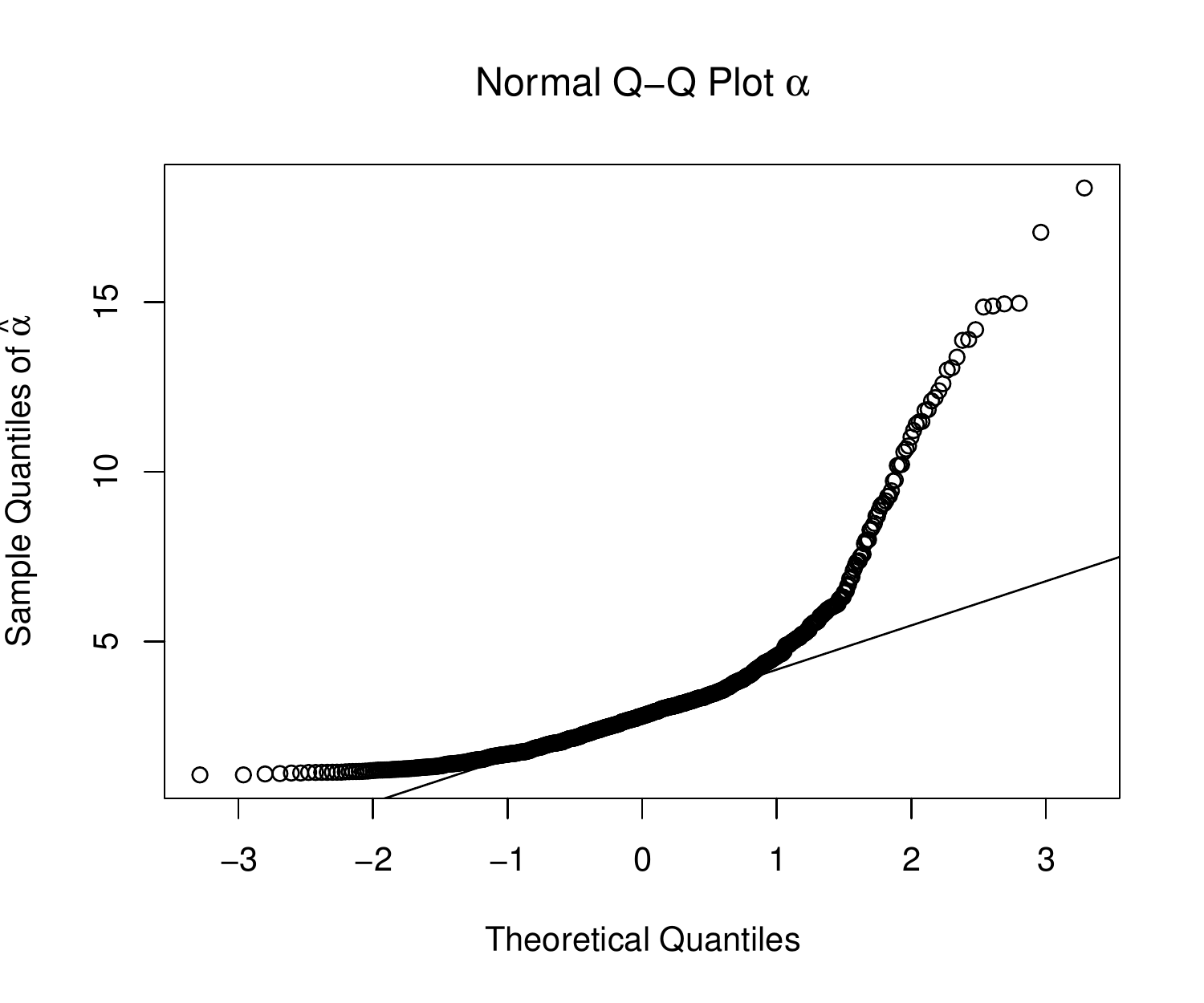}
\includegraphics[height=0.49\textwidth,angle=270]{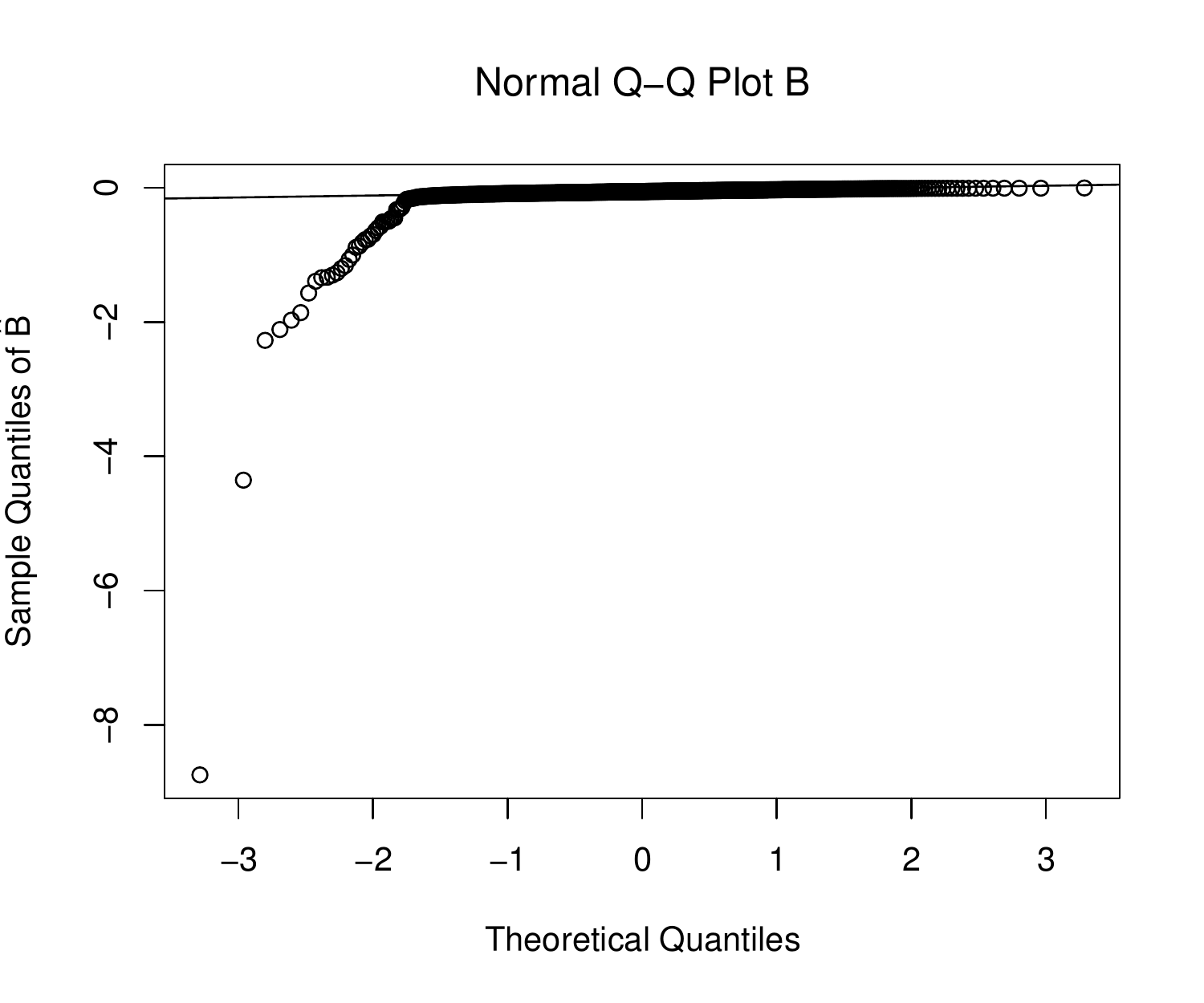}\\[5mm]
\includegraphics[height=0.49\textwidth,angle=270]{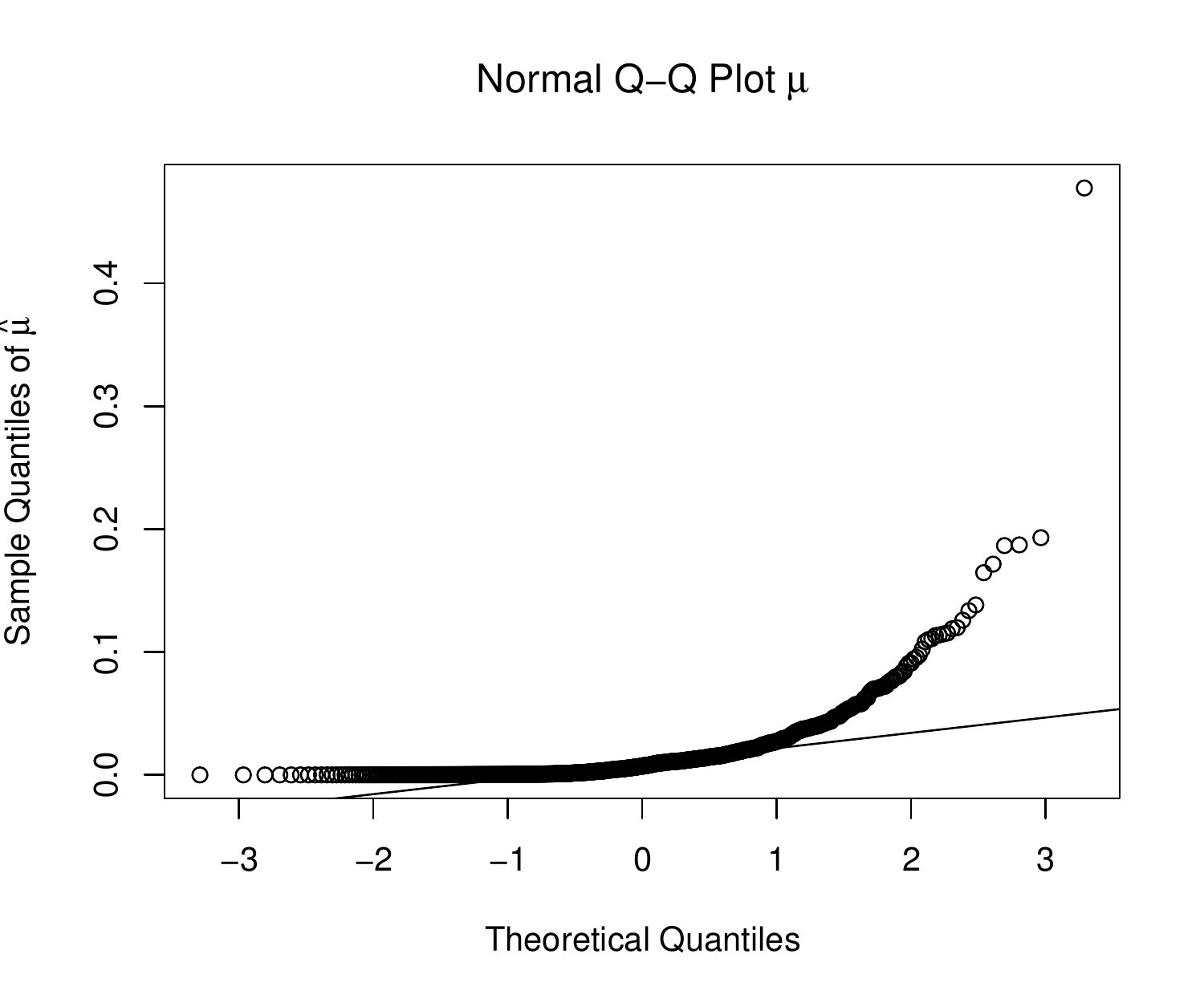}
\includegraphics[height=0.49\textwidth,angle=270]{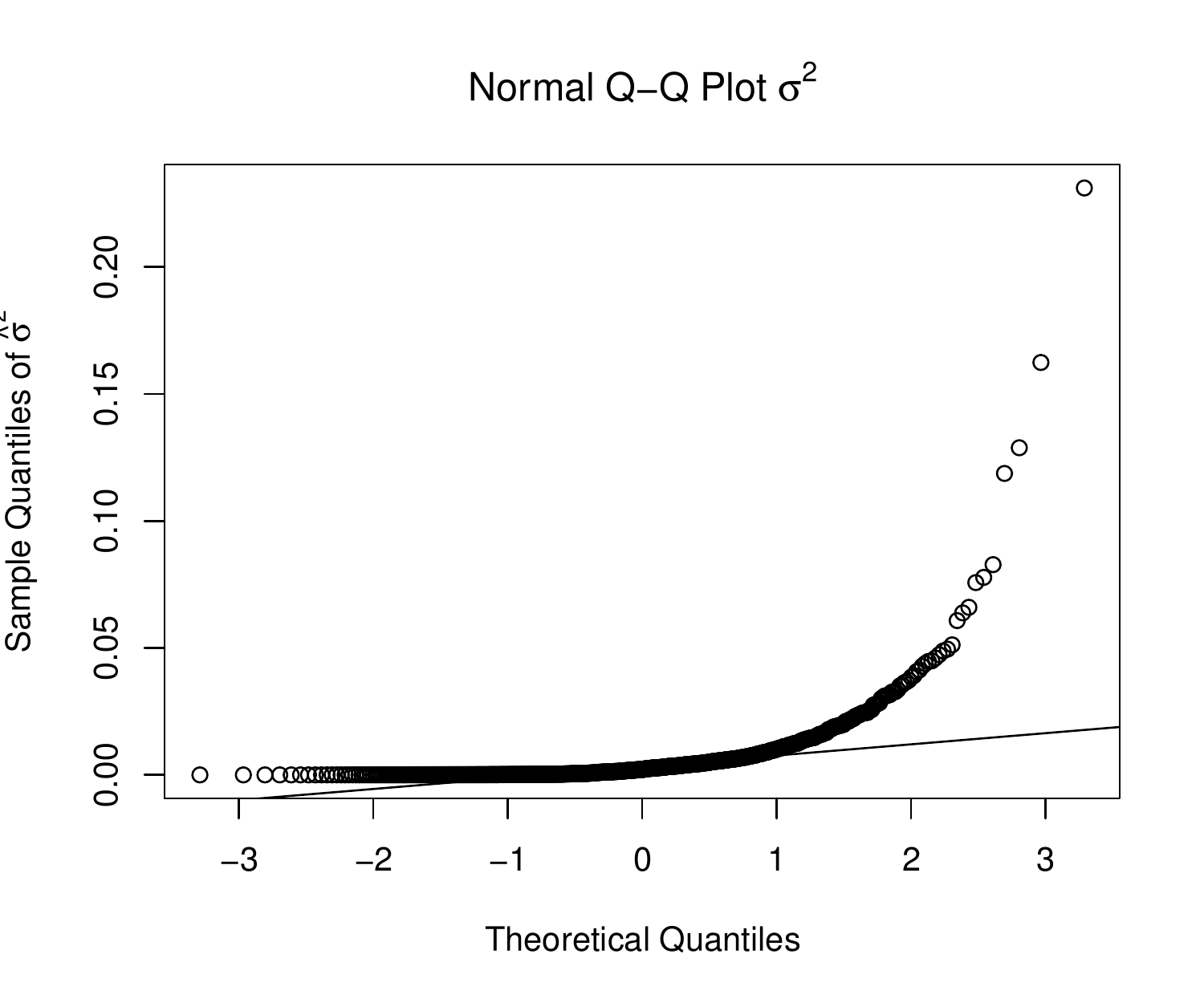}\\
\includegraphics[height=0.49\textwidth,angle=270]{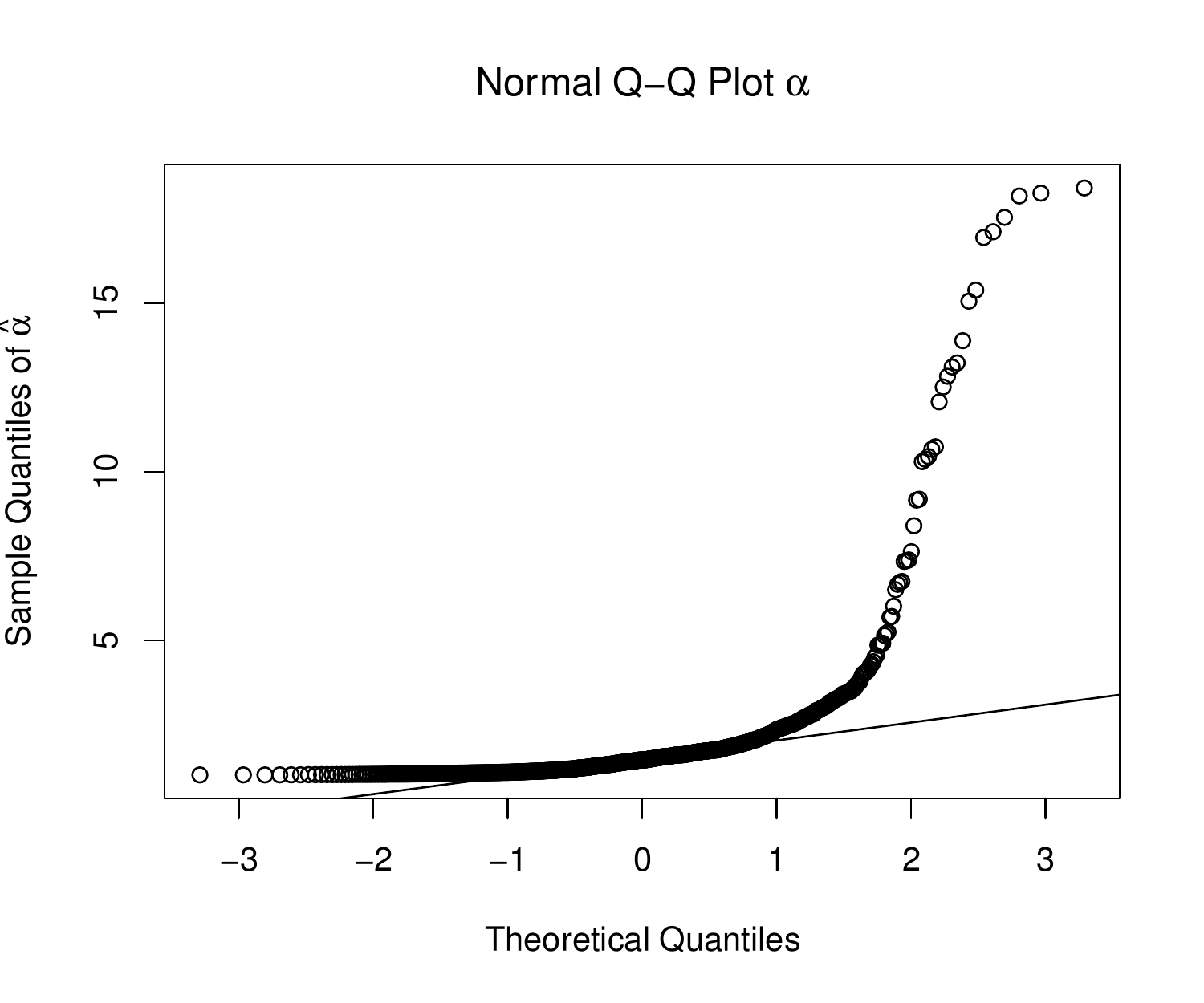}
\includegraphics[height=0.49\textwidth,angle=270]{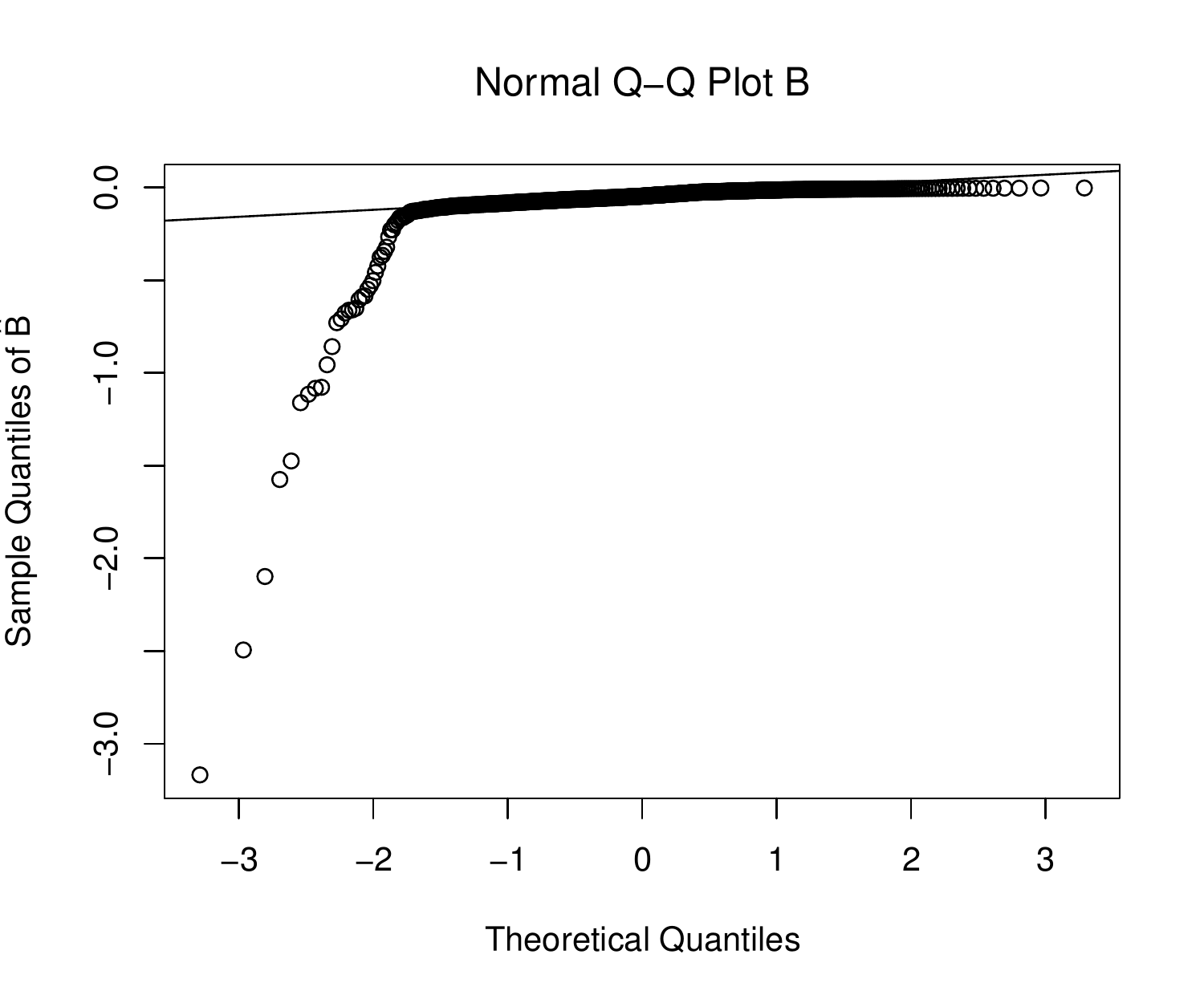}
\caption{Normal QQ-Plots of parameter estimates of 1000 paths of length 1000  of a supOU stochastic volatility model with short (upper set of plots) and with long memory (lower set of plots).}
\label{figSVsupOU1Kqq}
\end{figure}

If we compare these plots to the histograms (see Figure \ref{figsupOU1K}) and normal QQ-Plots (see Figure \ref{figsupOU1Kqq}) when we use only the last 1000 observations of every path for the estimators, we clearly see that then the quality of the estimation is considerably worse. This shows that reliable estimation of supOU processes needs a substantial amount of data. Clearly, the histograms are considerably more spread out (actually, the most extreme outliers -- less than 20 in all these cases -- are not shown in the histograms) and biases/asymmetries are much more distinct. In particular, it is noteworthy that the ``memory'' parameter $\alpha_\pi$ tends to be rather severely underestimated. Looking at the QQ-Plots, the estimators based on 1000 observations are close to normal for $\sigma^2$ and (to a lesser extent) for $\mu$, whereas for $\alpha_\pi$ and $B$ we have now distinctly non-normal tails.

\subsubsection{The supOU SV model}

Figures \ref{figSVsupOU10K} and \ref{figSVsupOU10Kqq} show the histograms and QQ-plots for the parameter estimators when using paths with 10000 log returns of the supOU SV model on a unit time grid. As is to be expected the estimation quality is worse than when using the observations of the supOU process (i.e. the latent volatility process itself), cf. Figures \ref{figsupOU10K} and \ref{figsupOU10Kqq}. This is clearly evident in the histograms for $\mu$ and $\sigma^2$ both in the short and long memory case. However, there appears to be much less asymmetry and bias, which may be surprising, but is also very nice, as in financial applications one typically has only log return data. The estimators for $\alpha_\pi$ and $B$ appear to be as good as when using the supOU observations 
in the long memory case and actually even  better in the short memory case, where $\alpha_\pi$ no longer tends to be underestimated. 

The normal QQ-plots  -- with the exception of $\alpha_\pi$ in the short memory case -- show now that the estimators are far away from a normal distribution. The histograms seem to suggest that there may well be a reasonable distributional limit result, but it should probably have heavy tails.

If we compare these plots to the histograms (see Figure \ref{figSVsupOU1K}) and normal QQ-Plots (see Figure \ref{figSVsupOU1Kqq}) when we use only the last 1000 observations of every path for the estimators, we clearly see again that then the quality of the estimation is considerably worse. Note that again the most extreme estimators are not depicted in the histogram. These were less than 30 data points, except for $\sigma^2$ in the short memory case (c. 70 points), $\mu$ in the long memory case (c. 70 points) and $\sigma^2$ in the long memory case (c. 150 points not shown).  So also for estimating supOU SV models it seems important to have a lot of data. Most interesting is a comparison with the estimations based on 1000 observations of the supOU process (Figures \ref{figsupOU1K}, \ref{figsupOU1Kqq}). Whereas $\mu$ and $\sigma^2$ are better estimated using the supOU/volatility data, the parameters $\alpha_\pi$ and $B$, which determine the decay of the acf, are clearly substantially better estimated using the simulated log return data both in the long and short memory case. Most notable is that  $\alpha_\pi$ is much less underestimated in the short memory case, although it still tends to be significantly underestimated.

\subsection{Empirical data illustration}

In an illustrative application to empirical data we estimate a supOU SV model under Assumption \ref{221} for the S\&P 500 using mean,  variance and lags $1$ to $5$  of the autocovariance function of the squared log returns. We use the daily time series from 03/29/2010 to 03/20/2013 which corresponds to 750 observations. The data source was Bloomberg Finance L.P. Before fitting the supOU SV model to the time series we demeaned it. 

The two step GMM estimation procedure gives $$(\hat{\mu},\hat{\sigma},\hat{\alpha}_{\pi},\hat{B}) = (6.1\times 10^{-6},1.4\times 10^{-9},6.8,-0.0086).$$
Note that the unit time scale is one day. The parameters can be ``annualized'' following Remark \ref{rem:annualize}.

\begin{figure}[h!]
\begin{center}
\includegraphics[height=0.49\textwidth,angle=270]{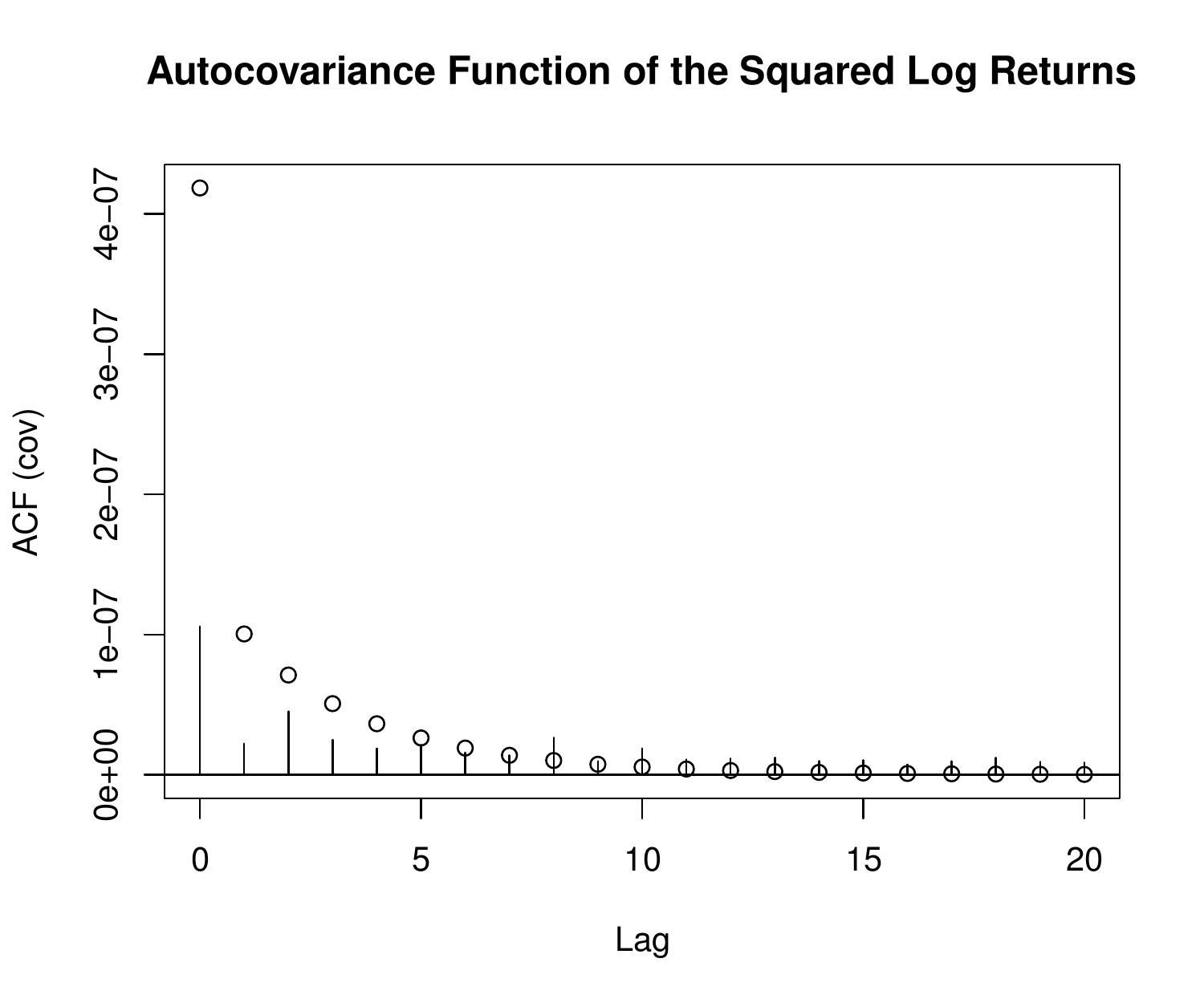}
\includegraphics[height=0.49\textwidth,angle=270]{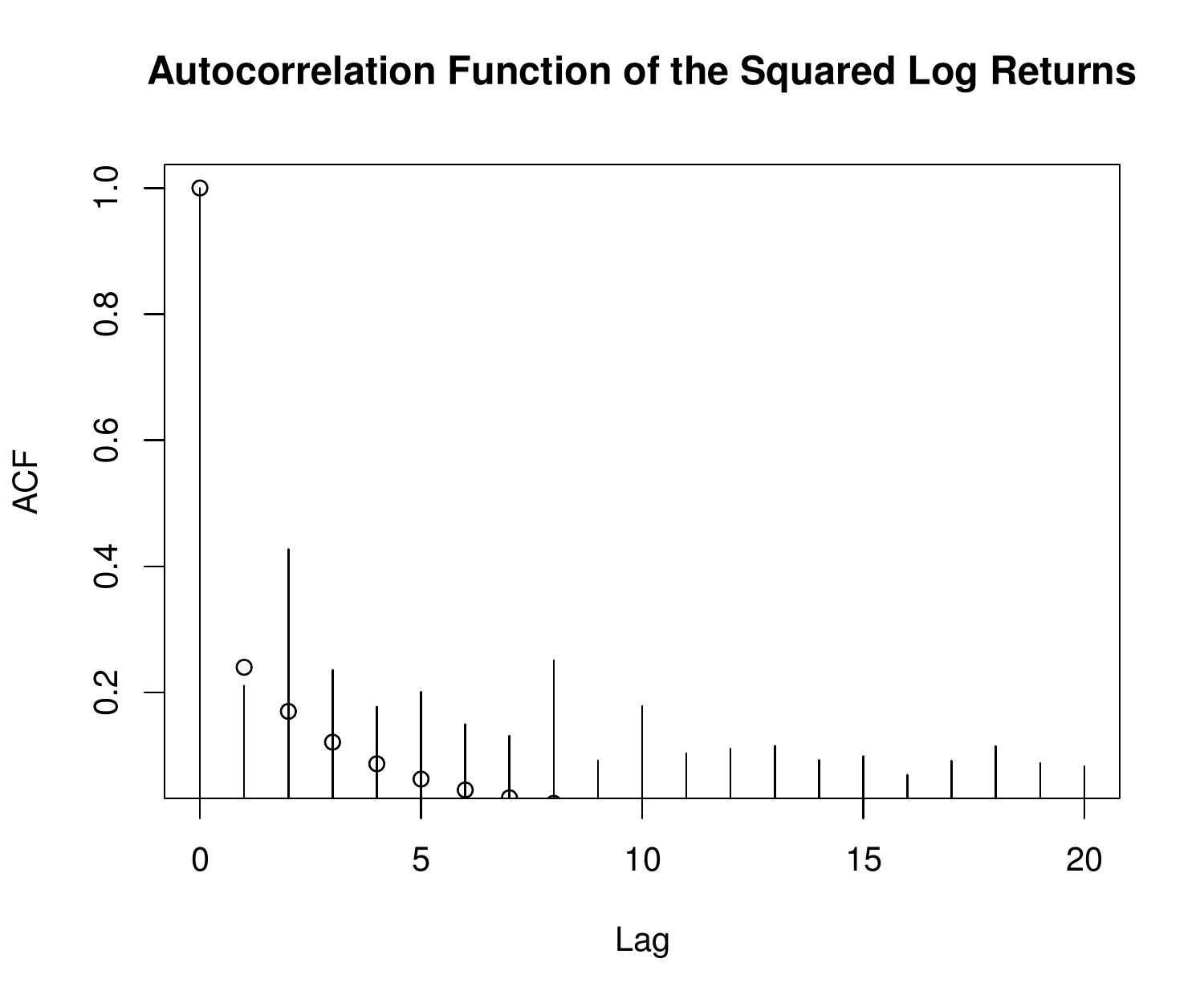}
\end{center}
\caption{The empirical autocovariance (left) and autocorrelation (right) function of the squared log returns of the S\&P 500 data set compared to the ones estimated in the \emph{first GMM step}. The circles depict the autocovariance/correlation function of the estimated supOU SV model and the bars depict the empirical one.}
\label{fig:datafirststep}
\end{figure}
\begin{figure}[h!]
\begin{center}
\includegraphics[height=0.49\textwidth,angle=270]{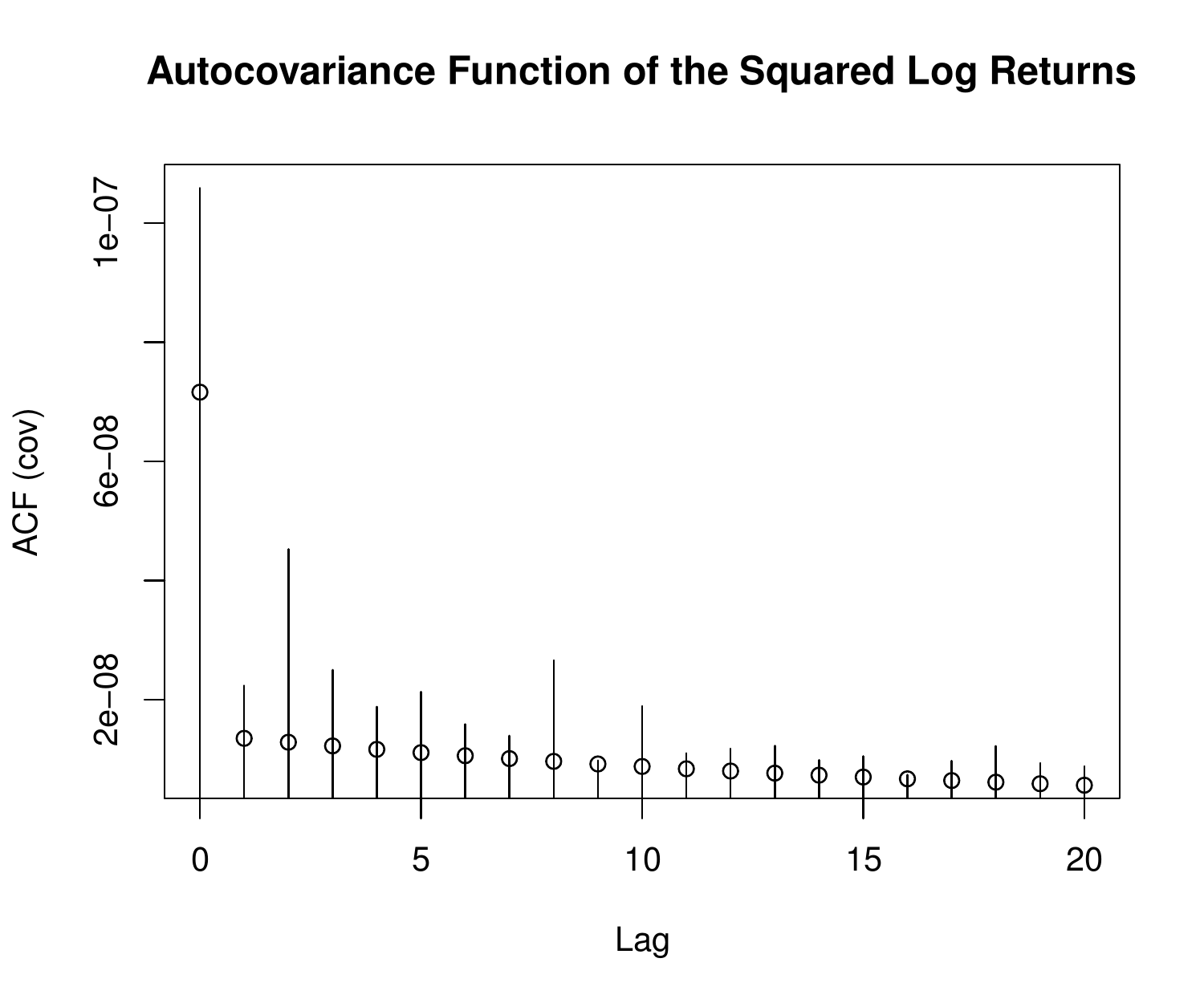}
\includegraphics[height=0.49\textwidth,angle=270]{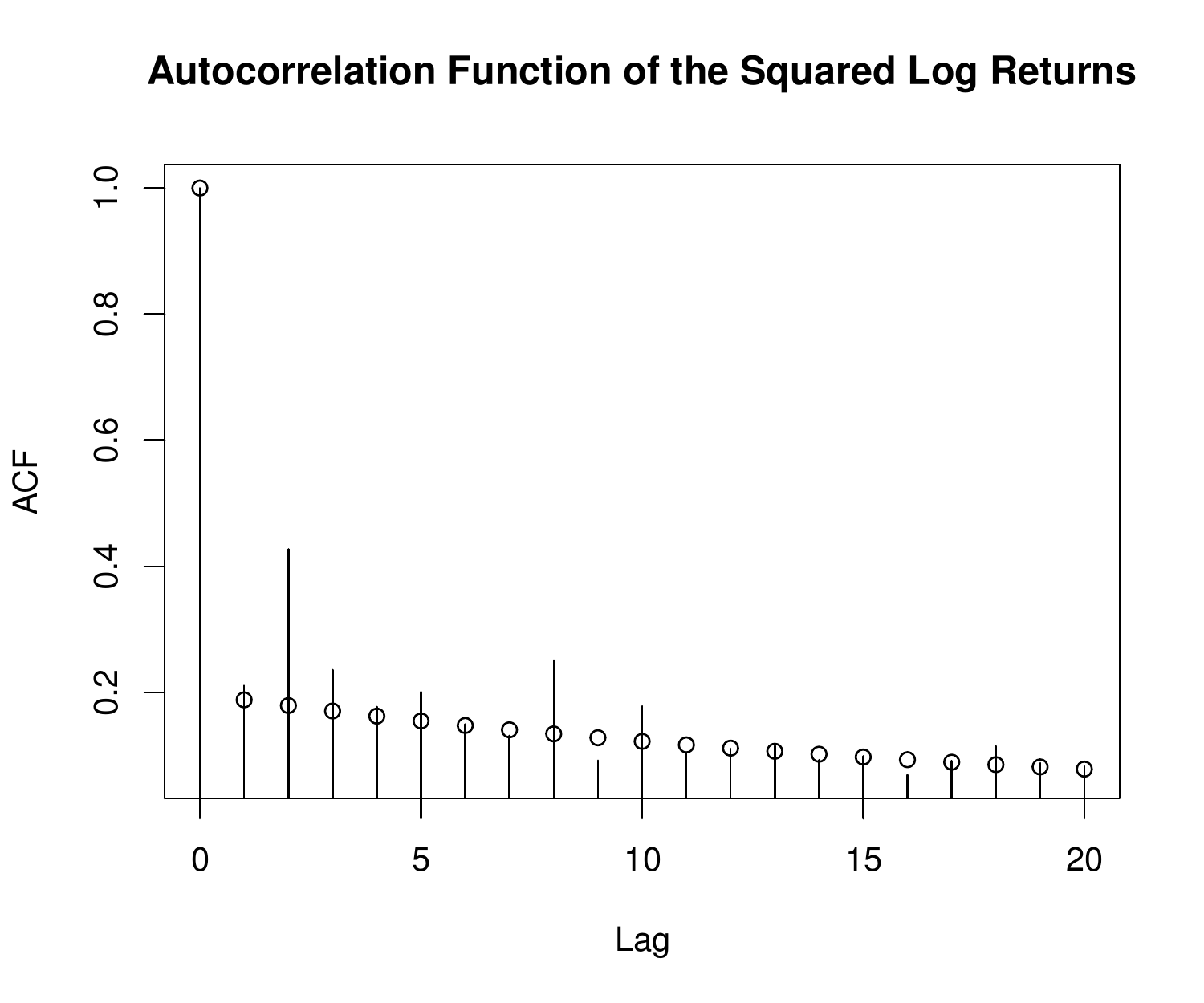}
\end{center}
\caption{The empirical autocovariance (left) and autocorrelation (right) function of the squared log returns of the S\&P 500 data set compared to the ones estimated by the \emph{two step GMM procedure}. The circles depict the  autocovariance/correlation function of the estimated supOU SV model and the bars depict the empirical one.}
\label{fig:datasecondstep}
\end{figure}

We plotted the empirical and the estimated autocovariance and  autocorrelation functions of the squared log returns in Figure \ref{fig:datafirststep} using the parameters obtained in the initial estimation step without weighting the moment conditions and in Figure \ref{fig:datasecondstep} using the parameters obtained from the two step GMM procedure. 
Comparing the figures shows that using an appropriate weighting matrix for the moment conditions is extremely important. Due to the sizes of magnitude the first step focuses on the mean of the squared log returns and thus is very inaccurate for the variance, as well as for the autocovariances and autocorrelations, whose ``decay parameter'' $\alpha$ is estimated to be $19$.

 In Figure \ref{fig:datasecondstep} we can see that after the second step the autocovariance function is well approximated. Regarding the autocorrelation function we have at the beginning  a small sinusoidal effect in the empirical autocorrelation function which the model by its nature cannot capture. However, in general the autocorrelation function is fitted extremely well, especially its decay. This is remarkable given that we used only the first five lags for the estimation.  The power decay at rate $h^{-5.8}$ of the model autocorrelation function clearly fits well with the rather slow decay of the empirical autocorrelation, but recall that the estimated model has no long memory effects as $\hat{\alpha}_{\pi} > 2$. Note that interestingly the calibration results of \cite{StelzerZavisin2014}, who report a calibrated $\alpha_\pi$ of $4.4$ for DAX option price data, are very similar in this respect. Of course, since we do not have any asymptotic distribution results for our estimators we cannot test whether the data suggests $\alpha_\pi>2$ (and thus short memory). Since we have only a ``semiparametric model'' (there are many very different L\'evy processes with the same $\mu$ and $\sigma^2$) and we do not know anything about the asymptotic dependence beyond autocorrelations, simulation based techniques like bootstrapping seem not to be feasible to attack this question either. If we look at our simulation study (with somewhat different parameters and a special choice for the underlying L\'evy process) we see that there $\alpha_\pi$ tended to be significantly underestimated, especially in the long memory case. This gives at least some support that our estimators for the S\&P 500 data  suggest that there is no long memory. In this context it seems also worth recalling that in our simulation study the ``acf decay parameters'' $\alpha_\pi$ and $B$ could be estimated comparably well from log return data.

\section{Conclusion}

This paper developed a GMM estimation method for supOU processes and supOU SV models which are of particular interest because of the possibility of long memory effects. In a simulation study the estimators behaved quite well and the results indicate that one has not only consistency (as shown in the paper) but also nice distributional limits, probably asymptotic normality when using supOU data.

How the estimators are actually distributed (e.g. asymptotic normality) is future work beyond the scope of the present paper. First one needs to show central limit theorems for supOU processes. The standard way via strong mixing appears very hard since supOU processes are non-Markovian and the usual approach to show strong mixing is to embed the model in a Markovian set-up and to show geometric ergodicity. In the future we intend to establish asymptotic distribution results of the estimators using alternative approaches like weak dependence. This may also allow to employ non-parametric techniques like bootstraping.

\bigskip
{\bf Acknowledgements}

The authors are very grateful to the editor and two anonymous referees for very helpful suggestions and to Christian Pigorsch for influential discussions.

\bibliographystyle{abbrvnat}

\bigskip
\noindent
\parbox[t]{.3\textwidth}{
Robert Stelzer\\
Institute of Mathematical Finance\\
Ulm University \\
Helmholtzstraße 18\\
89081 Ulm, Germany\\
robert.stelzer@uni-ulm.de\\
\url{http://www.uni-ulm.de/mawi/finmath.html} } \hfill
\parbox[t]{.3\textwidth}{
Thomas Tosstorff\\
\\
\\
\\
\\\\
 ttosstorff@yahoo.de
}
\hfill
\parbox[t]{.3\textwidth}{
Marc Wittlinger\\
Institute of Mathematical Finance\\
Ulm University \\
Helmholtzstraße 18\\
89081 Ulm, Germany\\
marc.wittlinger@posteo.de
}
\end{document}